%% file: main.tex
\documentclass{amsart}

\usepackage{pifont}
 \usepackage{booktabs} 
\newcommand{\Sym}{\mathrm{Sym}}
\newcommand{\GL}{\mathrm{GL}}
\newcommand{\Lie}{\mathrm{Lie}}
\usepackage{eucal}

\input{style_arxiv}
\usepackage{cleveref}
\usepackage{natbib}

\title{Optimal Transport on Lie Group Orbits}
\author[Taşkesen]{Bahar Taşkesen}\thanks{Booth School of Business, University of Chicago \url{bahar.taskesen@chicagobooth.edu}}
\date{}

\begin{document}
\maketitle

\begin{abstract} 
In its most general form, the optimal problem is an 
infinite-dimensional optimization problem, yet certain notable instances admit closed-form solutions. We identify the common source of this tractability as \textit{symmetry} and formalize it using Lie group theory. 
Fixing a Lie group action and a reference distribution, we
study optimal transport between measures in the same orbit.
For such pairs of measures, restricting the Monge problem to
transport maps induced by the group action yields the Lie group
orbit transport problem, an optimization over the stabilizer
of the reference law that provides an upper bound on the
Monge problem. Under suitable regularity conditions, a feasible
group-induced map satisfying a $c$-convex first-order
certificate attains this bound, solves the original Monge
problem, and induces an optimal Kantorovich plan.
For the quadratic cost, we establish a structural criterion based
on a global Cartan decomposition. 
Exploiting the nonuniqueness of
orbit representations, the criterion constructs an optimal map
through a path of group transformations with self-adjoint
derivatives. Once verified for an action and reference law, it
applies to every pair of orbit measures with finite second moments.
For affine-induced actions, Cartan structure on the linear image
suffices, and the means and covariance operators determine the
optimal map and value explicitly.
The framework recovers classical elliptical transport formulas
and yields closed-form solutions on congruence orbits generated
by absolutely continuous, spectrally invariant reference laws,
including Wishart, inverse-Wishart, and matrix beta type~II
families with fixed shape parameters. It also recovers the
common-copula optimal coupling for multivariate normal
sinh-arcsinh laws with unit tailweight parameters sharing a
Gaussian copula, and evaluates the optimal transport value
in closed form.
\end{abstract}

\input{plain}
\medskip
\medskip
\noindent \textbf{Acknowledgments.}
The author thanks Ya-Ping Hsieh, Wouter Jongeneel, and Teng Andrea Xu for helpful feedback. The author is also grateful to the corresponding editor, Dmitriy Morozov, and the anonymous associate editor for their careful handling of the manuscript, as well as to the referees for their constructive comments and suggestions, which greatly improved the paper.
This research was supported by the University of Chicago Booth School of Business, including through the Liew Family Junior Faculty Fellowship.

\bibliographystyle{plainnat}
\bibliography{main}

\newpage
\begin{center}
  \vspace*{2\baselineskip} 
  {\Huge \bfseries Appendix} 
  \vspace{2\baselineskip}   
\end{center}

\addtocontents{toc}{\protect\setcounter{tocdepth}{3}}

\noindent 
This document contains supplementary material for the paper
\emph{Optimal Transport on Lie Group Orbits
}.

\input{appendix}

\end{document}

%% file: style_arxiv.tex
\usepackage{algorithm}
\usepackage[noend]{algorithmic}
\usepackage{hyperref}

\usepackage{xcolor}

\definecolor{spring}{RGB}{174, 83, 139}

\hypersetup{colorlinks, breaklinks,
linkcolor={[RGB]{174,83,139}},
citecolor={[RGB]{97,62,173}},
urlcolor={[RGB]{115,135,195}}}

\usepackage{dsfont} 
\usepackage{bbm} 
\usepackage{footnote}

\usepackage{enumerate}
\usepackage{enumitem}    
\usepackage{comment}
\usepackage{enumerate}

\usepackage{multicol}

\usepackage{tcolorbox}
\usepackage{amsthm}
\usepackage{amsmath}
\usepackage{amssymb}
\usepackage{algorithm}
\newtheorem{theorem}{Theorem}[section]
\newtheorem{lemma}{Lemma}
\newtheorem{proposition}{Proposition}
\newtheorem{remark}{Remark}
\newtheorem{definition}{Definition}
\newtheorem{corollary}{Corollary}
\newtheorem{assumption}{Assumption}

\DeclareMathOperator{\Tr}{Tr}

\newcommand{\R}{\mathbb{R}}
\newcommand{\mc}{\mathcal}

\newcommand{\diff}{\mathrm{d}}
\newcommand{\opt}{^\star}

\newcommand{\EE}{\mathbb{E}}

\newcommand{\Stab}{\mathrm{Stab}}

\newcommand{\bs}{\boldsymbol}

\newcommand{\hf}{{1/2}}

\graphicspath{{figures/}}

\usepackage{xcolor}
\newcount\Comments
\newcommand{\kibitz}[2]{\ifnum\Comments=1{\textcolor{#1}{\textsf{\footnotesize #2}}}\fi}

\usepackage[margin=1.3in]{geometry}

\allowdisplaybreaks

%% file: plain.tex
\section{Introduction}
Originating in the seminal work of \citep{ref:monge1781memoire} and later given its analytic form by \citep{ref:kantorovich1942transfer}, the optimal transport problem seeks the minimal total cost of transforming one probability measure into another with respect to a prescribed transportation cost function defined on the underlying space.
When both probability measures are discrete and explicitly specified by enumerating their atoms and associated probabilities, the optimal transport problem reduces to a finite-dimensional linear program. Thus, classical polynomial-time algorithms ({e.g.}, interior-point methods \citep{ref:karmarkar1984new}) provide efficient solutions whose complexity scales polynomially with the input size. In sharp contrast, when discrete measures have implicitly defined supports, such as when distributions factorize across multiple dimensions, the number of atoms can grow exponentially with the dimension, resulting in a problem that remains polynomially describable yet is provably $\#$P-hard \citep{ref:tacskesen2023discrete}. Likewise, the optimal transport problem between a generic (possibly continuous) measure and a discrete measure is also known to be $\#$P-hard \citep{ref:tacskesen2023semi}.
When {both} probability measures are \textit{continuous}, the optimal transport problem manifests as an infinite-dimensional linear program. 
Following the previously observed complexity trend, one might naturally
conjecture that optimal transport between two continuous, non-atomic
measures, where mass is spread over a continuum rather than finitely many
atoms, would be equally formidable.
Curiously, however, certain instances of the optimal transport problem between continuous probability measures admit explicit closed-form solutions, rendering these problems uniquely tractable despite their inherently infinite-dimensional structure.

Early breakthroughs showed that infinite-dimensionality does not preclude analytic tracta\-bi\-lity. In one dimension, \citep{ref:dall1956sugli} showed that for any power of absolute distance of order at least one, the optimal transport problem admits an explicit solution in terms of the marginal quantile functions.
For multivariate Gaussian distributions, \citep{DowsonLandau1982} and \citep{OlkinPukelsheim1982} derived an explicit solution for the optimal transport problem induced by the quadratic cost function in terms of the means and covariance matrices. \citep{Gelbrich1990} then generalized this to elliptically contoured distributions sharing a common characteristic generator and having finite second moments, yielding a closed-form formula for the optimal transport problem and the associated optimal solution.
For a comprehensive survey of closed-form solutions of optimal transport, see \citep[\S 3]{ref:rachev1998mass}.

Taken together, these examples appear disparate: it is not evident
a priori why an inherently infinite-dimensional problem should admit
explicit closed-form solutions in these particular cases. We show that
their common structural mechanism is symmetry, and we use Lie group
actions to formalize and exploit this structure across families of
distributions.

Specifically, we fix a Lie group action on the outcome space and a
reference distribution, and study optimal transport between distributions
lying on the same Lie group orbit. This orbit perspective induces a
natural equivalence among distributions, allowing familiar families to
be recognized as members of a single orbit. Restricting attention to
transport maps induced by the group yields a structured upper bound on
the Monge value, formulated as an optimization over the stabilizer
subgroup of the reference law. Under suitable regularity conditions, we
prove that this upper bound is exact for both the Monge and Kantorovich
problems whenever an orbit-induced candidate map satisfies a
\(c\)-convex first-order certificate. In that case, the optimal transport
is realized by a suitable group element, and the Monge, Kantorovich, and
orbit transport values coincide.

{\color{black}
For the quadratic cost, we give a Cartan-based structural criterion that
exploits the nonuniqueness of orbit representations to construct an
orbit-induced map. By connecting this map to the identity through group
transformations with self-adjoint, invertible derivatives, the criterion
ensures that it satisfies the \(c\)-convex first-order certificate. Once verified for a group action and reference law, the criterion certifies
optimality for every pair of distributions in the corresponding orbit
with finite second moments.

When the action is induced by an affine representation on the ambient
Hilbert space, the structural criterion can be imposed on the linear
image of the acting group. In this setting, the resulting optimal map is
explicitly determined by the means and covariance operators of the
source and target distributions. Its transport value attains the lower
bound expressed in terms of these first two moments in
\citep[Theorem~2.1]{Gelbrich1990}. For elliptical distributions sharing a
common characteristic generator, this recovers the classical
closed-form transport formula and the corresponding equality case of
that bound.

We then apply the affine criterion to congruence orbits generated by
absolutely continuous, spectrally invariant reference laws on the cone
of positive-definite matrices. This yields closed-form quadratic optimal
transport maps and values between any two distributions in each such
orbit. In particular, we obtain explicit formulas for Wishart,
inverse-Wishart, and matrix beta type~II distributions with fixed shape
parameters and varying positive-definite scale matrices. Viewed as
probability distributions on $\Sym(d)$, these matrix-valued families are
not elliptically contoured and therefore furnish a natural
non-elliptical equality class for the moment-based lower bound in
\citep[Theorem~2.1]{Gelbrich1990}. To the best of our knowledge, the
corresponding closed-form transport formulas for these three families
are new. Finally, we apply the general Cartan criterion to an orbit of
multivariate normal sinh-arcsinh distributions with unit tailweight
parameters and a common Gaussian copula, recovering from the orbit
structure the common-copula optimal coupling of
\citep[Proposition~1.1]{AlfonsiJourdain2014} and evaluating the
corresponding quadratic transport value in closed form.
}

{\color{black}
Our main results are organized in three stages.
First, \Cref{thm:closed-form} establishes the general orbit-reduction
principle for arbitrary costs and gives a \(c\)-convex certificate under
which the Monge, Kantorovich, and Lie group orbit transport values
coincide. Second, for the quadratic cost,
\Cref{thm:nonlinear-cartan-exactness} provides a structural Cartan
criterion that produces an orbit-induced optimal map. Along the
resulting continuous path from the identity, self-adjointness and
invertibility force the derivatives to remain positive definite, so the
terminal map is the gradient of a strictly convex potential. Third,
\Cref{prop:affine-cartan-criterion} specializes this criterion to actions
induced by affine representations on the ambient Hilbert space. In this
setting, the means and covariance operators of the source and target
measures explicitly determine the optimal map, and the corresponding
transport value coincides with the moment-based lower bound in
\citep[Theorem~2.1]{Gelbrich1990}. \Cref{sec:examples} applies these
results to elliptical distributions, spectral-congruence orbits of
matrix-valued distributions on the positive-definite cone, and
multivariate normal sinh-arcsinh distributions.
}
\medskip
\paragraph{\textbf{Notation}}
We write
$\mathbb N_{>0}=\{1,2,\ldots\}$,
$\mathbb R_{>0}=(0,\infty)$, and
$\bar{\mathbb R}=\mathbb R\cup\{+\infty\}$
for the positive integers, the positive real line, and the upper
extended real line, respectively. For $n\in\mathbb N_{>0}$,
$\mathbb R^n$ is equipped with its Borel $\sigma$-algebra
$\mathcal B(\mathbb R^n)$, Lebesgue measure $\mathcal L^n$,
Euclidean inner product $\langle\cdot,\cdot\rangle$, and induced norm
$\|\cdot\|$. Unless stated otherwise, $\mathcal X$ denotes a nonempty
open subset of $\mathbb R^n$, equipped with the Borel $\sigma$-algebra
$\mathcal B(\mathcal X)$ inherited from $\mathbb R^n$. For any set
$E$, $\mathrm{id}_E:E\to E$ denotes the identity map. All measures are understood to be Borel measures on the indicated
spaces. For two measures $\mu$ and $\nu$ on the same measurable space,
$\mu\ll\nu$ denotes absolute continuity, and
$\mathrm d\mu/\mathrm d\nu$ denotes the corresponding
Radon-Nikodym derivative. If $\nu$ is a measure on $\mathbb R^n$,
then $\nu|_{\mathcal X}$ denotes its restriction to
$(\mathcal X,\mathcal B(\mathcal X))$. We write $\mathcal P(\mc E)$ for
the set of Borel probability measures on a Borel space $\mc E$. For a
Borel-measurable map $T: \mc E\to \mc F$ and $\mu\in\mathcal P(\mc E)$, the
push-forward $T_\#\mu\in\mathcal P(\mc F)$ is defined by
$T_\#\mu(\mc A)=\mu(T^{-1}(\mc A))$ for every $\mc A\in\mathcal B(\mc F)$.
The Dirac probability measure at $\bs x\in \mc E$ is denoted by $\delta_{\bs x}$.
For a differentiable map $F$ between open subsets of finite-dimensional
Euclidean spaces, $D_{\bs x}F$ denotes its derivative at $\bs x$.
When $F$ depends on additional variables, the subscript $\bs x$
indicates differentiation with respect to the state variable. For a
differentiable real-valued function $f$, $\nabla_{\bs x}f(\bs x)$
denotes its Euclidean gradient. For a smooth map $F$ between smooth
manifolds, $\mathrm dF_p$ denotes its differential at $p$.
For a Borel subset $\mc X$ of a finite-dimensional real Hilbert space
$\mc H$, we write
\[
\mc P_2(\mc X)
=
\left\{
\mu\in\mc P(\mc X):
\int_{\mc X}\|\bs x\|_{\mc H}^2\diff\mu(\bs x)<\infty
\right\}.
\]
We write $\GL(\mc H)$ for the group of invertible linear operators on
$\mc H$, $A^*$ for the adjoint of $A$, and $\Tr_{\mc H}(A)$ for its
operator trace. If $\bs X\sim\mu\in\mc P_2(\mc H)$, then
$\bs m_\mu=\EE[\bs X]$ denotes its mean and
$\mathsf C_\mu:\mc H\to\mc H$ its covariance operator, defined by $
\langle\bs u,\mathsf C_\mu\bs v\rangle_{\mc H}
=
\EE\!\left[
\langle\bs u,\bs X-\bs m_\mu\rangle_{\mc H}
\langle\bs v,\bs X-\bs m_\mu\rangle_{\mc H}
\right]$.
We denote the Lie-group exponential by
$\exp:\Lie(G)\to G$; for a matrix Lie group, this agrees with the
usual matrix exponential.

\section{Preliminaries} 
This section introduces mass transportation problems and develops the Lie group background needed for our main results.

\subsection{Mass transportation problems} 
In his memoir, \citep{ref:monge1781memoire} formulated the problem of transporting one distribution of mass into another at minimal cost. Throughout the paper, we assume that the transportation cost
$c:\mathcal X\times\mathcal X\to\R$ is Borel measurable and
bounded below. Given two measures
$\mu_0,\mu_1\in\mathcal P(\mathcal X)$, the \textit{Monge problem}
induced by $c$ is defined as
\begin{equation}\label{eq:monge-problem}
      \mathds M_{c}(\mu_0,\mu_1)
      =
      \inf_{T\in\mathcal T(\mu_0,\mu_1)}
      \int_{\mathcal X} c\left(\bs x,T(\bs x)\right) \diff \mu_0(\bs x),
      \tag{MP}
\end{equation}
where $
      \mathcal T(\mu_0,\mu_1)
      =
      \left\{
           T:\mathcal X\to\mathcal X
           \text{ Borel-measurable }
           \mid
           T_{\#}\mu_0=\mu_1
      \right\}$
is the set of {admissible transport maps}.  
If a map $T\opt\in\mathcal T(\mu_0,\mu_1)$ attains the infimum
in~\eqref{eq:monge-problem}, then it is called an \textit{optimal Monge map} between $\mu_0$ and~$\mu_1$.

The admissible set $\mathcal T(\mu_0,\mu_1)$ can be empty. For instance, if $\mu_0$ has an atom while $\mu_1$ is atomless, no measurable (deterministic) map can push $\mu_0$ to $\mu_1$ because a map cannot split mass. Even when $\mathcal T(\mu_0,\mu_1)$ is nonempty, the constraint $T_{\#}\mu_0=\mu_1$ of \eqref{eq:monge-problem} is {non-convex}. To see this, suppose that $\mu_0 = \mu_1 =  \mathcal U([0,1])$. Now, let $T_1( x) =  x$ and $T_2( x) = 1- x$. Then, a simple calculation shows that ${T_1}, T_2 \in \mathcal T(\mu_0, \mu_1)$. Their midpoint $\bar T = 0.5(T_1 + T_2) $ is the constant map sending every $x$ to $\hf$; that is, it pushes $\mu_0$ to a Dirac measure $\delta_{\hf}$, $\bar T_\# \mu_0 = \delta_{\hf} \neq \mu_1$ implying that $\bar T\notin \mathcal T(\mu_0, \mu_1)$. 
Thus the Monge problem~\eqref{eq:monge-problem} is an infinite-dimensional optimization over a generally nonconvex feasible set.

In 1942, \citep{ref:kantorovich1942transfer} introduced a convex relaxation of \eqref{eq:monge-problem} by replacing maps with probability couplings. Given $\mu_0,\mu_1\in\mathcal P(\mathcal X)$, the \emph{Kantorovich problem} induced by cost function $c$ is defined as
\begin{equation}
\label{eq:kantorovich-problem}
   \mathds K_c(\mu_0,\mu_1)=\inf_{\Gamma\in \Pi(\mu_0,\mu_1)}
        \int_{\mc X \times\mc X} c(\bs x,\bs y) \diff \Gamma(\bs x,\bs y),
        \tag{KTP}
\end{equation}
where $\Pi(\mu_0, \mu_1)$ denotes the set of all probability measures on $\mc X\times \mc X$ with first marginal $\mu_0$ and second marginal $\mu_1$. If $\Gamma\opt \in \Pi(\mu_0, \mu_1)$ solves \eqref{eq:kantorovich-problem}, then it is called an \textit{optimal transportation plan} between $\mu_0$ and $\mu_1$.
The optimization in~\eqref{eq:kantorovich-problem} is commonly called the optimal transport problem; to avoid ambiguity we will refer to \eqref{eq:kantorovich-problem} as the \emph{Kantorovich problem} and to \eqref{eq:monge-problem} as the \emph{Monge problem}.

The deterministic nature of the Monge problem is appealing as it yields an explicit transport map that relocates mass without splitting it, in contrast to the probabilistic transportation plans of the Kantorovich formulation in the form of couplings. This brings interpretability and aligns with settings where splitting is not physically meaningful or desirable (e.g., moving a pile of soil or routing indivisible items).
There is a tight connection between the Monge problem
\eqref{eq:monge-problem} and the Kantorovich problem
\eqref{eq:kantorovich-problem}. {\color{black}Under suitable regularity assumptions,
optimal Kantorovich plans concentrate on graphs of transport maps, and
these maps solve the corresponding Monge problem.}  Existence of Monge solutions for the quadratic cost on Euclidean space was first developed in \citep{ref:brenier1987decomposition, ref:ruschendorf1990characterization}; for non-quadratic costs, existence was investigated in \citep{ref:ruschendorf1991frechet, ref:smith1992hoeffding, McCann1995}. One of the most general results linking Monge solutions to gradients of $c$-convex functions is due to \citep{ref:villani2008optimal}. For ease of reference, we now adapt \citep[Theorem~10.28]{ref:villani2008optimal} to our notation and restate it here. Before doing so, we record the regularity hypotheses on cost function $c$ and recall the definition of $c$-convexity required by that theorem.

\begin{assumption}
\label{ass:c-regularity}
The cost function $c:\mc X\times\mc X\to\R$ {is} continuous and bounded below. Additionally, for every $\bs y\in\mc X$, the map $\bs x\mapsto c(\bs x,\bs y)$ belongs to $\mc C^1(\mc X)$, and 
{\color{black}for every $\bs x\in\mc X$, the map $
\mc X \ni \bs y \mapsto \nabla_{\bs x} c(\bs x,\bs y)\in\R^n$
}is injective.
\comment{for every $\bs x\in\mc X$ the map $\bs y\mapsto \nabla_{\bs x}c(\bs x,\bs y)$ is injective. {\color{red}: What is the domain of $\bs y \mapsto \nabla_{\bs x} c(\bs x, \bs y)$ here?}}
\end{assumption}
Assumption~\ref{ass:c-regularity} contains the twist condition (see \citep[\S10]{ref:villani2008optimal}): for each fixed $\bs x$, no two distinct targets $\bs y\neq\bs y'$ yield the same $\bs x$-gradient of the cost. Geometrically, different destinations exert different first-order ``forces'' at $\bs x$.

\begin{definition}[$c$-convexity and $c$-subdifferential]
Suppose $c:\mc X\times\mc X\to\R$.
A function $\psi:\mc X\to\bar\R$ is said to be \emph{$c$-convex} if it
is not identically $+\infty$ and there exists a function
$\phi:\mc X\to \R\cup\{\pm\infty\}$ such that $  \psi(\bs x)
  =
  \sup_{\bs y\in\mc X}\{\phi(\bs y)-c(\bs x,\bs y)\}$, $\forall \bs x\in\mc X$.
{\color{black}For a $c$-convex function $\psi:\mc X\to\bar\R$, its
{$c$-subdifferential} at a point $\bs x\in\mc X$
with $\psi(\bs x)<\infty$ is defined by $
  \partial^c\psi(\bs x)=\{
  \bs y\in\mc X:
  \psi(\bs z)\ge
  \psi(\bs x)+c(\bs x,\bs y)-c(\bs z,\bs y)~
  \forall \bs z\in\mc X
  \}$. If $\psi(\bs x)=+\infty$, we set
$\partial^c\psi(\bs x)=\varnothing$.
We say that $\psi$ is $c$-subdifferentiable at $\bs x$ if
$\partial^c\psi(\bs x)\neq\varnothing$, and we call $  \mathrm{dom}(\partial^c\psi)=
  \{\bs x\in\mc X:\partial^c\psi(\bs x)\neq\varnothing\}$
the \emph{domain of $c$-subdifferentiability}.}
\end{definition}
{\color{black}The notion of $c$-convexity is a cost-dependent analogue of ordinary
convexity. In classical convex analysis, convex functions arise as
pointwise suprema of affine functions. Here, the affine family is
replaced by the cost-generated family $
  \bs x\mapsto \phi(\bs y)-c(\bs x,\bs y)$,
  $\bs y\in\mc X$,
which leads to the definition above.
In the same spirit, the notion of the $c$-subdifferential is the
corresponding cost-dependent analogue of the ordinary subdifferential.
In classical convex analysis, a vector $\bs p\in\R^n$ belongs to the
subdifferential of a convex function $\psi$ at $\bs x$ if the affine
function $
  \bs z\mapsto \psi(\bs x)+\langle \bs p,\bs z-\bs x\rangle$
agrees with $\psi$ at $\bs x$ and lies below $\psi$ for all
$\bs z\in\mc X$. Here, affine supporting functions are replaced by the
cost-adjusted family $
  \bs z\mapsto \psi(\bs x)+c(\bs x,\bs y)-c(\bs z,\bs y)$,~
  $\bs y\in\mc X$,
which leads to the definition above.

  \begin{assumption}
\label{ass:c-mu-0}
The cost function $c:\mc X\times\mc X\to\R$ and the measure
$\mu_0\in\mc P(\mc X)$ are such that, for every $c$-convex function
$\psi:\mc X\to\bar\R$, the function $\psi$ is differentiable
$\mu_0$-almost everywhere on $\mathrm{dom}(\partial^c\psi)$.
\end{assumption}
{\color{black}Assumption~\ref{ass:c-mu-0} allows the geometric condition
$\bs y\in\partial^c\psi(\bs x)$ to be converted into a first-order
identity. Indeed, let $\psi$ be $c$-convex, let
$\bs x\in\mathrm{dom}(\partial^c\psi)$ be a point at which $\psi$ is
differentiable, and choose
$\bs y\in\partial^c\psi(\bs x)$. By the definition of the
$c$-subdifferential, $
  \psi(\bs z)
  \ge
  \psi(\bs x)+c(\bs x,\bs y)-c(\bs z,\bs y)$ for all $ \bs z\in\mc X$.
Equivalently, $
  \psi(\bs z)+c(\bs z,\bs y)
  \ge
  \psi(\bs x)+c(\bs x,\bs y)$ for all $ \bs z\in\mc X$.
Hence the function $
  \bs z\mapsto \psi(\bs z)+c(\bs z,\bs y)$
attains a minimum at $\bs z=\bs x$. Since $\mc X$ is open and
$\bs z\mapsto c(\bs z,\bs y)$ is $\mc C^1$, the first-order optimality
condition gives $
  \nabla_{\bs x}\psi(\bs x)
  +
  \nabla_{\bs x}c(\bs x,\bs y)
  =
  0$.
Together with the twist condition in Assumption~\ref{ass:c-regularity}, this
identity identifies $\bs y$ uniquely.}}

{\color{black}\begin{assumption}
\label{ass:H-infty}
The cost \(c:\mc X\times\mc X\to \R\) satisfies condition
\((H_\infty)\) in the sense of \citep[Chapter~10]{ref:villani2008optimal}. 
 More explicitly,
for a set \(\mc S\subseteq\mc X\) and a point \(\bs x\in\overline {\mc S}\), write $
  \mc T(\mc S,\bs x)
  =
  \{
  \lim_{k\to\infty}\frac{\bs x_k-\bs x}{t_k}
  :
  \bs x_k\in \mc S, \bs x_k\to\bs x, t_k>0,\ t_k\to0
  \}$
for the tangent cone to \(\mc S\) at \(\bs x\). Then:
\begin{enumerate}[label=\textnormal{(\roman*)}]
  \item For every measurable set
  \(\mc S\subseteq\mc X\) {\color{black}and every $\bs x \in \bar{\mc S} \cap \mc X$} whose tangent cone \(\mc T(\mc S,\bs x)\) is not contained
  in a half-space, there exist points
  \(\bs z_1,\ldots,\bs z_k\in\mc S\) and a ball \(B\subseteq\mc X\)
  containing \(\bs x\) such that, for all \(\bs y\) outside a compact
  subset of \(\mc X\), $
    \inf_{\bs w\in B} c(\bs w,\bs y)
    \ge
    \inf_{1\le j\le k} c(\bs z_j,\bs y)$.

  \item For every \(\bs x\in\mc X\) and every neighborhood
  \(\mc U\subseteq\mc X\) of \(\bs x\), there exists a ball
  \(B\subseteq\mc X\) containing \(\bs x\) such that
   $
    \lim_{\bs y\to\infty}
    \sup_{\bs w\in B}
    \inf_{\bs z\in \mc U}
    \{
      c(\bs z,\bs y)-c(\bs w,\bs y)\}
    =
    -\infty.$
\end{enumerate}
Here \(\bs y\to\infty\) means that \(\bs y\) eventually leaves every
compact subset of \(\mc X\).
\end{assumption}
Assumption~\ref{ass:H-infty} is a no-escape-at-infinity condition on
the cost. It controls what happens when the second argument
\(\bs y\) leaves every compact subset of \(\mc X\). The first part says
that, near any point \(\bs x\) and along any set \(\mc S\) that has enough
directions around \(\bs x\), the cost of sending nearby points
\(\bs w\in B\) to a far-away target \(\bs y\) is bounded from below by
the cost of sending finitely many comparison points
\(\bs z_1,\ldots,\bs z_k\in \mc S\) to the same target. Thus far-away
targets cannot distinguish the point \(\bs x\) from its surrounding
directions in an uncontrolled way. The second part is a stronger separation condition. It says that, if
\(\bs z\) is allowed to range in any prescribed neighborhood \(\mc U\) of
\(\bs x\), then for targets \(\bs y\) going to infinity the cost
difference $
  c(\bs z,\bs y)-c(\bs w,\bs y)$
can be made uniformly very negative, with \(\bs w\) ranging over a
small ball \(B\) around \(\bs x\). In other words, when the target moves
far away, points near \(\bs x\) cannot all behave as equally good
first-order competitors. The assumption therefore rules out pathological
behavior caused by target points escaping to infinity.}

{\color{black}For ease of reference, we now recall the relevant consequences of
\citep[Theorem~10.28]{ref:villani2008optimal} in the notation of this paper.}

\begin{theorem}
{\color{black}Let \(\mu_0,\mu_1\in\mc P(\mc X)\). Suppose that
\(c:\mc X\times\mc X\to\R\) satisfies
Assumption~\ref{ass:c-regularity}, that the pair
\((c,\mu_0)\) satisfies Assumption~\ref{ass:c-mu-0}, and that $\mathds K_c(\mu_0, \mu_1) < \infty$.} Then:
\begin{enumerate}[label=\textnormal{(\roman*)}]
 \item\eqref{eq:kantorovich-problem} admits a unique (in law) optimal transportation plan $\Gamma\opt \in \Pi(\mu_0, \mu_1)$.
 \item There exists a unique {\color{black}optimal} Monge map $T\opt : \mc X \to \mc X$ {\color{black}, up to $\mu_0$-almost everywhere equality} solving \eqref{eq:monge-problem}. 
 \item 
 {\color{black}There exists a \(c\)-convex function \(\varphi:\mc X\to\bar\R\)
  such that }
  \begin{equation}
  \tag{\ensuremath\heartsuit}
    \nabla_{\bs x}\varphi(\bs x)
    +
    \nabla_{\bs x}c(\bs x,T\opt(\bs x))
    =
    0 \quad \mu_0\text{-almost surely}.
  \label{eq:T-optimal-c-convex-equation}
  \end{equation}
 \item  $\Gamma\opt$ is concentrated on the graph of $T{\opt}$, that is, $\Gamma\opt = (\mathrm{id}_{\mc X}, T\opt)_{\#}\mu_0$.
 \item {\color{black} If, in addition, \(c\) satisfies Assumption~\ref{ass:H-infty}, then
\eqref{eq:T-optimal-c-convex-equation} characterizes the optimal
coupling in the following sense: if \(\Gamma\in\Pi(\mu_0,\mu_1)\) is
concentrated on pairs \((\bs x,\bs y)\) satisfying $
  \nabla_{\bs x}\varphi(\bs x)
  +
  \nabla_{\bs x}c(\bs x,\bs y)
  =
  0$
for some \(c\)-convex function \(\varphi\), then \(\Gamma\) is the
unique optimal Kantorovich plan. In particular, if a feasible map
\(T:\mc X\to\mc X\) satisfies \eqref{eq:T-optimal-c-convex-equation}
for some \(c\)-convex function \(\varphi\), then \(T\) is the unique
optimal Monge map up to \(\mu_0\)-almost everywhere equality.}
\end{enumerate}
\label{thm:monge-kantorovich}
\end{theorem}

Under Assumptions~\ref{ass:c-regularity},
\ref{ass:c-mu-0}, and~\ref{ass:H-infty}, \Cref{thm:monge-kantorovich}(v) delineates conditions under which the Monge map is characterized via the gradient condition \eqref{eq:T-optimal-c-convex-equation}.
\subsection{Lie group theory} 
Originating with \citep{ref:lie1891vorlesungen}, Lie groups formalize {continuous} symmetry, that is, they are groups that are simultaneously smooth manifolds, with multiplication and inversion smooth.
A canonical example is the circle (planar rotations), where composition and inversion vary smoothly with the angle.
In what follows, we will review the minimal Lie-theoretic background used throughout; for comprehensive treatments see \citep{ref:boumal2023introduction, ref:duistermaat2012lie,ref:hall2013lie}.

\begin{definition}[Lie group]
    Let $ G$ be both a {smooth manifold} and a {group}. If the product map ${\rm{prod}} :  G\times  G \to  G : (g,h) \mapsto {\rm{prod}}(g,h) = g h$
    and the inverse map ${\rm{inv}} :  G \to  G: g \mapsto \rm{inv}(g) =g^{-1}$
    are smooth, then $G$ is a {Lie group}. Smoothness of $\rm{prod}$ is understood with respect to the product manifold structure on $ G \times  G$. 
    \label{def:lie-group}
\end{definition}
We now illustrate \Cref{def:lie-group} with a standard example: the general linear group $\mathrm{GL}(d)$.
Note that determinant is a polynomial map, so 
$\mathrm{GL}(d)=\det^{-1}(\R\setminus\{0\})$ is an open subset of 
$\R^{d^{2}}$. Hence, $\mathrm{GL}(d)$ is a smooth manifold of dimension $d^2$. 
The group operation is matrix multiplication and inversion. The multiplication map $
\mathrm{prod}:\mathrm{GL}(d)\times \mathrm{GL}(d)\to \mathrm{GL}(d),$ with $ (\bs A,\bs B)\mapsto \bs A\bs B$
is smooth because each entry of $\bs A\bs B$ is a polynomial in the entries of $\bs A$ and $\bs B$.
The inverse map $
\mathrm{inv}:\mathrm{GL}(d)\to \mathrm{GL}(d),$ with $ \bs A\mapsto \bs A^{-1}$
is smooth on $\mathrm{GL}(d)$ since $\bs A^{-1}=\mathrm{adj}(\bs A)/\det(\bs A)$, where $\mathrm{adj}(\bs A)$
is polynomial in the entries
of $\bs A$ and $\det(\bs A)\neq 0$ on $\mathrm{GL}(d)$.
Thus both $\mathrm{prod}$ and $\mathrm{inv}$ are smooth, so $\mathrm{GL}(d)$ is a Lie group.
Many other familiar matrix groups are Lie groups including $
\mathcal O(d) $ and $
\mathrm{SL}(d)=\{\bs A\in\R^{d\times d}\mid \det(\bs A)=1\}$; see \citep[Chapter 1]{ref:hall2013lie}.

\begin{definition}[Left action and orbit]
\label{def:group-action+orbit}
Let $G$ be a Lie group and let $\mathcal M$ be a set. A \emph{left action} of $G$ on $\mc M$ is a map $\phi :  G \times \mathcal M\to \mathcal M$ such that:
\begin{enumerate}[label=\textnormal{(\roman*)}]
    \item for all $\rho \in \mathcal M$, $\phi(e, \rho) = \rho$,
    \item for all $g,h \in  G$ and $\rho \in \mathcal M$, $\phi(g h, \rho) = \phi(g, \phi(h, \rho))$,
\end{enumerate}
where $e$ is the identity element of $G$. The \emph{orbit} of $\rho \in \mathcal M$ under the action $\phi$ of $G$ is the set $ G_\rho =  \{\phi(g, \rho)\mid g \in  G\}$.
\end{definition}
Orbits induce an equivalence relation on $\mathcal M$: two points $\rho_1,\rho_2\in\mathcal M$ are equivalent whenever one can be reached from the other by applying an element of $G$, that is, whenever there exists $g\in G$ such that $
\rho_2=\phi(g,\rho_1)$.
Thus, the equivalence classes are the orbits of the action.

\section{Distributions in Lie Group Orbits}
We denote the infinite-dimensional group of all smooth diffeomorphisms of $\mc X$ by $\mathrm{Diff}(\mc X)$, defined as $
      { \rm{Diff}}(\mc X)=\{ \Phi:\mc X\to\mc X\mid 
             \Phi\text{ is a }\mathcal C^{\infty}$ bijection and 
             $\Phi^{-1}\text{ is }\mathcal C^{\infty}\}.$
{The group operation on $\mathrm{Diff}(\mc X)$ is composition, the identity element is $\mathrm{id}_{\mc X}$, and inversion is the usual map inverse.}

{In the remainder of the paper, let \(G\) be a Lie group acting smoothly
on \(\mc X\). Thus we are given a smooth map $
  G\times\mc X\to\mc X$,~$(g,\bs x)\mapsto \alpha(g)(\bs x)$,
such that \(\alpha(e)=\mathrm{id}_{\mc X}\) and
\(\alpha(gh)=\alpha(g)\circ\alpha(h)\) for all \(g,h\in G\). Equivalently,
the action determines a group homomorphism $
  \alpha:G\to\mathrm{Diff}(\mc X)$,~$g\mapsto \alpha(g)$,
whose associated evaluation map \((g,\bs x)\mapsto\alpha(g)(\bs x)\) is
smooth. We refer to the image \(\alpha(G)\subseteq\mathrm{Diff}(\mc X)\)
as the associated acting transformation group. 
For example, let $G=\operatorname{GL}(d)$ and $\mc X=\mathbb{R}^d$, and define $\alpha: G \to \mathrm{Diff}(\mathbb{R}^d)$ by $\alpha(\bs A)(x)=\bs A x$. Then $\alpha$ is a group homomorphism, and $\alpha(G)$ is the group of linear diffeomorphisms of $\mathbb{R}^d$.}

The goal of this paper is to identify when and explain why the optimal transport problem admits closed-form solutions.
We show that many tractable instances occur when the source and target measures lie in the same orbit of a group acting on the outcome space. Accordingly, we begin by describing the push-forward action of subgroups of diffeomorphisms on probability measures, which lets us speak of orbits exactly as in classical group actions, but now at the level of distributions.

{The action $\alpha$ induces a left action of $G$ on $\mathcal P(\mc X)$ by push-forward: $
\phi: G \times \mathcal P(\mc X) \to \mathcal P(\mc X)$,
$\phi(g,\mu) = \alpha(g)_{\#}\mu$.
Indeed, since $\alpha(e)=\mathrm{id}_{\mc X}$ and $\alpha(gh)=\alpha(g)\circ\alpha(h)$, it follows that $
\alpha(e)_{\#}\mu = \mu$, and 
$\alpha(gh)_{\#}\mu = \alpha(g)_{\#}(\alpha(h)_{\#}\mu)$
for all $g,h\in G$ and $\mu\in\mathcal P(\mc X)$.}

With this action in hand, we introduce the basic objects it generates.
\begin{definition}[Orbit of a measure]\label{def:orbit-distribution}For $\rho \in \mathcal P(\mc X)$, we define its $G$-orbit as $
G_{\#}\rho = \{\alpha(g)_{\#}\rho \mid g \in G\} \subset \mathcal P(\mc X)$.
\end{definition}
Thus two measures are {equivalent} if one is the push-forward of the other by some $g\in G$.

\begin{lemma}[Orbit measures remain absolutely continuous]\label{lem:ac}
If $\rho\ll\mathcal L^{n}|_{\mc X}$, then for every $g\in  G$ we have $\alpha(g)_{\#}\rho\ll\mathcal L^{n}|_{\mc X}$ and, for $\mathcal L^{n}$-almost every $\bs y\in\mc X$, 
\[
\frac{\mathrm d(\alpha(g)_{\#}\rho)}{\mathrm d\mathcal L^{n}}(\bs y)
=
\frac{\mathrm d\rho}{\mathrm d\mathcal L^{n}}
\bigl(\alpha(g)^{-1}(\bs y)\bigr)
\left|\det\bigl(D(\alpha(g)^{-1})(\bs y)\bigr)\right|,
\]
where $D(\alpha(g)^{-1})(\bs y)$ denotes the Jacobian matrix of $\alpha(g)^{-1}$ evaluated at $\bs y$.
\end{lemma}
\begin{proof}
Since $\alpha(g)\in \mathrm{Diff}(\mc X)$, the map $\alpha(g)$ is a $\mathcal C^{1}$-diffeomorphism of $\mc X$. The conclusion therefore follows from the change-of-variables formula.
\end{proof}

\begin{definition}[Stabilizer subgroup]\label{def:stabilizer}
For a reference measure $\rho\in \mathcal P(\mc X)$, the \emph{stabilizer} of $\rho$ in $G$ is 
$
   \mathrm{Stab}_{G}(\rho)= \{
         h\in G\mid \alpha(h)_{\#}\rho=\rho\}.$
\end{definition}
Thus $\mathrm{Stab}_{G}(\rho)$ consists precisely of those elements of $G$ whose induced push-forward action leaves the reference measure $\rho$ invariant. $\mathrm{Stab}_G(\rho)$ is nonempty because $\alpha(e)=\mathrm{id}_{\mc X}$, and hence $\alpha(e)_{\#}\rho=\rho$.

With the action, orbits, and stabilizer in place, the subsequent section studies the optimal transport problem between measures lying on a common orbit.

\section{Transport of Distributions within Lie Group Orbits}
For the remainder of this paper, we fix a reference probability measure $\rho \in \mathcal{P}(\mc X) $. 
We study the Monge transportation problem \eqref{eq:monge-problem} induced by a Borel-measurable cost function $c:\mathcal{X}\times\mathcal{X}\to \R$, between two measures in the orbit of $\rho$, that is, $\mu_0, \mu_1 \in G_{\#}(\rho)$. 
Equivalently, there exist $g_0, g_1 \in G$, such that $\mu_0 = \alpha(g_0)_{\#}\rho$ and $\mu_1 = \alpha(g_1)_{ \#}\rho$. {\color{black}We fix one such pair \((g_0,g_1)\) throughout this section.} 
We assume $\rho\ll \mc L^n|_{\mc X}$, then by \Cref{lem:ac}, $\mu_0$ and $\mu_1$ are absolutely continuous as well, and we write $r_i=\mathrm d\mu_i/\mathrm d\mc L^n$ for their densities.

A crucial observation underpinning our analysis is that, while the Monge problem
is typically an infinite-dimensional and notoriously challenging optimization problem, a subtle algebraic structure emerges when the measures reside in a common $G$-orbit.
First, {\color{black} the set of admissible Monge maps} can be empty in general. When the source and target measures reside on a common orbit, the map
$T=\alpha(g_1)\circ \alpha(g_0)^{-1}\in \alpha(G)$ satisfies $T_{\#}\mu_0=\mu_1$, so $\mc T(\mu_0,\mu_1)$ is nonempty.
Moreover, every admissible transport admits a canonical factorization through the reference law: it can be pulled back to the $\rho$-coordinates. The next lemma formalizes this observation and serves as our bridge from arbitrary transportation maps to $\rho$-preserving transformations.
\begin{lemma}\label{lem:factorisation}
If a measurable map $T:\mathcal X \to \mathcal X$ satisfies
$T_{\#}\mu_0=\mu_1$, then the composite $
   H = \alpha(g_1)^{-1}\circ T\circ \alpha(g_0)$ is such that $ H_{\#}\rho=\rho $ and $
   T=\alpha(g_1)\circ H\circ \alpha(g_0)^{-1}.$
\end{lemma}
\begin{proof}
By the functoriality of push-forward measures and the identities
$\mu_i=\alpha(g_i)_\#\rho$, $i=0,1$, we have
\[
\begin{aligned}
H_\#\rho
&=
\bigl(\alpha(g_1)^{-1}\circ T\circ\alpha(g_0)\bigr)_\#\rho =
\bigl(\alpha(g_1)^{-1}\bigr)_\#
T_\#
\bigl(\alpha(g_0)_\#\rho\bigr) =
\bigl(\alpha(g_1)^{-1}\bigr)_\#
\bigl(T_\#\mu_0\bigr) \\
&=\bigl(\alpha(g_1)^{-1}\bigr)_\#\mu_1 =
\bigl(\alpha(g_1)^{-1}\bigr)_\#
\bigl(\alpha(g_1)_\#\rho\bigr)
=
\rho.
\end{aligned}
\]
The identity $
T=\alpha(g_1)\circ H\circ\alpha(g_0)^{-1}$
follows directly from the definition of $H$.
\end{proof}
 Lemma~\ref{lem:factorisation} shows that any admissible map
\(T\in\mathcal T(\mu_0,\mu_1)\) can be pulled back to a
\(\rho\)-preserving transformation \(H\) on the reference space, and
then recovered by {\color{black}composing} with the orbit representatives
\(\alpha(g_0)\) and \(\alpha(g_1)\).
 Thus, the Monge problem can be viewed as a search over $\rho$-preserving maps on the reference space. If, in addition, $H$ belongs to the acting transformation group $\alpha(G)$, say $H=\alpha(h)$ for some $h\in G$, then $h\in \mathrm{Stab}_{G}(\rho)$. To systematically exploit this structure, we introduce the \emph{Lie group orbit transport problem} between $\alpha(g_0)_{\#}\rho$ and $\alpha(g_1)_{\#}\rho$ induced by the cost function $c$ as follows:
\begin{equation}
    \mathds J_c(g_0, g_1) = \inf\limits_{h\in \mathrm{Stab}_{G}(\rho)}\int_{\mathcal X} c(\alpha(g_0)(\bs x), \alpha(g_1h)(\bs x)) \diff \rho(\bs x).
    \tag{LGOP}
    \label{eq:transport-wihtin-orbits}
\end{equation}
{\color{black}\begin{lemma}
\label{lem:Jc-well-defined}
Let \(s_0,s_1\in\operatorname{Stab}_G(\rho)\). Then $
  \mathds J_c(g_0s_0,g_1s_1)=\mathds J_c(g_0,g_1)$.
\end{lemma}
\begin{proof}
By definition, $
\mathds J_c(g_0s_0,g_1s_1)
=
\inf_{h\in\operatorname{Stab}_G(\rho)}
\int_{\mathcal X}
c(\alpha(g_0s_0)(\bs x),\alpha(g_1s_1h)(\bs x))\diff \rho(\bs x)$.
 
 As \(s_0\in\operatorname{Stab}_G(\rho)\), the change of variables
\(\bs z=\alpha(s_0)(\bs x)\) preserves \(\rho\). Therefore $\mathds J_c(g_0s_0,$ $ g_1s_1)$
equals
\[
\inf_{h\in\operatorname{Stab}_G(\rho)}
\int_{\mathcal X}
c\bigl(\alpha(g_0)(\bs z),\alpha(g_1s_1hs_0^{-1})(\bs z)\bigr)\diff\rho(\bs z).
\]
Since \(\operatorname{Stab}_G(\rho)\) is a subgroup of $G$, the map $
  h\mapsto s_1hs_0^{-1}$
is a bijection of \(\operatorname{Stab}_G(\rho)\) onto itself. Hence
\[
\mathds J_c(g_0s_0,g_1s_1)
=
\inf_{\widetilde h\in\operatorname{Stab}_G(\rho)}
\int_{\mathcal X}
c\bigl(\alpha(g_0)(\bs z),\alpha(g_1\widetilde h)(\bs z)\bigr)\diff \rho(\bs z)
=
\mathds J_c(g_0,g_1).
\]
\end{proof}

By \Cref{lem:Jc-well-defined}, one may equivalently regard
\(\mathds J_c\) as a function of the quotient representatives
\(g_0\operatorname{Stab}_G(\rho)\) and
\(g_1\operatorname{Stab}_G(\rho)\), or equivalently as a function of the
orbit measures \(\mu_0,\mu_1\).
}


{Moreover, because every $h \in \Stab_{G}(\rho)$ produces an admissible transport map $T_h = \alpha(g_1 h g_0^{-1})$, \eqref{eq:transport-wihtin-orbits} provides a structured upper bound on \eqref{eq:monge-problem} indexed by the stabilizer of the reference law.
Since $\Stab_G(\rho)$ is finite-dimensional, this reduces the search to
a finite-dimensional problem; when the stabilizer is trivial, it reduces
to evaluating a single candidate map.}

{
The key issue is when this upper bound is tight. Under the hypotheses of
\Cref{thm:monge-kantorovich}, optimality of a candidate map is certified
by the existence of a $c$-convex potential satisfying
\eqref{eq:T-optimal-c-convex-equation}. In general, finding such a
potential is an analytic problem and, for smooth costs, leads to the
partial differential equations of optimal transport
\citep[\S~12]{Villani2003}, for which explicit solutions are rarely
available.}

{The point of the orbit framework is that, for maps of the special form $
  T_h=\alpha(g_1hg_0^{-1})$,
this analytic certificate can sometimes be verified algebraically. When
this happens, $T_h$ is optimal for the Monge problem, the orbit upper
bound is tight, and the same element $h$ automatically solves
\eqref{eq:transport-wihtin-orbits}. The following theorem formalizes this
certificate.}
\begin{theorem}
\label{thm:closed-form}
The following statements hold.
\begin{enumerate}[label=\textnormal{(\roman*)}]
    \item $\mathds M_{c}(\mu_0, \mu_1) \leq \mathds J_c(g_0, g_1)$.
    \item {\color{black}Suppose that $\mathds K_c(\mu_0, \mu_1)<\infty$, that \(c\)
    satisfies Assumptions~\ref{ass:c-regularity}
    and~\ref{ass:H-infty}, and that the pair \((c,\mu_0)\) satisfies
    Assumption~\ref{ass:c-mu-0}.} Suppose further that
there exist { $h\in \mathrm{Stab}_G(\rho)$} and a $c$-convex function $\varphi$ such that the map $T_h = \alpha(g_1 h g_0^{-1})$ satisfies \eqref{eq:T-optimal-c-convex-equation} $\mu_0$-almost surely. Then:
    \begin{enumerate}[label=\textnormal{(\alph*)}]
         \item $T_h$ solves \eqref{eq:monge-problem}, and it is the unique optimal Monge map up to $\mu_0$-almost
everywhere equality,
        \item $h$ solves \eqref{eq:transport-wihtin-orbits},
        \item $(\mathrm{id}_{\mc X},T_h)_\#\mu_0$ solves
        \eqref{eq:kantorovich-problem} and it is the unique optimal Kantorovich plan,
        \item $
          \mathds K_c(\mu_0,\mu_1)
          =
          \mathds M_c(\mu_0,\mu_1)
          =
          \mathds J_c(g_0,g_1)$.
    \end{enumerate}
\end{enumerate}
\end{theorem}

\begin{proof} We first show that \eqref{eq:monge-problem} provides a lower bound on \eqref{eq:transport-wihtin-orbits}, and this proves the first claim in the theorem statement. 
To this end, choose any $h$ feasible in \eqref{eq:transport-wihtin-orbits}, and define $T_h = \alpha( g_1  hg_0^{-1})$. {\color{black}Since \(h\in\mathrm{Stab}_G(\rho)\), we have
\[
\begin{aligned}
(T_h)_\#\mu_0 = 
\alpha(g_1hg_0^{-1})_\#\alpha(g_0)_\#\rho =
\alpha(g_1h)_\#\rho =
\alpha(g_1)_\#\alpha(h)_\#\rho =
\alpha(g_1)_\#\rho
=
\mu_1.
\end{aligned}
\]
Thus \(T_h\in\mathcal T(\mu_0,\mu_1)\).}
We next verify that the objective value attained by $h$ in \eqref{eq:transport-wihtin-orbits} coincides with the value attained by $T_h$ in \eqref{eq:monge-problem}:
\begin{align*}
\int_{\mc X} c\bigl(\alpha(g_0)(\bs x),\alpha(g_1h)(\bs x)\bigr)\diff\rho(\bs x)
&=
\int_{\mc X} c\bigl(\alpha(g_0)(\bs x),(\alpha(g_1)\circ \alpha(h))(\bs x)\bigr)\diff\rho(\bs x) \\
&=
\int_{\mc X} c\bigl(\alpha(g_0)(\bs x),(T_h\circ \alpha(g_0))(\bs x)\bigr)\diff\rho(\bs x) \\
&=
\int_{\mc X} c\bigl(\bs x,T_h(\bs x)\bigr)\diff\mu_0(\bs x).
\end{align*}
Taking the infimum over $h$ in both sides of the equality above implies that \eqref{eq:monge-problem} provides a lower bound on \eqref{eq:transport-wihtin-orbits}.
This observation proves assertion~(i).

In the remainder of the proof, assume that $c$ satisfies
Assumptions~\ref{ass:c-regularity} and~\ref{ass:H-infty},
that the pair $(c,\mu_0)$ satisfies
Assumption~\ref{ass:c-mu-0}, and that there exist
$h\in\Stab_G(\rho)$ and a $c$-convex function $\varphi$ such that
$T_h$ satisfies \eqref{eq:T-optimal-c-convex-equation}
$\mu_0$-almost surely.
Set $
  \Gamma_h=(\mathrm{id}_{\mc X},T_h)_\#\mu_0$.
Since \(T_h\in\mathcal T(\mu_0,\mu_1)\), we have
\(\Gamma_h\in\Pi(\mu_0,\mu_1)\). Moreover, the assumed identity in \eqref{eq:T-optimal-c-convex-equation}
is equivalent to $
  \nabla_{\bs x}\varphi(\bs x)
  +
  \nabla_{\bs x}c(\bs x,\bs y)
  =
  0$,~$\Gamma_h\text{-almost surely}$.
By
\Cref{thm:monge-kantorovich}(v), the coupling \(\Gamma_h\) is the unique
optimal Kantorovich plan. This proves assertion (ii)(c). Since \(\Gamma_h\) is induced by the
deterministic map \(T_h\), the map \(T_h\) solves the Monge problem.
The uniqueness of \(T_h\) up to \(\mu_0\)-almost everywhere equality
follows from \Cref{thm:monge-kantorovich}. This proves (ii)(a).
 Hence $
  \mathds M_c(\mu_0,\mu_1)
  =
  \int_{\mc X}
  c(\bs x,T_h(\bs x))\diff\mu_0(\bs x)$.
As $T_h = \alpha(g_1 h g_0^{-1})$, we have 
\[
  \mathds M_c(\mu_0,\mu_1)
  =
  \int_{\mc X}
  c\bigl(\alpha(g_0)(\bs z),\alpha(g_1h)(\bs z)\bigr)
  \diff\rho(\bs z).\]
Since $h\in\Stab_G(\rho)$, this $h$ is feasible for
\eqref{eq:transport-wihtin-orbits}. Therefore
\[
  \mathds J_c(g_0,g_1)
  \le
  \int_{\mc X}
  c\bigl(\alpha(g_0)(\bs z),\alpha(g_1h)(\bs z)\bigr)
  \diff\rho(\bs z)
  =
  \mathds M_c(\mu_0,\mu_1).
\]
Together with assertion~\textnormal{(i)}, the inequality above implies $
  \mathds M_c(\mu_0,\mu_1)=\mathds J_c(g_0,g_1)$, and that the feasible element $h$
attains the infimum in \eqref{eq:transport-wihtin-orbits}, proving (ii)(b).

Since $(\mathrm{id}_{\mc X},T_h)_\#\mu_0$ has the same cost as the
optimal Monge map, we obtain $
  \mathds K_c(\mu_0,\mu_1)
  =
  \mathds M_c(\mu_0,\mu_1)
  =
  \mathds J_c(g_0,g_1)$. This proves (ii)(d).
\end{proof}

\subsection{Quadratic cost}
\label{sec:quadratic-cost}
We now specialize to the quadratic cost.
Let $\mathcal H$ be a finite-dimen\-si\-onal real Hilbert space with inner
product $\langle\cdot,\cdot\rangle_{\mathcal H}$ and induced norm
$\|\cdot\|_{\mathcal H}$, and let $\mathcal X\subseteq\mathcal H$ be a
nonempty open convex set. Unless stated otherwise, the cost function is
the quadratic cost $
c(\bs x,\bs y)=\|\bs x-\bs y\|_{\mathcal H}^{2}$, $\bs x,\bs y\in\mathcal X$,
and the reference measure $\rho\in\mathcal P(\mathcal X)$ satisfies $
\rho\ll
\mathcal L_{\mathcal H}\big|_{\mathcal B(\mathcal X)}$ and
$\int_{\mathcal X}\|\bs x\|_{\mathcal H}^{2}\diff\rho(\bs x)<\infty$,
where $\mathcal L_{\mathcal H}$ denotes the Lebesgue measure on
$\mathcal H$ induced by its Hilbert structure.

The following lemma records the basic convex-analytic identities for the quadratic cost that will be used repeatedly in the remainder of the paper.

{\color{black}

\begin{lemma}[Quadratic $c$-convexity and $c$-subdifferential]
\label{lem:quadratic-c-convex}
For a function $\psi:\mc X\to\bar\R$, define $
  \bar\psi(\bs x)=\psi(\bs x)+\|\bs x\|_{\mc H}^2$.
Then the following hold:
\begin{enumerate}[label=\textnormal{(\roman*)}]
  \item If $\psi$ is $c$-convex, then $\bar\psi$ is convex. 
  \item For $\bs x,\bs y\in\mc X$, $
    \bs y\in\partial^c\psi(\bs x)$ if and only if $
    2\bs y\in\partial\bar\psi(\bs x)$, where $\partial\bar\psi(\bs x)$ denotes the
convex subdifferential of $\bar\psi$ at $\bs x$. 
  \item If $\mu_0\ll\mc L_{\mc H}$, then
  Assumption~\ref{ass:c-mu-0} holds.
  \item $c$ satisfies Assumption~\ref{ass:H-infty} on $\mc H$.
\end{enumerate}
\end{lemma}
\begin{proof}
Recall that $
  \partial\bar\psi(\bs x)
  =\{
  \bs p\in\mc H:
  \bar\psi(\bs z)\ge
  \bar\psi(\bs x)+\langle \bs p,\bs z-\bs x\rangle_{\mc H}~
  \forall \bs z\in\mc X\}$.
For the quadratic cost, $
  c(\bs x,\bs y)
  =
  \|\bs x-\bs y\|_{\mc H}^2
  =
  \|\bs x\|_{\mc H}^2
  -2\langle \bs x,\bs y\rangle_{\mc H}
  +\|\bs y\|_{\mc H}^2$.

\noindent\textnormal{(i)}
Suppose that $\psi$ is $c$-convex. Then there exists a function
$\phi:\mc X\to\R\cup\{\pm\infty\}$ such that $
  \psi(\bs x)
  =
  \sup_{\bs y\in\mc X}
  \{
    \phi(\bs y)-c(\bs x,\bs y)
  \}$ for all $ \bs x\in\mc X$.
Then, we have $
  \bar\psi(\bs x) =
  \sup_{\bs y\in\mc X}
  \{
    \phi(\bs y)
    -\|\bs y\|_{\mc H}^2
    +2\langle \bs x,\bs y\rangle_{\mc H}
  \}$.
For each fixed $\bs y\in\mc X$, the map $
  \bs x\mapsto
  \phi(\bs y)-\|\bs y\|_{\mc H}^2
  +2\langle \bs x,\bs y\rangle_{\mc H}$
is affine on $\mc H$. Hence $\bar\psi$ is the pointwise supremum of
affine functions, and therefore is convex on $\mc X$.
\noindent\textnormal{(ii)}
Fix $\bs x,\bs y\in\mc X$. By definition,
$\bs y\in\partial^c\psi(\bs x)$ if and only if $
  \psi(\bs z)+\|\bs z\|_{\mc H}^2
  \ge
  \psi(\bs x)+\|\bs x\|_{\mc H}^2
  +2\langle \bs y,\bs z-\bs x\rangle_{\mc H}$ for all $ \bs z\in\mc X$,
that is, $
  \bar\psi(\bs z)
  \ge
  \bar\psi(\bs x)
  +\langle 2\bs y,\bs z-\bs x\rangle_{\mc H}$ for all $\bs z\in\mc X$, which corresponds to the condition $
  2\bs y\in\partial\bar\psi(\bs x)$.
\noindent\textnormal{(iii)}
Let $\psi$ be $c$-convex. By part \textnormal{(i)}, $\bar\psi$ is
convex. Let $
  \mathrm{dom}_{\mathrm{eff}}(\bar\psi)=
  \{\bs x\in\mc X:\bar\psi(\bs x)<\infty\}$
be the effective domain of $\bar\psi$. Since $\bar\psi$ is convex,
$\mathrm{dom}_{\mathrm{eff}}(\bar\psi)$ is convex. By
\citep[Theorem~2.1.12]{ref:borwein2010convex}, $\bar\psi$ is locally
Lipschitz on the interior of its effective domain. Therefore, by \citep[Theorem~3.2]{ref:evans2025measure}, $\bar\psi$
is differentiable $\mc L_{\mc H}$-almost everywhere on
$\operatorname{int}(\mathrm{dom}_{\mathrm{eff}}(\bar\psi))$.
Since $
  \psi=\bar\psi-\|\cdot\|_{\mc H}^2$,
and $\bs x\mapsto \|\bs x\|_{\mc H}^2$ is smooth, it follows that $\psi$
is also differentiable $\mc L_{\mc H}$-almost everywhere on
$\operatorname{int}(\mathrm{dom}_{\mathrm{eff}}(\bar\psi))$.
Now let $\bs x\in\mathrm{dom}(\partial^c\psi)$. Then
$\partial^c\psi(\bs x)\neq\varnothing$. By part \textnormal{(ii)},
$\partial\bar\psi(\bs x)\neq\varnothing$, and therefore
$\bar\psi(\bs x)<\infty$. Thus $
  \mathrm{dom}(\partial^c\psi)\subseteq \mathrm{dom}_{\mathrm{eff}}(\bar\psi)$.
Since $\mathrm{dom}_{\mathrm{eff}}(\bar\psi)$ is convex, its boundary has
$\mc L_{\mc H}$-measure zero. Hence $\psi$ is
$\mc L_{\mc H}$-almost everywhere differentiable on
$\mathrm{dom}(\partial^c\psi)$. If $\mu_0\ll\mc L_{\mc H}$, the same holds
$\mu_0$-almost surely. Hence, Assumption~\ref{ass:c-mu-0} holds.
\noindent (iv) \(\mc H\) is a flat Riemannian manifold, its Riemannian distance is
\(d(\bs x,\bs y)=\|\bs x-\bs y\|_{\mc H}\), and its sectional curvature is identically zero. Therefore,
 \citep[Example~10.36]{ref:villani2008optimal} applied with \(M=\mc  X=\mc Y=\mc H\) and
\(c(\bs x,\bs y)=d(\bs x,\bs y)^2\), implies that the quadratic cost
\(c(x,y)=\|x-y\|_{\mc H}^2\) satisfies
Assumption~\ref{ass:H-infty}.
\end{proof}

The preceding lemma shows that, for the quadratic cost, the shifted
potential $\bar\psi=\psi+\|\cdot\|_{\mc H}^2$ is convex, and the
$c$-subdifferential relation becomes an ordinary subgradient relation
for this shifted potential. Consequently, when the potential is
differentiable, the first-order condition in
Theorem~\ref{thm:monge-kantorovich} identifies the optimal map as the
gradient of a convex function, in agreement with Brenier's
convex-gradient characterization of quadratic optimal transport
\citep[Theorem~1.1-1.2]{Brenier1991}.

{\color{black}We now turn to the additional structure supplied by the group action.
While the group describes the transformations relating the measures,
optimal transport imposes a much more restrictive geometric
requirement on the transformation between them. To identify when these
two structures are compatible, it is useful to first understand the
group transformations infinitesimally, through their behavior near the
identity. For a Lie group, this local structure is encoded by its Lie
algebra.}
We begin with the necessary Lie-algebraic background. Recall that the identity element of $G$ is denoted by $e$. We write $\Lie(G)$ for its Lie
algebra, that is, the tangent space $T_e G$ endowed with its canonical
Lie bracket. 
{A \textit{Lie group involution} is a smooth group automorphism $\theta : G \to G$ satisfying $\theta^2 = \mathrm{id}_G$, where $\mathrm{id}_G : G \to G$ satisfies $\mathrm{id}_G(g) = g$ for every $g \in G$.}
 Differentiating this identity at the
identity element $e$ gives $
  (\diff\theta_e)^2=\mathrm{id}_{\Lie(G)}$.
  Thus $\diff\theta_e$ is a linear involution on the {finite-dimensional real vector space} 
  $\Lie(G)$. Define 
\[
  \mathfrak k=\{X\in\Lie(G):\diff\theta_e(X)=X\},
  \qquad
  \mathfrak p=\{X\in\Lie(G):\diff\theta_e(X)=-X\}.
\]
{For every $X\in\Lie(G)$, we have $X=\frac12\bigl(X+\diff\theta_e(X)\bigr) +\frac12\bigl(X-\diff\theta_e(X)\bigr)$. Since $(\diff\theta_e)^2=\mathrm{id}_{\Lie(G)}$, the first term belongs to $\mathfrak k$ and the second belongs to $\mathfrak p$. Moreover, $\mathfrak k\cap\mathfrak p=\{0\}$, since $\diff\theta_e(X)=X=-X$ for every $X\in\mathfrak k\cap\mathfrak p$. Therefore, $ \Lie(G)=\mathfrak k\oplus\mathfrak p$.}

\begin{definition}[Cartan decomposition]
\label{def:cartan-decomposition}
Let
$\theta:G\to G$ be a Lie group involution. Write
\[
  \mathfrak g=\Lie(G),
  \qquad
  \mathfrak k=\{X\in\mathfrak g:\diff\theta_e(X)=X\},
  \qquad
  \mathfrak p=\{X\in\mathfrak g:\diff\theta_e(X)=-X\},
\]
and let $
  G^\theta=\{\ell\in G:\theta(\ell)=\ell\}$.
We say that $(G,\theta)$ admits a \emph{global Cartan decomposition}
if the map $
  G^\theta\times\mathfrak p\to G$,
  $(k,S)\mapsto k\exp(S)$,
is a diffeomorphism.
\end{definition}
Definition~\ref{def:cartan-decomposition} imposes a global property on $(G,\theta)$ that is stronger than the infinitesimal splitting $\mathfrak g=\mathfrak k\oplus\mathfrak p$. In particular, it implies that every $g\in G$ admits a unique factorization $g=k\exp(S)$ with $k\in G^\theta$ and $S\in\mathfrak p$, and that these factors depend smoothly on $g$. In the standard reductive setting, this factorization separates a compact, rotation-like component from a noncompact, stretching-like component. More precisely, when $G$ is a real reductive linear Lie group and $\theta$ is a Cartan involution, $G^\theta$ is the compact factor and the map $G^\theta\times\mathfrak p\to G,~(k,S)\mapsto k\exp(S)$,  gives the global Cartan decomposition; see, for example, \citep[\S~VII.2]{knapp1996lie}.

\begin{lemma}[Cartan absorption]
\label{lem:cartan-absorption}
Let $(G,\theta)$ admit a global Cartan decomposition in the sense of
Definition~\ref{def:cartan-decomposition}. Then, for every
$\ell_0,\ell_1\in G$, there exists $h\in G^\theta$ such that $
  \ell_1h\ell_0^{-1}\in\exp(\mathfrak p)$.
\end{lemma}
\begin{proof}
Define $
\eta:G\to G$ such that
$\eta(\ell)=\theta(\ell)^{-1}\ell$.
We will first show that
$\{\eta(\ell) : \ell \in G\}=\exp(\mathfrak p)$.
Let $\ell\in G$. By the Cartan decomposition, there exist $k\in G^\theta$ and $X\in\mathfrak p$ such that $\ell=k\exp(X)$.
Since $\theta(k)=k$ and $\diff \theta_e(X)=-X$, we have
\begin{equation}
\theta(\exp(X))=\exp({\diff \theta_e(X)})=\exp({-X}).
\label{eq:theta-inv-exp}
\end{equation}
Therefore
\begin{align*}
\eta(\ell)
&=
\theta(k \exp(X))^{-1}k \exp(X)
=
(\theta(k)\theta(\exp(X)))^{-1}k \exp(X)\\
&=
(k \exp({-X}))^{-1}k \exp(X)
=
\exp({2X})\in \exp(\mathfrak p),
\end{align*}
which implies $ \{\eta(\ell) : \ell \in G\} \subseteq \exp(\mathfrak p)$.
Conversely, if $p=\exp(Y)\in \exp(\mathfrak p)$ with $Y\in\mathfrak p$, then
\[
\eta(\exp({Y/2}))
=
\theta(\exp({Y/2}))^{-1}\exp({Y/2})
=
\exp({-Y/2})^{-1}\exp({Y/2})
=
\exp(Y)
=
p,
\]
where the second equality follows by \eqref{eq:theta-inv-exp}.
Hence $\exp(\mathfrak p)\subseteq \{\eta(\ell) : \ell \in G
\}$, and thus $\{\eta(\ell) : \ell \in G\}=\exp(\mathfrak p)$.

Next, for each $a\in G$, define the map $
\tau_a:\exp(\mathfrak p)\to G$ as
$\tau_a(p)=\theta(a)^{-1}pa$.
In what follows, we will show that $\{\tau_a(p): p \in \exp(\mathfrak p)\}\subseteq \exp(\mathfrak p)$ for every $a\in G$.
Indeed, if $p\in \exp(\mathfrak p)$, then by $\{\eta(\ell) : \ell\in G\}=\exp(\mathfrak p)$, there exists $x\in G$ such that
$p=\eta(x)=\theta(x)^{-1}x$. Therefore
\[
\tau_a(p)
=
\theta(a)^{-1}\theta(x)^{-1}x a
=
\theta(xa)^{-1}(xa)
=
\eta(xa)\in \exp(\mathfrak p).
\]
{Thus $\tau_a(\exp(\mathfrak p)) \subseteq \exp(\mathfrak p)$ for every $a \in G$.}
Now fix $\ell_0,\ell_1\in G$. By the Cartan decomposition, we may write $
\theta(\ell_0)^{-1}\ell_1 = k p$ for some $k\in G^\theta,~ p\in \exp(\mathfrak p)$.
Set $
h=k^{-1}\in G^\theta$.
Since $k\in G^\theta$, we have $\theta(k) =k$ by definition of $G^\theta$. 
As group automorphisms preserve inverses $\theta(h) = \theta(k)^{-1} =k^{-1} = h$. Hence $h$ indeed belongs to $G^\theta$ and we have $\theta(h)^{-1}=k$. Therefore,
\[
kpk^{-1}=\theta(k^{-1})^{-1} p k^{-1} = \tau_{k^{-1}}(p)\in \exp(\mathfrak p).
\]
Now, define $
q = kpk^{-1}\in \exp(\mathfrak p)$.
Then
\[
\ell_1 h \ell_0^{-1} = \ell_1 k^{-1} \ell_0^{-1} 
=
\theta(\ell_0)k pk^{-1} \ell_0^{-1}
=
\theta(\ell_0)q\ell_0^{-1}
=
\tau_{\ell_0^{-1}}(q)\in \exp(\mathfrak p),
\]
where the second equality follows because $\theta(\ell_0)^{-1} \ell_1 = kp$, and thus $\ell_1 = \theta(\ell_0) kp $, and the inclusion follows because $\tau_a(\exp(\mathfrak p))\subseteq \exp(\mathfrak p)$ for every $a \in G$.
This concludes our proof. 
\end{proof}
In words, Lemma~\ref{lem:cartan-absorption} shows that, after
right-composing by a suitable element of $G^\theta$, the relative
element between $\ell_0$ and $\ell_1$ can always be moved into
$\exp(\mathfrak p)$. Thus, the fixed-point factor in the Cartan
decomposition (\Cref{def:cartan-decomposition}) can be absorbed, leaving only the
$\exp(\mathfrak p)$-factor.

{\color{black}
\begin{theorem}[Cartan criterion]
\label{thm:nonlinear-cartan-exactness}
Let $\mu_i = \alpha(g_i)_\#\rho \in \mc P_2(\mc X)$, $i=0,1$ and let $\theta: G \to G$ be a Lie group involution. Suppose that:
\begin{enumerate}[label=\textnormal{(\roman*)}]
\item \label{ass:global-cartan-decomp} $(G,\theta)$ admits a global Cartan decomposition in the sense of \Cref{def:cartan-decomposition},
\item $G^\theta \subseteq \mathrm{Stab}_G(\rho)$, and 
\item {$D_{\bs x} \alpha(\exp(S)) = D_{\bs x}\alpha(\exp(S))^*$ for all $S \in \mathfrak p$ and $\bs x \in \mc X$.}
\end{enumerate}
Then, there exists $h \in G^\theta$,
such that:
\begin{enumerate}[label=\textnormal{(\alph*)}]
    \item $T_h = \alpha(g_1 h g_0^{-1})$ solves \eqref{eq:monge-problem}, and is the unique optimal Monge map up to $\mu_0$-almost everywhere equality,
    \item $(\mathrm{id}_{\mc X}, T_h)_\#\mu_0$ solves \eqref{eq:kantorovich-problem}, and is the unique optimal Kantorovich plan,
    \item $h$ solves \eqref{eq:transport-wihtin-orbits},
    \item $\mathds M_c(\mu_0, \mu_1) = \mathds K_c(\mu_0, \mu_1) = \mathds J_c(g_0, g_1)$.
\end{enumerate}
\end{theorem}}
\begin{proof}
By \Cref{lem:cartan-absorption}, there exists
$h\in G^\theta$ such that
$g_1hg_0^{-1}\in\exp(\mathfrak p)$. Hence there exists
$S\in\mathfrak p$ satisfying
$g_1hg_0^{-1}=\exp(S)$. Moreover, assumption~\textnormal{(ii)}
implies that $h\in\Stab_G(\rho)$.

Set $T_h=\alpha(g_1hg_0^{-1})=\alpha(\exp(S))$. Since $h\in\Stab_G(\rho)$, 
\[
(T_h)_\#\mu_0
=
\alpha(g_1hg_0^{-1})_\#
(\alpha(g_0)_\#\rho)
=
\alpha(g_1)_\#
(\alpha(h)_\#\rho)
=
\alpha(g_1)_\#\rho
=
\mu_1.\]
Hence $T_h\in\mc T(\mu_0,\mu_1)$.

\medskip
{\color{black}For $t\in[0,1]$, define
$F_t=\alpha(\exp(tS))$ and
$J_t(\bs x)=D_{\bs x}F_t$. Since $\mathfrak p$ is a vector space,
$tS\in\mathfrak p$ for every $t\in[0,1]$. Assumption~\textnormal{(iii)}
therefore yields
$J_t(\bs x)=J_t(\bs x)^*$ for every
$t\in[0,1]$ and $\bs x\in\mc X$.
For each $t$, the map $F_t$ is a diffeomorphism of $\mc X$,
with inverse $F_{-t}=\alpha(\exp(-tS))$. Differentiating
$F_{-t}\circ F_t=\mathrm{id}_{\mc X}$ at $\bs x$ gives
$D_{F_t(\bs x)}F_{-t}\circ J_t(\bs x)
=
\mathrm{id}_{\mc H}$.
Hence $J_t(\bs x)$ is invertible for every
$t\in[0,1]$ and $\bs x\in\mc X$.
Since the map $(t,\bs x)\mapsto F_t(\bs x)
=\alpha(\exp(tS))(\bs x)$ is smooth, the map
$t\mapsto J_t(\bs x)$ is continuous for every fixed
$\bs x\in\mc X$. Moreover, $J_0(\bs x)=\mathrm{id}_{\mc H}$.
Thus $\{J_t(\bs x)\}_{t\in[0,1]}$ is a continuous path of
invertible self-adjoint operators beginning at the identity.
Since $J_t(\bs x)$ is self-adjoint, all of its eigenvalues are real.
Moreover, $J_t(\bs x)$ is invertible for every $t\in[0,1]$, so none
of its eigenvalues can be zero. As $t\mapsto J_t(\bs x)$ is
continuous, its eigenvalues can be chosen continuously as functions
of $t$. Since $J_0(\bs x)=\mathrm{id}_{\mc H}$ has only positive
eigenvalues, no eigenvalue can become negative without first passing
through zero. Therefore, $J_t(\bs x)\succ0$ for every
$t\in[0,1]$.

Fix $\bs x_\circ\in\mc X$ and define
$\Phi(\bs x)
=
\int_0^1
\left\langle
T_h\bigl(\bs x_\circ+s(\bs x-\bs x_\circ)\bigr),
\bs x-\bs x_\circ
\right\rangle_{\mc H}\diff s$
for $\bs x\in\mc X$. This is well defined because $\mc X$ is convex.
Fix $\bs x\in\mc X$ and $\bs v\in\mc H$, and write
$\bs z_s=\bs x_\circ+s(\bs x-\bs x_\circ)$. Differentiating under the
integral sign gives
$D\Phi(\bs x)[\bs v]
=
\int_0^1
\left\langle
D_{\bs z_s}T_h[s\bs v],
\bs x-\bs x_\circ
\right\rangle_{\mc H}\diff s
+
\int_0^1
\left\langle
T_h(\bs z_s),\bs v
\right\rangle_{\mc H}\diff s$.
Since $D_{\bs z_s}T_h$ is self-adjoint, the first integrand may be
rewritten as
$\left\langle
sD_{\bs z_s}T_h[\bs x-\bs x_\circ],
\bs v
\right\rangle_{\mc H}$.
Consequently,
$D\Phi(\bs x)[\bs v]
=
\int_0^1
\left\langle
sD_{\bs z_s}T_h[\bs x-\bs x_\circ]+T_h(\bs z_s),
\bs v
\right\rangle_{\mc H}\diff s$.
Since
$\frac{\diff}{\diff s}\bigl(sT_h(\bs z_s)\bigr)
=
T_h(\bs z_s)
+
sD_{\bs z_s}T_h[\bs x-\bs x_\circ]$,
the fundamental theorem of calculus yields
$D\Phi(\bs x)[\bs v]
= \langle sT_h(\bs z_s),\bs v\rangle_{\mc H}|_{s=0}^{s=1}=
\left\langle T_h(\bs x),\bs v\right\rangle_{\mc H}$.
Therefore $\nabla_{\bs x}\Phi(\bs x)=T_h(\bs x)$ for every
$\bs x\in\mc X$.
Since $T_h=F_1$, we have
$\nabla_{\bs x}^2\Phi(\bs x)=D_{\bs x}T_h=J_1(\bs x)$.
By the previous step, $J_t(\bs x) \succ 0$ for every $t \in [0,1]$ and thus $J_1(\bs x)\succ0$ for every
$\bs x\in\mc X$. Hence $\Phi$ is strictly convex on $\mc X$.

To obtain a convex function on the ambient Hilbert space, define
$\overline{\Phi}:\mc H\to(-\infty,+\infty]$ by
$\overline{\Phi}(\bs y)
=
\sup_{\bs x\in\mc X}
\{
\Phi(\bs x)
+
\langle T_h(\bs x),\bs y-\bs x\rangle_{\mc H}
\}$.
As the supremum of continuous affine functions,
$\overline{\Phi}$ is lower semicontinuous and convex.
For every $\bs y\in\mc X$, convexity and differentiability of $\Phi$
give
$\Phi(\bs y)
\geq
\Phi(\bs x)
+
\langle T_h(\bs x),\bs y-\bs x\rangle_{\mc H}$
for every $\bs x\in\mc X$. Hence
$\overline{\Phi}(\bs y)\leq\Phi(\bs y)$. Taking $\bs x=\bs y$ in the
definition of $\overline{\Phi}$ gives the reverse inequality, and
therefore $\overline{\Phi}(\bs y)=\Phi(\bs y)$ for every
$\bs y\in\mc X$. Thus $\overline{\Phi}$ is proper and agrees with
$\Phi$ on the open set $\mc X$. In particular,
$\overline{\Phi}$ is differentiable on $\mc X$ and
$\nabla_{\bs x}\overline{\Phi}(\bs x)=T_h(\bs x)$ for every
$\bs x\in\mc X$.

We now regard $\mu_0$ and $\mu_1$ as probability measures on
$\mc H$ supported on $\mc X$. More precisely, if
$\iota:\mc X\hookrightarrow\mc H$ denotes the inclusion map, we
identify each $\mu_i\in\mc P(\mc X)$ with $\iota_\#\mu_i\in
\mc P(\mc H)$. Define the ambient
quadratic cost $\bar c:\mc H\times\mc H\to\R$ by
$\bar c(\bs x,\bs y)=\|\bs x-\bs y\|_{\mc H}^2$.

Define $\varphi:\mc H\to\bar\R$ by
$\varphi(\bs x)=2\overline{\Phi}(\bs x)-\|\bs x\|_{\mc H}^2$.
Let $\overline{\Phi}^{\dagger}$ denote the Fenchel conjugate of
$\overline{\Phi}$ on $\mc H$, and define
$\chi:\mc H\to\R\cup\{\pm\infty\}$ by
$\chi(\bs y)=\|\bs y\|_{\mc H}^2
-2\overline{\Phi}^{\dagger}(\bs y)$.

We first show that $\varphi$ is $\bar c$-convex on $\mc H$. For
every $\bs x\in\mc H$, the definitions of $\chi$ and $\bar c$ give
$\sup_{\bs y\in\mc H}
\{\chi(\bs y)-\bar c(\bs x,\bs y)\}
=
-\|\bs x\|_{\mc H}^2
+
2\sup_{\bs y\in\mc H}
\{\langle\bs x,\bs y\rangle_{\mc H}
-\overline{\Phi}^{\dagger}(\bs y)\}
=
-\|\bs x\|_{\mc H}^2
+
2\overline{\Phi}^{\dagger\dagger}(\bs x)$.
Since $\overline{\Phi}$ is proper, lower semicontinuous, and
convex, the Fenchel-Moreau theorem gives
$\overline{\Phi}^{\dagger\dagger}=\overline{\Phi}$. Consequently,
$\sup_{\bs y\in\mc H}
\{\chi(\bs y)-\bar c(\bs x,\bs y)\}
=
2\overline{\Phi}(\bs x)-\|\bs x\|_{\mc H}^2
=
\varphi(\bs x)$.
Moreover, $\varphi$ is not identically $+\infty$, because
$\overline{\Phi}$ is finite on the nonempty set $\mc X$. Thus,
$\varphi$ is $\bar c$-convex on $\mc H$.

For completeness, the maximizing point in the preceding
representation can be identified on $\mc X$. Indeed, for every
$\bs x\in\mc X$, the identity
$\nabla\overline{\Phi}(\bs x)=T_h(\bs x)$ implies that
$T_h(\bs x)\in\partial\overline{\Phi}(\bs x)$. Therefore, the
equality condition in the Fenchel-Young inequality gives
$\overline{\Phi}(\bs x)
+
\overline{\Phi}^{\dagger}(T_h(\bs x))
=
\langle\bs x,T_h(\bs x)\rangle_{\mc H}$.
It follows that
$\chi(T_h(\bs x))
-
\bar c(\bs x,T_h(\bs x))
=
\varphi(\bs x)$.
Thus, for every $\bs x\in\mc X$, the supremum in the
$\bar c$-convex representation of $\varphi(\bs x)$ is attained at
$\bs y=T_h(\bs x)$.

Since $\nabla\overline{\Phi}=T_h$ on $\mc X$, we have
$\nabla_{\bs x}\varphi(\bs x)=2T_h(\bs x)-2\bs x$ for every
$\bs x\in\mc X$. Moreover,
$\nabla_{\bs x}\bar c(\bs x,T_h(\bs x))
=
2(\bs x-T_h(\bs x))$.
Therefore,
$\nabla_{\bs x}\varphi(\bs x)
+
\nabla_{\bs x}\bar c(\bs x,T_h(\bs x))
=
0$
for every $\bs x\in\mc X$. Thus, with the ambient quadratic cost $\bar c$, the map $T_h$
satisfies the optimality condition
\eqref{eq:T-optimal-c-convex-equation} at every
$\bs x\in\mc X$, and hence $\mu_0$-almost surely.

We next verify the hypotheses of
\Cref{thm:monge-kantorovich} on the ambient state space $\mc H$.
Since $\rho\ll\mc L_{\mc H}|_{\mc X}$ and $\alpha(g_0)$ is a
diffeomorphism of $\mc X$, \Cref{lem:ac} gives
$\mu_0\ll\mc L_{\mc H}|_{\mc X}$. Under the preceding ambient
identification, this implies $\mu_0\ll\mc L_{\mc H}$. Indeed, if
$\mc A\in\mc B(\mc H)$ satisfies $\mc L_{\mc H}(\mc A)=0$, then
$\mu_0(\mc A)=\mu_0(\mc A\cap\mc X)=0$. 

Thus, \Cref{lem:quadratic-c-convex}(iii) implies that
$(\bar c,\mu_0)$ satisfies Assumption~\ref{ass:c-mu-0}, while
\Cref{lem:quadratic-c-convex}(iv) implies that $\bar c$ satisfies
Assumption~\ref{ass:H-infty} on $\mc H$. The ambient quadratic cost
also satisfies Assumption~\ref{ass:c-regularity}. Moreover, since
$\mu_0,\mu_1\in\mc P_2(\mc X)$, their ambient extensions belong to
$\mc P_2(\mc H)$. The independent coupling therefore has finite
quadratic cost, because
$\bar c(\bs x,\bs y)
\leq
2\|\bs x\|_{\mc H}^2+2\|\bs y\|_{\mc H}^2$.
Hence, the ambient Kantorovich value is finite.

Set $\Gamma_h=(\mathrm{id}_{\mc X},T_h)_\#\mu_0$ and regard $\Gamma_h$ as
a probability measure on $\mc H\times\mc H$. Since
$(T_h)_\#\mu_0=\mu_1$, we have
$\Gamma_h\in\Pi(\mu_0,\mu_1)$. Furthermore, $\Gamma_h$ is
concentrated on pairs $(\bs x,\bs y)$ satisfying
$\nabla\varphi(\bs x)
+
\nabla_{\bs x}\bar c(\bs x,\bs y)
=
0$.
Therefore, \Cref{thm:monge-kantorovich}(v), applied with state
space $\mc H$, cost $\bar c$, and potential $\varphi$, implies
that $\Gamma_h$ is the unique optimal Kantorovich plan for the
ambient problem.

Because both marginals are supported on $\mc X$, every coupling
of $\mu_0$ and $\mu_1$ on $\mc H$ is concentrated on
$\mc X\times\mc X$. Since $\bar c$ agrees with $c$ on
$\mc X\times\mc X$, the ambient and original Kantorovich problems
have the same feasible couplings and the same optimal value.
Consequently, $(\mathrm{id}_{\mc X},T_h)_\#\mu_0$ is the unique optimal
Kantorovich plan for the original problem on $\mc X$.}

Since $T_h\in\mc T(\mu_0,\mu_1)$, we have
\[\mathds K_c(\mu_0,\mu_1)
\leq
\mathds M_c(\mu_0,\mu_1)
\leq
\int_{\mc X}c(\bs x,T_h(\bs x))\diff\mu_0(\bs x)
=
\mathds K_c(\mu_0,\mu_1).\]
Hence $T_h$ solves \eqref{eq:monge-problem}. If
$\widetilde T\in\mc T(\mu_0,\mu_1)$ is any other optimal Monge
map, then $(\mathrm{id}_{\mc X},\widetilde T)_\#\mu_0$ is an optimal
Kantorovich plan. By uniqueness of $\Gamma_h$, we have
$(\mathrm{id}_{\mc X},\widetilde T)_\#\mu_0
=
(\mathrm{id}_{\mc X},T_h)_\#\mu_0$, which implies
$\widetilde T=T_h$ $\mu_0$-almost surely. This proves
assertions~\textnormal{(a)} and~\textnormal{(b)}.

By \Cref{thm:closed-form}(i), we have
$\mathds M_c(\mu_0,\mu_1)\leq\mathds J_c(g_0,g_1)$. On the other
hand, the element $h$ constructed above belongs to
$\Stab_G(\rho)$ and is therefore feasible for
\eqref{eq:transport-wihtin-orbits}. Using
$\mu_0=\alpha(g_0)_\#\rho$, the left-action property, and the
definition of $T_h$, we obtain
\[\mathds J_c(g_0,g_1)
\leq
\int_{\mc X}
c\bigl(\alpha(g_0)(\bs z),\alpha(g_1h)(\bs z)\bigr)
\diff\rho(\bs z)
=
\int_{\mc X}c(\bs x,T_h(\bs x))\diff\mu_0(\bs x)
=
\mathds M_c(\mu_0,\mu_1).\]
Combining the two inequalities gives
$\mathds J_c(g_0,g_1)=\mathds M_c(\mu_0,\mu_1)$, and the
constructed element $h$ attains the infimum in
\eqref{eq:transport-wihtin-orbits}. This proves
assertion~\textnormal{(c)}.

Finally, the preceding argument already gives
$\mathds K_c(\mu_0,\mu_1)=\mathds M_c(\mu_0,\mu_1)$.
Together with
$\mathds J_c(g_0,g_1)=\mathds M_c(\mu_0,\mu_1)$, this yields
$\mathds K_c(\mu_0,\mu_1)
=
\mathds M_c(\mu_0,\mu_1)
=
\mathds J_c(g_0,g_1)$,
which proves assertion~\textnormal{(d)}.
\end{proof}

{\color{black}{
The central insight of
\Cref{thm:nonlinear-cartan-exactness} is that the symmetry-induced
nonuniqueness of orbit representations is a
resource for identifying the optimal transport map. Under
assumption~\textnormal{(ii)}, replacing $g_1$ by $g_1h$ with
$h\in G^\theta$ does not change the represented measure, since
$\alpha(g_1h)_\#\rho=\alpha(g_1)_\#\rho$. The Cartan absorption
mechanism uses this freedom to choose $h$ so that the relative
transformation satisfies $
    g_1hg_0^{-1}=\exp(S)$, $ S\in\mathfrak p$.
Thus, the component that can be absorbed without changing the target
measure is removed, leaving a stretch-like transformation canonically
connected to the identity through $
    F_t=\alpha(\exp(tS))$, $t\in[0,1]$.
Self-adjointness of $D_{\bs x}F_t$ plays two complementary roles. It
rules out an infinitesimal rotational component, allowing the terminal
map to be integrated to a scalar potential on the convex domain
$\mc X$, and it ensures that the spectrum of $D_{\bs x}F_t$ remains
real along the path. Since every $F_t$ is a diffeomorphism,
$D_{\bs x}F_t$ remains invertible; because the path begins at
$D_{\bs x}F_0=\mathrm{id}_{\mc H}$, none of its eigenvalues can change
from positive to negative without passing through zero. Hence the
terminal map is the gradient of a strictly convex potential on
$\mc X$. Convexity is therefore not imposed or verified separately for
each candidate map, but inherited from the identity through the
geometry of the action. The theorem consequently replaces the
pair-by-pair search for a convex transport potential by structural
conditions on the action and the symmetries of the reference measure.
Once these conditions are verified at the level of
$(G,\theta,\alpha,\rho)$, they certify, for every pair of measures in
the resulting orbit family with finite second moments, the unique optimal Monge map, the unique
optimal Kantorovich plan, and the equality $
    \mathds M_c(\mu_0,\mu_1)
    =
    \mathds K_c(\mu_0,\mu_1)
    =
    \mathds J_c(g_0,g_1)$.
}}

{\color{black}In what follows, we consider actions induced by affine representations on the ambient Hilbert space. In this setting, the full acting group is not required to admit a Cartan decomposition. Instead, it suffices to impose a Cartan structure on the linear image $L=\pi(G)$. The Cartan decomposition then selects an orbit-induced affine map with a positive-definite linear part. This map is determined explicitly by the means and covariance operators of the source and target measures, and the resulting optimal transport value attains the lower bound on $\mathds K_c(\mu_0,\mu_1)$ provided in \citep[Theorem~2.1]{Gelbrich1990}.}
\begin{assumption}[Affine-induced action]
\label{ass:affine-induced-pi}
There exist a smooth group homomorphism
$\pi:G\to\GL(\mc H)$ and a smooth map $b:G\to\mc H$ satisfying
$b(gh)=b(g)+\pi(g)b(h)$ for every $g,h\in G$, such that
$\alpha(g)(\bs x)=b(g)+\pi(g)\bs x$ for every
$g\in G$ and $\bs x\in\mc X$.
\end{assumption}

{
\begin{proposition}[Affine Cartan criterion]
\label{prop:affine-cartan-criterion}
Let
$\mc X\subseteq\mc H$ be nonempty, open, and convex, and let
$G$ act smoothly on $\mc X$ through $\alpha$. Suppose that
$\rho\in\mc P(\mc X)$ is absolutely continuous with respect to
$\mc L_{\mc H}|_{\mc X}$ and satisfies
$\int_{\mc X}\|\bs x\|_{\mc H}^2\diff\rho(\bs x)<\infty$.
Fix $g_0,g_1\in G$, set
$\mu_i=\alpha(g_i)_\#\rho$ for $i=0,1$, and let
$c(\bs x,\bs y)=\|\bs x-\bs y\|_{\mc H}^2$.

Suppose that Assumption~\ref{ass:affine-induced-pi} holds, and set
$L=\pi(G)\subseteq\GL(\mc H)$. Assume that $L$ is a Lie subgroup of
$\GL(\mc H)$, and let $\theta:L\to L$ be a Lie group involution.
Write
$\Lie(L)=\mathfrak k\oplus\mathfrak p$
for the corresponding eigenspace decomposition. Suppose further that:
\begin{enumerate}[label=\textnormal{(\roman*)}]
    \item
    $(L,\theta)$ admits a global Cartan decomposition in the sense of
    \Cref{def:cartan-decomposition};

    \item
    $L^\theta\subseteq\pi(\Stab_G(\rho))$;

    \item
    every $S\in\mathfrak p$ is self-adjoint on $\mc H$.
    \label{thm:hyp:cartan-self-adj}
\end{enumerate}

{\color{black}Let $\bs m_i\in\mc H$ and
$\mathsf C_i:\mc H\to\mc H$ denote the mean and covariance operator
of $\mu_i$, respectively, for $i=0,1$. Then
$\mathsf C_0$ and $\mathsf C_1$ are self-adjoint and positive
definite. Moreover, there exists $h\in\Stab_G(\rho)$ satisfying
$\pi(h)\in L^\theta$ such that, upon defining
$T_h:=\alpha(g_1hg_0^{-1})$ and
$\mathsf A_h=\pi(g_1hg_0^{-1})
\in\exp(\mathfrak p)\subseteq L\subseteq\GL(\mc H)$,
the operator $\mathsf A_h$ is self-adjoint and positive definite and
\begin{equation}
\begin{aligned}
    T_h(\bs x)
    &=
    \bs m_1+\mathsf A_h(\bs x-\bs m_0),
    \quad \bs x\in\mc X,\quad
    \mathsf A_h
    &=
    \mathsf C_0^{-1/2}
    \left(
        \mathsf C_0^{1/2}
        \mathsf C_1
        \mathsf C_0^{1/2}
    \right)^{1/2}
    \mathsf C_0^{-1/2}.
\end{aligned}
\label{eq:affine-cartan-map}
\end{equation}
Here the inverses and square roots in
\eqref{eq:affine-cartan-map} are taken in the cone of self-adjoint
positive-definite operators on $\mc H$.

For this choice of $h$, the following conclusions hold: \textnormal{(a)}
    The map $T_h$ solves \eqref{eq:monge-problem} and is the unique
    optimal Monge map from $\mu_0$ to $\mu_1$, up to
    $\mu_0$-almost-everywhere equality.
    \textnormal{(b)} The coupling
    $(\mathrm{id}_{\mc X},T_h)_\#\mu_0$
    solves \eqref{eq:kantorovich-problem} and is the unique optimal
    Kantorovich plan.
\textnormal{(c)}
    The element $h$ solves the Lie group orbit transport problem
    \eqref{eq:transport-wihtin-orbits}.
\textnormal{(d)}
    $\mathds K_c(\mu_0,\mu_1)
        =
        \mathds M_c(\mu_0,\mu_1)
        =
        \mathds J_c(g_0,g_1)$ and
    \begin{equation}
    \begin{aligned}
        \mathds K_c(\mu_0, \mu_1)=
        \|\bs m_0-\bs m_1\|_{\mc H}^2
        +
        \operatorname{Tr}_{\mc H}
        \left(
            \mathsf C_0+\mathsf C_1
            -
            2
            \left(
                \mathsf C_0^{1/2}
                \mathsf C_1
                \mathsf C_0^{1/2}
            \right)^{1/2}
        \right).
    \end{aligned}
    \label{eq:affine-cartan-value}
    \end{equation}}
\end{proposition}

\begin{proof}
Set $\ell_i=\pi(g_i)\in L$ for $i=0,1$. By
\Cref{lem:cartan-absorption}, there exists $k\in L^\theta$ such that
$\ell_1k\ell_0^{-1}\in\exp(\mathfrak p)$. Assumption~\textnormal{(ii)}
therefore yields $h\in\Stab_G(\rho)$ satisfying $\pi(h)=k$. In
particular, $\pi(h)\in L^\theta$.

Define $T_h:=\alpha(g_1hg_0^{-1})$ and
$\mathsf A_h:=\pi(g_1hg_0^{-1})$. Since $\pi$ is a group
homomorphism, $\mathsf A_h=\ell_1k\ell_0^{-1}\in\exp(\mathfrak p)$.
Thus there exists $S\in\mathfrak p$ such that
$\mathsf A_h=\exp(S)$. By assumption~\textnormal{(iii)}, $S$ is
self-adjoint on $\mc H$. The spectral theorem then implies that
$\mathsf A_h$ is self-adjoint and positive definite.

Set $\bs a_h=b(g_1hg_0^{-1})$. By
Assumption~\ref{ass:affine-induced-pi}, we have
$T_h(\bs x)=\bs a_h+\mathsf A_h\bs x$ for every $\bs x\in\mc X$.
Moreover, since $h\in\Stab_G(\rho)$ and $\alpha$ is a left action,
$(T_h)_\#\mu_0
=\alpha(g_1hg_0^{-1})_\#(\alpha(g_0)_\#\rho)
=\alpha(g_1)_\#(\alpha(h)_\#\rho)
=\alpha(g_1)_\#\rho
=\mu_1$.
Hence $T_h\in\mc T(\mu_0,\mu_1)$.

Assumption~\ref{ass:affine-induced-pi} defines an affine action
$\hat\alpha$ of $G$ on $\mc H$ by
$\hat\alpha(g)(\bs x)=b(g)+\pi(g)\bs x$. Indeed, the homomorphism
property of $\pi$ and the cocycle identity for $b$ imply
$\hat\alpha(gh)=\hat\alpha(g)\circ\hat\alpha(h)$, and
$\hat\alpha(g^{-1})$ is the inverse of $\hat\alpha(g)$.
Moreover, $\hat\alpha(g)|_{\mc X}=\alpha(g)$ for every $g\in G$.
We regard $\rho$, $\mu_0$, and $\mu_1$ as probability measures on
$\mc H$ supported on $\mc X$.

Since $\rho\ll\mc L_{\mc H}|_{\mc X}$, its density extends by zero to
a density on $\mc H$. Each $\hat\alpha(g_i)$ is an affine
diffeomorphism of $\mc H$, and hence $\mu_i\ll\mc L_{\mc H}$ for
$i=0,1$. Furthermore, the finite second moment of $\rho$ and the
bound
$\|b(g_i)+\pi(g_i)\bs x\|_{\mc H}^2
\leq2\|b(g_i)\|_{\mc H}^2
+2\|\pi(g_i)\|_{\mathrm{op}}^2\|\bs x\|_{\mc H}^2$
imply that $\mu_i\in\mc P_2(\mc H)$.

{\color{black}We next show that the covariance operators are positive definite.
Fix $i\in\{0,1\}$ and suppose that
$\langle\bs u,\mathsf C_i\bs u\rangle_{\mc H}=0$ for some
$\bs u\neq\bs 0$. Then
$\langle\bs u,\bs X-\bs m_i\rangle_{\mc H}=0$ holds
$\mu_i$-almost surely, so $\mu_i$ is concentrated on the affine
hyperplane
$\{\bs x\in\mc H:
\langle\bs u,\bs x-\bs m_i\rangle_{\mc H}=0\}$.
This hyperplane has $\mc L_{\mc H}$-measure zero, contradicting
$\mu_i\ll\mc L_{\mc H}$. Thus
$\mathsf C_0,\mathsf C_1\succ0$.

Let $\hat T_h:\mc H\to\mc H$ be the affine extension
$\hat T_h(\bs x)=\bs a_h+\mathsf A_h\bs x$. Then
$\hat T_h=\hat\alpha(g_1hg_0^{-1})$ and
$\hat T_h|_{\mc X}=T_h$. Define
$\Phi_h:\mc H\to\R$ by
$\Phi_h(\bs x)=\langle\bs a_h,\bs x\rangle_{\mc H}
+\frac12\langle\bs x,\mathsf A_h\bs x\rangle_{\mc H}$.
Since $\mathsf A_h=\mathsf A_h^*\succ0$, the function $\Phi_h$ is
strictly convex and satisfies $\nabla\Phi_h=\hat T_h$.

Let $\hat c(\bs x,\bs y)=\|\bs x-\bs y\|_{\mc H}^2$ on
$\mc H\times\mc H$, define
$\varphi_h(\bs x)=2\Phi_h(\bs x)-\|\bs x\|_{\mc H}^2$, and let
$\Phi_h^\dagger$ denote the Fenchel conjugate of $\Phi_h$. If
$\chi_h(\bs y)=\|\bs y\|_{\mc H}^2-2\Phi_h^\dagger(\bs y)$, then the
Fenchel-Moreau theorem gives
$\varphi_h(\bs x)=\sup_{\bs y\in\mc H}
\{\chi_h(\bs y)-\hat c(\bs x,\bs y)\}$
for every $\bs x\in\mc H$. Hence $\varphi_h$ is
$\hat c$-convex.

Moreover,
$\nabla\varphi_h(\bs x)=2\hat T_h(\bs x)-2\bs x$ and
$\nabla_{\bs x}\hat c(\bs x,\hat T_h(\bs x))
=2(\bs x-\hat T_h(\bs x))$.
Therefore,
$\nabla\varphi_h(\bs x)
+\nabla_{\bs x}\hat c(\bs x,\hat T_h(\bs x))=0$
for every $\bs x\in\mc H$.

Because $\hat\alpha(g)|_{\mc X}=\alpha(g)$ and $\rho$ is
supported on $\mc X$, the stabilizer of $\rho$ is the same for the two
actions. In particular, the element $h$ selected above is feasible for
the ambient Lie group orbit problem. We now apply
\Cref{thm:closed-form}(ii) with state space $\mc H$, action
$\hat\alpha$, and cost $\hat c$.

The quadratic cost satisfies Assumption~\ref{ass:c-regularity}; by
\Cref{lem:quadratic-c-convex}(iii), the absolute continuity of
$\mu_0$ implies Assumption~\ref{ass:c-mu-0}; and by
\Cref{lem:quadratic-c-convex}(iv), $\hat c$ satisfies
Assumption~\ref{ass:H-infty} on $\mc H$. Finally, since
$\mu_0,\mu_1\in\mc P_2(\mc H)$, the independent coupling has finite
quadratic cost, and hence the ambient Kantorovich value is finite.
Thus \Cref{thm:closed-form}(ii) implies that $\hat T_h$ is the
unique optimal Monge map up to $\mu_0$-almost-everywhere equality,
its graph coupling is the unique optimal Kantorovich plan, $h$ solves
the ambient Lie group orbit transport problem, and the three ambient
optimal values coincide.

These ambient problems agree with the corresponding problems on
$\mc X$. Indeed, $\rho$, $\mu_0$, and $\mu_1$ are supported on
$\mc X$, and $\hat\alpha(g)|_{\mc X}=\alpha(g)$ for every
$g\in G$. Hence the ambient and original stabilizers and Lie group
orbit objectives coincide. Every coupling of $\mu_0$ and $\mu_1$ on
$\mc H$ is supported on $\mc X\times\mc X$, so the ambient and
original Kantorovich problems coincide.

Finally, every Borel map from $\mc X$ to $\mc X$ can be extended to
a Borel map on $\mc H$, whereas an ambient transport map can be
modified on a $\mu_0$-null set and restricted to $\mc X$ to obtain a
transport map from $\mc X$ to $\mc X$ with the same cost. Thus the
ambient and original Monge problems also coincide. Since
$\hat T_h|_{\mc X}=T_h$, conclusions~\textnormal{(a)}-\textnormal{(c)} and the equality of the three values in
\textnormal{(d)} follow.

Let $\bs X\sim\mu_0$ and set
$\bs Y=T_h(\bs X)\sim\mu_1$. Taking expectations in
$\bs Y=\bs a_h+\mathsf A_h\bs X$ gives
$\bs m_1=\bs a_h+\mathsf A_h\bs m_0$. Therefore,
$\bs a_h=\bs m_1-\mathsf A_h\bs m_0$, and hence
$T_h(\bs x)=\bs m_1+\mathsf A_h(\bs x-\bs m_0)$
for every $\bs x\in\mc X$.

Since
$\bs Y-\bs m_1=\mathsf A_h(\bs X-\bs m_0)$, the covariance
operators satisfy
$\mathsf C_1=\mathsf A_h\mathsf C_0\mathsf A_h^*
=\mathsf A_h\mathsf C_0\mathsf A_h$.
To see this directly, for $\bs u,\bs v\in\mc H$ the covariance
bilinear form gives
$\langle\bs u,\mathsf C_1\bs v\rangle_{\mc H}
=\langle\mathsf A_h^*\bs u,
\mathsf C_0\mathsf A_h^*\bs v\rangle_{\mc H}
=\langle\bs u,
\mathsf A_h\mathsf C_0\mathsf A_h^*\bs v\rangle_{\mc H}$.

Set
$\mathsf R_h
=\mathsf C_0^{1/2}\mathsf A_h\mathsf C_0^{1/2}$.
Since $\mathsf C_0\succ0$ and $\mathsf A_h\succ0$, the operator
$\mathsf R_h$ is self-adjoint and positive definite. Moreover,
$\mathsf R_h^2
=\mathsf C_0^{1/2}\mathsf A_h\mathsf C_0
\mathsf A_h\mathsf C_0^{1/2}
=\mathsf C_0^{1/2}\mathsf C_1\mathsf C_0^{1/2}$.
By uniqueness of the positive-definite square root,
$\mathsf R_h
=(\mathsf C_0^{1/2}\mathsf C_1
\mathsf C_0^{1/2})^{1/2}$.
Multiplying on the left and right by
$\mathsf C_0^{-1/2}$ proves \eqref{eq:affine-cartan-map}.

Let $\bs Z=\bs X-\bs m_0$. Then $\EE[\bs Z]=\bs 0$ and
$\bs X-T_h(\bs X)
=\bs m_0-\bs m_1
+(\mathrm{id}_{\mc H}-\mathsf A_h)\bs Z$.
The cross term vanishes after taking expectations.

Using the identity
$\EE[\|B\bs Z\|_{\mc H}^2]
=\operatorname{Tr}_{\mc H}(B\mathsf C_0B^*)$
for a linear operator $B$ on $\mc H$, together with
$\mathsf A_h=\mathsf A_h^*$ and
$\mathsf A_h\mathsf C_0\mathsf A_h=\mathsf C_1$, we obtain
$\EE[\|\bs X-T_h(\bs X)\|_{\mc H}^2]
=
\|\bs m_0-\bs m_1\|_{\mc H}^2
+
\operatorname{Tr}_{\mc H}
(\mathsf C_0+\mathsf C_1-2\mathsf A_h\mathsf C_0)$.

Finally, cyclicity of the operator trace and the identity for
$\mathsf R_h$ give
\[\operatorname{Tr}_{\mc H}(\mathsf A_h\mathsf C_0)
=
\operatorname{Tr}_{\mc H}
(\mathsf C_0^{1/2}\mathsf A_h\mathsf C_0^{1/2})
=
\operatorname{Tr}_{\mc H}
((\mathsf C_0^{1/2}\mathsf C_1
\mathsf C_0^{1/2})^{1/2}).\]
Substitution proves \eqref{eq:affine-cartan-value} and completes the
proof.}
\end{proof}}

{\color{black}The expression on the right-hand side of
\eqref{eq:affine-cartan-value} is the lower bound on the quadratic
Kantorovich value established in
\citep[Theorem~2.1]{Gelbrich1990}.
\Cref{prop:affine-cartan-criterion} provides a sufficient structural
criterion for this lower bound to be attained. Indeed, the global
Cartan decomposition of $L=\pi(G)$, together with the inclusion
$L^\theta\subseteq\pi(\Stab_G(\rho))$, produces
$h\in\Stab_G(\rho)$ such that
$\pi(g_1hg_0^{-1})\in\exp(\mathfrak p)$; the self-adjointness of
$\mathfrak p$ then implies that this operator is positive definite.
The resulting orbit-induced affine map transports $\mu_0$ to
$\mu_1$ and attains the lower bound.}}

\section{Examples}
\label{sec:examples}
{\color{black}
This section applies the Cartan criteria in
\Cref{thm:nonlinear-cartan-exactness,prop:affine-cartan-criterion}
to derive closed-form quadratic optimal transport maps and values for
several distribution families. We begin with elliptical distributions,
for which the affine Cartan criterion recovers the classical optimal
transport formula for laws sharing a common characteristic generator,
equivalently the corresponding equality case of Gelbrich's covariance
lower bound \citep[Theorem~2.1]{Gelbrich1990}. We then turn to
spectral-congruence orbits of matrix-valued distributions and conclude
with the sinh-arcsinh family generated by a nonlinear group action.} 
\subsection{Elliptical distributions}
A random vector $\bs x\in\R^d$ has an elliptical distribution with characteristic generator $\varsigma:\R_{\geq}\to\R$ if its characteristic function evaluated at $\bs t \in \R^d$ is $\exp(i\bs t^\top\bs m)\varsigma(\bs t^\top \bs S \bs t),$
for some location $\bs m\in\R^d$ and dispersion $\bs S\in\mathbb S_{\succ}^d$; and we write $\bs x \sim \mathcal E_\varsigma(\bs m, \bs S)$.
If $\varsigma'(0)$ exists and is finite, {\color{black}then the location parameter is the mean, $\EE_{\bs x \sim\mc E_{\varsigma}(\bs m , \bs S)}[\bs x] = \bs m$ and} the covariance matrix of $\bs x \sim \mathcal E_\varsigma(\bs m, \bs S)$, satisfies
$\mathrm{Cov}(\bs x)=(-2\varsigma'(0))\bs S$.
Reparameterizing the generator $\varsigma$ as
$\tilde\varsigma(s)=\varsigma(s/(-2\varsigma'(0)))$ yields
$\mathrm{Cov}(\bs x)=\bs S$ with $\bs x \sim \mathcal E_{\tilde\varsigma}(\bs m,\bs S)$. 

The elliptical distribution class is broad and includes heavy-tailed examples such as
Cauchy and multivariate stable laws. In this subsection, however, we
restrict to absolutely continuous elliptical laws with finite second
moment, so that the quadratic cost is finite. This finite-moment
subclass includes the multivariate Gaussian, finite-variance Student-$t$,
and multivariate Laplace families.
We henceforth write $\varsigma$ for the normalized generator and let $\rho = \mathcal E_{\varsigma}(\bs 0_d, \bs I_d)$, which is absolutely continuous and has finite second moment. Hence, $\rho$ has mean $\bs 0_d$ and covariance matrix $\bs I_d$.

Let $\mc X=\mc H=\R^d$, equipped with the standard inner product
$\langle\bs x,\bs y\rangle_{\mc H}=\bs x^\top\bs y$. Let
$G_{\mathrm{aff}}=\R^d\rtimes\GL(d)$ be the affine group with
multiplication
$(\bs m,\bs A)(\bs n,\bs B)
=
(\bs m+\bs A\bs n,\bs A\bs B)$.
Its identity element is $(\bs 0_d,\bs I_d)$, and
$(\bs m,\bs A)^{-1}
=
(-\bs A^{-1}\bs m,\bs A^{-1})$.  The group $G_{\mathrm{aff}}$ acts smoothly on $\R^d$ through
$\alpha((\bs m,\bs A))(\bs x)=\bs m+\bs A\bs x$.

Define $b:G_{\mathrm{aff}}\to\mc H$ by
$b((\bs m,\bs A))=\bs m$ and
$\pi:G_{\mathrm{aff}}\to\GL(\mc H)$ by
$\pi((\bs m,\bs A))=\bs A$, and set
$L=\pi(G_{\mathrm{aff}})$.

{\color{black}\begin{corollary}[Closed-form quadratic transport between elliptical distributions]
\label{cor:elliptical-ot}
Let
$\mu_i=\mc E_{\varsigma}(\bs m_i,\bs\Sigma_i)$, where
$\bs m_i\in\R^d$ and
$\bs\Sigma_i\in\mathbb S_{\succ}^d$ for $i=0,1$, and set
$g_i=(\bs m_i,\bs\Sigma_i^{1/2})\in G_{\mathrm{aff}}$.
Define
$\bs A\opt
=
\bs\Sigma_0^{-1/2}
(\bs\Sigma_0^{1/2}\bs\Sigma_1\bs\Sigma_0^{1/2})^{1/2}
\bs\Sigma_0^{-1/2}$,
$\bs Q\opt
=
\bs\Sigma_1^{-1/2}\bs A\opt\bs\Sigma_0^{1/2}$,
and
$h\opt=(\bs 0_d,\bs Q\opt)$.
Then $\bs Q\opt\in\mc O(d)$ and
$h\opt\in\Stab_{G_{\mathrm{aff}}}(\rho)$.
Moreover, all the conclusions of
\Cref{prop:affine-cartan-criterion} hold with $h=h\opt$ and
$\mathsf C_i=\bs\Sigma_i$ for $i=0,1$.
\end{corollary}

The affine orbit representation and the Cartan structure play
complementary roles in this example. The former identifies the
elliptical family and its stabilizer, while the latter uses
$L^\theta=\pi(\Stab_{G_{\mathrm{aff}}}(\rho))=\mc O(d)$ and
$\mathfrak p=\Sym(d)$ to select the stabilizing element $h\opt$ for
which the relative linear transformation is self-adjoint and positive
definite. Consequently,
\Cref{prop:affine-cartan-criterion} identifies the induced map
$T_{h\opt}$ as the optimal affine transport and yields the common
Monge, Kantorovich, and Lie group orbit transport value.}
When
$\varsigma(s)=\exp(-s/2)$, the reference law $\rho$ is the
standard Gaussian distribution, and
\Cref{cor:elliptical-ot} recovers the classical Gaussian formulas of
\citep{DowsonLandau1982,OlkinPukelsheim1982}. For a general common
characteristic generator, it recovers 
\citep[Theorem~2.1]{Gelbrich1990}.

\subsection{Spectrally invariant distributions}
\label{subsec:congruence}
We next consider probability distributions on
$\mathbb S_{\succ}^d$ generated by congruence transformations of an
absolutely continuous, spectrally invariant reference law $\rho$, whose
density depends on $\bs X$ only through its unordered eigenvalues.
 We first introduce the congruence action and verify the Cartan structure required by \Cref{prop:affine-cartan-criterion}. We then identify the class of reference distributions for which the fixed-point subgroup of this Cartan decomposition is contained in the stabilizer. The resulting orbit contains, in particular, Wishart, inverse-Wishart, and matrix beta type~II families with general positive-definite scale matrices; although the reference law is spectrally invariant, a general member of its congruence orbit need not be.

Let the state space be
$\mc X=\mathbb S_{\succ}^d$, the cone of $d\times d$ symmetric
positive-definite matrices, and let $G=\GL(d)$ act on $\mc X$ by
congruence:
$\alpha(\bs A)(\bs X)=\bs A\bs X\bs A^\top$ for
$\bs A\in\GL(d)$ and $\bs X\in\mathbb S_{\succ}^d$.
The ambient Hilbert space is
$\mc H=\Sym(d)$, equipped with the Frobenius inner product
$\langle\bs X,\bs Y\rangle_{\mc H}=\Tr(\bs X\bs Y)$.
Let $\mathcal L_{\Sym(d)}$ denote the Euclidean volume measure on the
finite-dimensional Hilbert space $\Sym(d)$ induced by the Frobenius
inner product. Equivalently, it is the
$d(d+1)/2$-dimensional Hausdorff measure induced by the Frobenius
norm. All densities on $\mathbb S_{\succ}^d$ in this subsection are
taken with respect to the restriction
$\mathcal L_{\Sym(d)}|_{\mathbb S_{\succ}^d}$.

\begin{definition}[Spectrally invariant distribution]
\label{def:spectrally-invariant}
A probability measure
$\rho\in\mc P(\mathbb S_{\succ}^d)$ is
\emph{spectrally invariant} if
$\alpha(\bs Q)_\#\rho=\rho$ for every
$\bs Q\in\mc O(d)$.
\end{definition}

Orthogonal congruence preserves the unordered eigenvalues of a
symmetric matrix and acts orthogonally on $\Sym(d)$ with respect to
the Frobenius inner product. Consequently, any density on
$\mathbb S_{\succ}^d$ that depends on $\bs X$ only through its
unordered eigenvalues defines a spectrally invariant distribution.
For $\bs\Sigma\in\mathbb S_{\succ}^d$, write
$\mu_{\bs\Sigma}
=
\alpha(\bs\Sigma^{1/2})_\#\rho$.



Define $b:G\to\mc H$ by $b\equiv0$ and
$\pi:G\to\GL(\mc H)$ by
$\pi(\bs A)(\bs X)=\bs A\bs X\bs A^\top$.
For $\bs H\in\mathfrak{gl}(d)$, define
$\mc D_{\bs H}$ by
$\mc D_{\bs H}(\bs X)=\bs H\bs X+\bs X\bs H^\top$,
and set $L=\pi(G)$.

\begin{lemma}
\label{lem:congruence-cartan-structure}
Suppose that
$\rho\in\mc P_2(\mathbb S_{\succ}^d)$ is spectrally invariant.
Then the following statements hold.
\begin{enumerate}[label=\textnormal{(\roman*)}]
    \item
    The map $\pi$ is a smooth group homomorphism, and
    $\alpha(\bs A)(\bs X)
    =
    b(\bs A)+\pi(\bs A)(\bs X)$
    for every $\bs A\in G$ and
    $\bs X\in\mathbb S_{\succ}^d$.
    Hence Assumption~\ref{ass:affine-induced-pi} holds.

    \item
    The subgroup $L$ is a closed Lie subgroup of
    $\GL(\Sym(d))$. The map
    $\theta:L\to L$ defined by
    $\theta(\ell)=(\ell^*)^{-1}$
    is a Lie group involution. Moreover,
    $\Lie(L)
    =
    \{\mc D_{\bs H}:
    \bs H\in\mathfrak{gl}(d)\}$,
    $L^\theta=\pi(\mc O(d))$,
    $\mathfrak k
    =
    \{\mc D_{\bs K}:
    \bs K^\top=-\bs K\}$,
    and
    $\mathfrak p
    =
    \{\mc D_{\bs S}:
    \bs S\in\Sym(d)\}$.
    For every $\bs S\in\Sym(d)$,
    $\exp(\mc D_{\bs S})
    =
    \pi(\exp(\bs S))$,
    and $(L,\theta)$ admits a global Cartan decomposition in
    the sense of \Cref{def:cartan-decomposition}.

    \item
    Every element of $\mathfrak p$ is self-adjoint on
    $\Sym(d)$ with respect to the Frobenius inner product, and
    $L^\theta
    =
    \pi(\Stab_G(\rho))$.
\end{enumerate}
Consequently, Assumption~\ref{ass:affine-induced-pi} and all three
structural hypotheses of
\Cref{prop:affine-cartan-criterion} are satisfied.
\end{lemma}

\begin{lemma}[Moment structure of a spectrally invariant law]
\label{lem:spectral-moment-structure}
Let \(d\geq2\), and let \(\rho\in\mc P(\mathbb S_{\succ}^d)\) be spectrally invariant with finite second moment. Define \(\tau_\rho=d^{-1}\EE_{\bs X\sim\rho}[\Tr(\bs X)]\). Then \(\EE_{\bs X\sim\rho}[\bs X]=\tau_\rho\bs I_d\) and \(\tau_\rho>0\). Moreover, there exist unique constants \(a_\rho,b_\rho\in\R\) such that
\begin{equation}
    \EE_{\bs X\sim\rho}[X_{ij}X_{kl}]
    =a_\rho\delta_{ij}\delta_{kl}
    +b_\rho\bigl(\delta_{ik}\delta_{jl}+\delta_{il}\delta_{jk}\bigr)
    \qquad
    \text{for all }i,j,k,l\in\{1,\ldots,d\}.
    \label{eq:second-moment-ansatz}
\end{equation}
Consequently, for every \(\bs U,\bs V\in\R^{d\times d}\),
\begin{equation}
    \EE_{\bs X\sim\rho}[\Tr(\bs U\bs X\bs V\bs X)]
    =a_\rho\Tr(\bs U\bs V)
    +b_\rho\bigl(\Tr(\bs U\bs V^\top)+\Tr(\bs U)\Tr(\bs V)\bigr).
    \label{eq:master-trace}
\end{equation}
The constants are determined by the two spectral moments through
\begin{equation}
    b_\rho
    =\frac{\EE_{\bs X\sim\rho}[\Tr(\bs X^2)]-d^{-1}\EE_{\bs X\sim\rho}[(\Tr\bs X)^2]}{d^2+d-2},
    \qquad
    a_\rho
    =\frac{\EE_{\bs X\sim\rho}[(\Tr\bs X)^2]-2db_\rho}{d^2}.
    \label{eq:alpha-beta-from-traces}
\end{equation}
Finally, for \(\bs\Sigma\in\mathbb S_{\succ}^d\), the mean \(\bar{\bs M}_{\bs\Sigma}\) and covariance operator \(\mathsf C_{\bs\Sigma}\) of \(\mu_{\bs\Sigma}=\alpha(\bs\Sigma^{1/2})_\#\rho\) satisfy
\begin{equation}
    \bar{\bs M}_{\bs\Sigma}=\tau_\rho\bs\Sigma,
    \qquad
    \mathsf C_{\bs\Sigma}(\bs H)
    =2b_\rho\bs\Sigma\bs H\bs\Sigma
    +(a_\rho-\tau_\rho^2)\Tr(\bs\Sigma\bs H)\bs\Sigma,
    \qquad \bs H\in\Sym(d).
    \label{eq:spectral-mean-covariance}
\end{equation}
\end{lemma}
\begin{theorem}[Closed-form transport on spectral congruence orbits]
\label{thm:congruence-template}
Let \(d\geq2\), and let
\(\rho\in\mc P(\mathbb S_{\succ}^d)\) be absolutely continuous with
respect to
\(\mathcal L_{\Sym(d)}|_{\mathbb S_{\succ}^d}\), spectrally invariant,
and have a finite second moment with respect to the Frobenius norm.
Let $
    c(\bs X,\bs Y)=\|\bs X-\bs Y\|_{\mathrm F}^2$.
For \(\bs\Sigma_i\in\mathbb S_{\succ}^d\), set $
    g_i=\bs\Sigma_i^{1/2},~
    \mu_i=\alpha(g_i)_\#\rho$,
$    i=0,1$.
Let \(a_\rho\) and \(b_\rho\) be given by
\eqref{eq:alpha-beta-from-traces}, and define $
    \bs C
    =
    \bigl(
        \bs\Sigma_0^{1/2}
        \bs\Sigma_1
        \bs\Sigma_0^{1/2}
    \bigr)^{1/2}$, $
    \bs A\opt
    =
    \bs\Sigma_0^{-1/2}
    \bs C
    \bs\Sigma_0^{-1/2},~
    \bs Q\opt
    =
    \bs\Sigma_1^{-1/2}
    \bs A\opt
    \bs\Sigma_0^{1/2},~
    h\opt=\bs Q\opt$.
Then $
    \bs Q\opt\in\mc O(d)$ and $
    h\opt\in\Stab_{\GL(d)}(\rho)$.
Moreover, all the conclusions of
\Cref{prop:affine-cartan-criterion} hold with \(h=h\opt\).
The corresponding optimal Monge map is 
\begin{equation}
    T\opt(\bs X)
    =
    \bs A\opt\bs X\bs A\opt,\quad
    \bs X\in\mathbb S_{\succ}^d,
    \label{eq:congruence-map}
    \end{equation}
and
\begin{equation}
\begin{aligned}
    \mathds K_c(\mu_0,\mu_1)
    &=
    (a_\rho+b_\rho)
    \|\bs\Sigma_0-\bs\Sigma_1\|_{\mathrm F}^2
    +
    b_\rho
    \left(
        (\Tr\bs\Sigma_0)^2
        +
        (\Tr\bs\Sigma_1)^2
        -
        2(\Tr\bs C)^2
    \right).
\end{aligned}
\label{eq:congruence-value}
\end{equation}
\end{theorem}

We conclude by presenting three families of spectrally invariant reference measures, each yielding a closed-form optimal transport value through
\eqref{eq:congruence-value}. 
\Cref{tab:congruence-families} summarizes these families and the corresponding coefficients $a_\rho$ and $b_\rho$.
We relegate the derivations to \Cref{app:cong-examples-additional}.

\begin{table}[ht]
\centering
\footnotesize
\renewcommand{\arraystretch}{1.6}
\caption{
Congruence families generated by spectrally invariant reference laws
on $\mathbb S_\succ^d$, $d\geq 2$, their densities, reference laws, and
second-moment constants entering the transport
value~\eqref{eq:congruence-value}.}
\label{tab:congruence-families}
\smallskip
\color{black}
\begin{tabular}{@{}llccc@{}}
\toprule
\textbf{Family}
  & \textbf{Density} ($\propto$)
  & ${\rho}$
  & $a_\rho$
  & $b_\rho$ \\
\midrule
\multicolumn{5}{@{}l}{\textit{\textbf{Wishart}}
  ($\bs\Sigma\in\mathbb S_\succ^d$, $p> d-1$)} \\[2pt]
$\mc W_d(\bs\Sigma,p)$
  & $\det(\bs X)^{(p-d-1)/2}
     \exp({-\frac{1}{2}\Tr(\bs\Sigma^{-1}\bs X)})$
  & $\mc W_d(\bs I_d,p)$
  & $p^2$
  & $p$ \\
\hline
\multicolumn{5}{@{}l}{\textit{\textbf{Inverse-Wishart}}
  ($\bs\Psi\in\mathbb S_\succ^d$, $p>d{+}3$,
  $m=p{-}d$)} \\[2pt]
$\mc{IW}_d(\bs\Psi,p)$
  & $\det(\bs X)^{-(p+d+1)/2}
     \exp({-\frac{1}{2}\Tr(\bs\Psi\bs X^{-1})})$
  & $\mc{IW}_d(\bs I_d,p)$
  & {\tiny{$\dfrac{m{-}2}{m(m{-}1)(m{-}3)}$}}
  & {\tiny$\dfrac{1}{m(m{-}1)(m{-}3)}$} \\
\hline
\multicolumn{5}{@{}l}{\textit{\textbf{Matrix beta type~II}}
  ($\bs\Sigma\in\mathbb S_\succ^d$,
  $q_1>(d{-}1)/2$, $2q_2>d{+}3$,
  $q=2q_2{-}d$)} \\[2pt]
$\mc{MB}_d^{\mathrm{II}}(q_1,q_2,\bs\Sigma)$
  & $\det(\bs X)^{q_1-(d+1)/2}
     \det(\bs\Sigma{+}\bs X)^{-(q_1+q_2)}$
  & $\mc{MB}_d^{\mathrm{II}}(q_1,q_2,\bs I_d)$
  & {\tiny$\dfrac{4q_1(q_1(q{-}2){+}1)}{q(q{-}1)(q{-}3)}$}
  & {\tiny$\dfrac{2q_1(2q_1{+}q{-}1)}{q(q{-}1)(q{-}3)}$} \\
\bottomrule
\end{tabular}
\end{table}

{\color{black}\begin{remark}[Tightness of {\citep[Eq. (3)]{Gelbrich1990}} beyond elliptical families]
\label{rem:gelbrich-tightness-congruence}
The expression in \eqref{eq:affine-cartan-value} coincides with the
lower bound for quadratic optimal transport established in
\citep[Theorem~2.1]{Gelbrich1990}. A classical equality class for this
bound is given by elliptically contoured distributions sharing a
common characteristic generator.
\Cref{thm:congruence-template} shows that the same bound is attained
on congruence orbits of absolutely continuous, spectrally invariant
reference laws with finite second moment.
The Wishart, inverse-Wishart, and matrix beta type~II families in
\Cref{tab:congruence-families} are not elliptically contoured when
viewed as probability distributions on \(\Sym(d)\). Indeed, the Wishart, inverse-Wishart, and matrix beta type~II densities are
strictly positive throughout $\mathbb S_{\succ}^d$. Consequently,
when these laws are viewed as probability measures on $\Sym(d)$,
their topological support is $\mathbb S_{\succeq}^d$.
This support is not centrally symmetric about any
$\bs M\in\Sym(d)$: for all sufficiently large $t>0$,
$t\bs I_d\in\mathbb S_{\succeq}^d$, whereas
$2\bs M-t\bs I_d\notin\mathbb S_{\succeq}^d$.
By contrast, an elliptically contoured distribution is centrally
symmetric about its location and therefore has centrally symmetric
support. Hence these matrix-valued distributions are not elliptically
contoured when viewed as probability measures on $\Sym(d)$. Thus, \Cref{thm:congruence-template} identifies a
matrix-valued, non-elliptical equality class for the lower bound
established in \citep[Theorem~2.1]{Gelbrich1990}.
\end{remark}
}

{\color{black}
\subsection{Multivariate normal sinh-arcsinh distributions}
\label{sec:sinh-arcsinh}

\citep[Section~7]{JonesPewsey2009} introduced multivariate normal
sinh-arcsinh distributions by transforming the marginals of a standardized,
correlated multivariate normal vector. To describe this family, for
$\epsilon\in\R$ and $\delta>0$, define
$S_{\epsilon,\delta}(x)=\sinh(\epsilon+\delta\operatorname{arsinh}(x))$. Here $\epsilon$ is a skewness parameter and $\delta$ is a tailweight
parameter. A random variable $X$ has the canonical univariate normal
sinh-arcsinh distribution with parameters $(\epsilon,\delta)$ if
$S_{\epsilon,\delta}(X)$ is standard normal. Equivalently, if
$Z\sim\mc N(0,1)$, then
$S_{\epsilon,\delta}^{-1}(Z)
=\sinh((\operatorname{arsinh}(Z)-\epsilon)/\delta)$ has the canonical normal sinh-arcsinh distribution with parameters
$(\epsilon,\delta)$.
For the multivariate construction, let
$\bs\epsilon=(\epsilon_1,\ldots,\epsilon_d)\in\R^d$,
$\bs\delta=(\delta_1,\ldots,\delta_d)\in\R_{>0}^d$, and let
$\bs\Omega\in\mathbb S_{\succ}^d$ be a correlation matrix. If
$\bs Z=(Z_1,\ldots,Z_d)\sim\mc N(\bs 0_d,\bs\Omega)$, define
$\bs X=(X_1,\ldots,X_d)$ by
$X_j=S_{\epsilon_j,\delta_j}^{-1}(Z_j)$ for $j=1,\ldots,d$.
We denote the law of $\bs X$ by
$\mc{NSAS}_d(\bs\epsilon,\bs\delta;\bs\Omega)$.
Its $j$th marginal is the canonical univariate normal sinh-arcsinh law
with parameters $(\epsilon_j,\delta_j)$, while its dependence is inherited
from the Gaussian copula with correlation matrix $\bs\Omega$.

In this subsection, we fix the tailweight vector at
$\bs\delta=\bs 1_d$, fix a positive-definite correlation matrix
$\bs\Omega$, and allow the skewness vector $\bs\epsilon$ to
vary. Let $\rho=\mc N(\bs 0_d,\bs\Omega)$. For
$\bs\epsilon\in\R^d$, define $T_{\bs\epsilon}:\R^d\to\R^d$
coordinatewise by
$T_{\bs\epsilon}(\bs z)_j
=\sinh(\operatorname{arsinh}(z_j)-\epsilon_j)$, and set
$\mu_{\bs\epsilon}=(T_{\bs\epsilon})_\#\rho$.
Thus
$\mu_{\bs\epsilon}
=\mc{NSAS}_d(\bs\epsilon,\bs 1_d;\bs\Omega)$.
For convenience, define
$\bs u_{\bs\epsilon}:\R^d\to\R^d$ by
$u_{\bs\epsilon,j}(\bs x)
=\sinh(\epsilon_j+\operatorname{arsinh}(x_j))$.
Then $\bs u_{\bs\epsilon}=T_{\bs\epsilon}^{-1}$, and
$\mu_{\bs\epsilon}$ has density
\begin{equation}
\begin{aligned}
r_{\bs\Omega,\bs\epsilon}(\bs x)
=
\frac{1}{(2\pi)^{d/2}\det(\bs\Omega)^{1/2}}
\exp\!\left(
-\frac12
\bs u_{\bs\epsilon}(\bs x)^\top
\bs\Omega^{-1}
\bs u_{\bs\epsilon}(\bs x)
\right)\!\times\!
\prod_{j=1}^d
\frac{\cosh(\epsilon_j+\operatorname{arsinh}(x_j))}
{\sqrt{1+x_j^2}},
~
\bs x\in\R^d.
\end{aligned}
\label{eq:dependent-sinh-arcsinh-density}
\end{equation}
Equation~\eqref{eq:dependent-sinh-arcsinh-density} is the multivariate
normal sinh-arcsinh density of \citep[Section~7]{JonesPewsey2009},
specialized to unit tailweight parameters. Since each marginal
transformation $T_{\epsilon_j}$ is a strictly increasing diffeomorphism,
all measures $\mu_{\bs\epsilon}$ have the Gaussian copula with correlation
matrix $\bs\Omega$.


Let $\mathfrak B_d\subseteq\mc O(d)$ be the hyperoctahedral group of
signed permutation matrices, and define
$\mathfrak B_{\bs\Omega}
=\{\bs Q\in\mathfrak B_d:
\bs Q\bs\Omega\bs Q^\top=\bs\Omega\}$.
Consider the semidirect product
$G_{\bs\Omega}=\R^d\rtimes\mathfrak B_{\bs\Omega}$ with multiplication
$(\bs\epsilon,\bs Q)(\bs\eta,\bs P)
=(\bs\epsilon+\bs Q\bs\eta,\bs Q\bs P)$.
Its identity is $e=(\bs 0_d,\bs I_d)$, and
$(\bs\epsilon,\bs Q)^{-1}
=(-\bs Q^{-1}\bs\epsilon,\bs Q^{-1})$.
Let $\mc X=\mc H=\R^d$, equipped with the standard inner product, and let
$c(\bs x,\bs y)=\|\bs x-\bs y\|^2$. Define $
\alpha((\bs\epsilon,\bs Q))(\bs x)
=
\sinh\bigl(
\bs Q\operatorname{arsinh}(\bs x)-\bs\epsilon
\bigr)$,
$(\bs\epsilon,\bs Q)\in G_{\bs\Omega}$,~
$\bs x\in\R^d$,
where $\sinh$ and $\operatorname{arsinh}$ are applied componentwise.

Define the involution $\theta:G_{\bs\Omega}\to G_{\bs\Omega}$ by
$\theta(\bs\epsilon,\bs Q)=(-\bs\epsilon,\bs Q)$.
Under the identification $\Lie(G_{\bs\Omega})\simeq\R^d$, the exponential
satisfies $\exp(\bs S)=(\bs S,\bs I_d)$; the Cartan structure is verified
in \Cref{lem:dependent-sinh-arcsinh-cartan}.
\begin{corollary}[Closed-form transport between multivariate normal
sinh-arcsinh laws]
\label{cor:dependent-sinh-arcsinh-transport}
Let $g_i=(\bs\epsilon_i,\bs Q_i)\in G_{\bs\Omega}$ and
$\mu_i=\alpha(g_i)_\#\rho$ for $i=0,1$. Set
$\bs\Delta=\bs\epsilon_1-\bs\epsilon_0$ and
$h\opt=(\bs 0_d,\bs Q_1^{-1}\bs Q_0)$. Then
$h\opt\in G_{\bs\Omega}^{\theta}
=\Stab_{G_{\bs\Omega}}(\rho)$ and
$g_1h\opt g_0^{-1}
=(\bs\Delta,\bs I_d)$.
Moreover, all the conclusions of
\Cref{thm:nonlinear-cartan-exactness} hold with $h=h\opt$.
The corresponding unique optimal Monge map from $\mu_0$ to $\mu_1$,
up to $\mu_0$-almost-everywhere equality, is
$T\opt(\bs x)
=
\sinh(\operatorname{arsinh}(\bs x)-\bs\Delta)$.
For $j=1,\ldots,d$, set
$\Delta C_j
=
\cosh(\epsilon_{1,j})-\cosh(\epsilon_{0,j})$ and
$\Delta S_j
=
\sinh(\epsilon_{1,j})-\sinh(\epsilon_{0,j})$.
Then
$\mathds K_c(\mu_0,\mu_1)
=
\sum_{j=1}^d
\bigl((\Delta C_j)^2+2(\Delta S_j)^2\bigr)$.
\end{corollary}

For every $\epsilon\in\R$, the map
$t_\epsilon(z)=\sinh(\operatorname{arsinh}(z)-\epsilon)$ is a strictly
increasing diffeomorphism of $\R$. Consequently, each measure in the
orbit is obtained from $\mc N(\bs 0_d,\bs\Omega)$ by coordinatewise
strictly increasing transformations, and all the orbit measures share
the Gaussian copula with correlation matrix $\bs\Omega$.
Let $F_{i,j}$ denote the distribution function of the $j$th marginal
of $\mu_i$, and let $\Phi$ denote the standard normal distribution
function. Since the $j$th marginal of $\mu_i$ is the push-forward of a
standard normal random variable by $t_{\epsilon_{i,j}}$, its quantile
function is
$F_{i,j}^{-1}=t_{\epsilon_{i,j}}\circ\Phi^{-1}$. It follows that
$F_{1,j}^{-1}\circ F_{0,j}
=t_{\epsilon_{1,j}-\epsilon_{0,j}}$.
Hence the coordinatewise quantile map associated with the common
copula is precisely the optimal map identified in
\Cref{cor:dependent-sinh-arcsinh-transport}.
Since the squared Euclidean cost is the sum of the coordinatewise
quadratic costs, the optimality of the corresponding common-copula
coupling also follows from
\citep[Proposition~1.1]{AlfonsiJourdain2014}. Thus,
\Cref{cor:dependent-sinh-arcsinh-transport} recovers the
common-copula optimal transport for this family through the nonlinear
Cartan criterion.

}

%% file: appendix.tex
\appendix

\section{Beyond Lie group orbits: one-dimensional distributions}
\label{app:1d-infinite-dimensional-case}
{\color{black}
The preceding examples fall within the finite-dimensional
Lie group framework developed in the main paper.  We now record the
classical one-dimensional case as a complementary example.
This example still has an orbit and stabilizer structure,
but the acting transformation group is infinite-dimensional
and is not a finite-dimensional Lie group.  The results in this section should
therefore be read as an orbit-based reformulation of the
classical monotone rearrangement formula
\citep{ref:cuestaalbertos1993optimal}, rather than as an
application of the finite-dimensional theory of the main paper.}

For an absolutely continuous probability measure
$\mu\ll\mc L^1$ on $\R$, let $F_\mu$ and $r_\mu$ denote its cumulative
distribution function and density, respectively.
We consider the class  
\[
   \mathcal P^{+}
   =\left\{\mu\in\mathcal P(\R): \mu \ll \mc L^1,~F_\mu\in\mathcal C^1,~
     F_\mu'(x) > 0 ~\forall x \in \R \right\}
\]
Fix a reference probability measure $\rho$ with a {smooth,
strictly positive} density
$r\in \mc C^{\infty}(\R)$; a concrete example is the logistic density $
   r(x)=\tfrac14 \mathrm{sech}^{2} \bigl(x/2\bigr)$,
and its cumulative distribution function $
   F_\rho(x)=\int_{-\infty}^{x} r(s) \diff s$
satisfies $F_\rho' (x)=r(x)>0$ for all $x$, so
$F_\rho:\R\to(0,1)$ is a $\mc C^{\infty}$, strictly increasing bijection.
By the inverse-function theorem,
$F_\rho$ is in fact a $\mc C^{\infty}$ {diffeomorphism}
with smooth inverse $F_\rho^{-1} :(0,1)\to\R$.

Let $\mc G^+$ be the group of orientation-preserving $\mc C^1$
diffeomorphisms of $\R$: $
  \mc G^+
  =\{
  g\in\mc C^1(\R;\R):
  g'(x)>0~\forall x\in\R,
  ~g \text{ is bijective},
 ~ g^{-1}\in\mc C^1(\R;\R)\}$,
equipped with composition.

{\color{black}We regard $\mc G^+$ only as a transformation group. In particular, this
one-dimensional example is not an application of the finite-dimensional
Lie-group framework developed in the main paper. Nevertheless, it has the same
orbit-stabilizer structure at the level of transformations. Namely, if $
\mu_i=(g_i)_\#\rho$, $g_i\in\mc G^+$, $i=0,1$,
then every map of the form $
  g_1\circ h\circ g_0^{-1}$, $h\in\Stab_{\mc G^+}(\rho)$,
transports $\mu_0$ to $\mu_1$. Thus we introduce the
one-dimensional analogue of the orbit-reduced problem \eqref{eq:transport-wihtin-orbits}. For this
subsection only, we define
\begin{equation}
\label{eq:orbit-1d}
  \mathds J^{1}_c(g_0,g_1)
  =
  \inf_{h\in\Stab_{\mc G^+}(\rho)}
  \int_\R c\bigl(g_0(x),(g_1\circ h)(x)\bigr)\,\diff\rho(x),
\end{equation}
where $
  \Stab_{\mc G^+}(\rho)
  =
  \{h\in\mc G^+:h_\#\rho=\rho\}$.}

\begin{lemma}\label{lem:diff-orbit-1d}
   For any $\mu \in \mathcal P^+$, $(F_{\mu}^{-1} \circ F_{\rho})_\#\rho = \mu$ and $\mc G^+_{\#}\rho=\mathcal P^{+}.$
\end{lemma}
\begin{proof}
First, we will show that $\mc G^+{_{\#}}\rho \subseteq \mathcal P^{+}$. 
Take $g\in\mc G^+$ and set $\mu=g_{\#}\rho$.
Since $g$ is $\mc C^{1}$, strictly increasing and surjective,
its inverse is likewise $\mathcal C^1$ and strictly increasing. Additionally, by the inverse function theorem $(g^{-1})'(t)=1/g'(g^{-1}(t))>0$.  A direct
computation shows that the cumulative distribution function of $\mu$ is $
    F_\mu(x)= \int_{-\infty}^x r(g^{-1}(t)) (g^{-1})'(t) \diff t.$
By applying the change of variables in the form of $s=g^{-1}(t)$ (hence
$\mathrm dt=g'(s) \mathrm ds$), we have
$F_\mu(x)=\int_{-\infty}^{g^{-1}(x)} r(s)\mathrm ds
            = F_{\rho} (g^{-1}(x)).$
Since \(F_\rho:\mathbb R\to(0,1)\) and \(g^{-1}:\mathbb R\to\mathbb R\)
are both \(\mc C^1\) and strictly increasing, their composition
\(F_\rho\circ g^{-1}\) is \(\mc C^1\) and strictly increasing on
\(\mathbb R\).
Moreover, 
\[
F_\mu'(x)=r\bigl(g^{-1}(x)\bigr)(g^{-1})'(x)
          =\frac{r(g^{-1}(x))}{g'(g^{-1}(x))}>0.
\]
Thus $F_\mu$ satisfies the defining properties of~$\mc P^{+}$, so
$\mu\in\mc P^{+}$.

Next, we will show that $\mc P^{+}\subseteq\mc G^+_{\#}\rho$.
Take an arbitrary $\mu\in\mc P^{+}$ and write $F_\mu$ for its
cumulative distribution function and $r_\mu$ for its density.  Because $F_\mu$ is strictly
increasing, the inverse
$F_\mu^{-1}:(0,1)\to\R$ is well defined and belongs to
$\mc C^1((0,1))$ and strictly increasing.
Define $
      g = F_\mu^{-1}\circ F_\rho :\R\to\R $,
 which also belongs to $\mc C^1(\R)$. 
Hence,
$g\in\mc G^+$.
For any $\mc A \in \mc B(\R)$,
\begin{align*}
   g_{\#}\rho(\mc A)
      &=\rho\left(g^{-1}(\mc A)\right)\\
      &=\rho\left( \left\{x\in\R:F_\mu^{-1}(F_\rho(x))\in \mc A\right\}\right)\\
      &=\rho\left( \left\{x\in \R :F_\rho(x)\in F_\mu(\mc A)\right\}\right)\\
      &= (F_\rho)_\#\rho(F_\mu(\mc A))=\lambda\left(F_\mu(\mc A)\right)=\mu(\mc A),
\end{align*}
where $\lambda$ is the Lebesgue measure on the unit interval $(0,1)$. The first equality follows by the definition of the push-forward operator, the second equality by construction of $g$, the third equality follows because $F_\mu^{-1}$ is the inverse of the strictly increasing function $F_\mu$. The fourth equality follows by the definition of the push-forward operator, the fifth equality follows because by the probability-integral transform $(F_\rho)_{\#}\rho = \lambda$. The last equality follows because 
 $(F_\mu)_{\#}\mu=\lambda$, which equivalently means that for every $\mc B'\in \mc B((0,1))$, we have
      $\lambda(\mc B')=\mu (F_\mu^{-1}(\mc B'))$.
      Taking $\mc B'=F_\mu(\mathcal A)$ gives
      $\lambda (F_\mu(\mathcal A))=\mu(\mathcal A)$.
Hence, we may conclude that $g_{\#}\rho = \mu$, which implies that $\mc P^+ \subseteq \mc G^+_\#\rho$. This observation completes our proof. 
\end{proof}

\begin{lemma}\label{lem:diff-stab-1d}
The stabilizer of $\rho$ under $\mc G^+$ is $\Stab_{\mc G^+}(\rho)=\{\mathrm{id}_{\R}\}.
$
\end{lemma}
\begin{proof}
Fix $g\in\mc G^+$ with $g_{\#}\rho=\rho$.
For every $x\in\R$ we have
\[
   F_\rho(x)
      = 
     \rho\left((-\infty,x]\right)
      = 
     \rho \left(g^{-1}((-\infty,x])\right)
      = 
     \rho\left((-\infty,g^{-1}(x)]\right)
      = 
     F_\rho\left(g^{-1}(x)\right).
\]
Thus $
    F_\rho(x)=F_\rho(g^{-1}(x))$ for all $ x \in \R$. 
Since $r>0$, the function $F_\rho$ is {strictly increasing} on $\R$.
Hence $F_\rho$ is injective, and the identity above implies
$g^{-1}(x)=x$ for every $x\in\R$; equivalently, $g(x)=x$.  Therefore $g=\mathrm{id}_{\R}$, completing the proof.
\end{proof}

\begin{lemma}
Suppose that $g_i = F_{\mu_i}^{-1}\circ F_\rho$, $i= 0,1$. Then, \eqref{eq:orbit-1d} induced by the cost function $c(x,y) = (x-y)^2$ is solved by $h\opt= \mathrm{id}_{\R}$ and admits the following closed form expression
     \begin{equation}
        \mathds J^{1}_c(g_0, g_1) = \int_0^{1} (F_{\mu_1}^{-1}(t) - F_{\mu_0}^{-1}(t))^2 \diff t. 
    \end{equation}
    \label{lem:1d-orbits-closed-form}
\end{lemma}
\begin{proof}
   By \Cref{lem:diff-stab-1d}, the stabilizer group is a singleton, and thus $h\opt = \mathrm{id}_{\R}$ solves \eqref{eq:orbit-1d}. Next, we evaluate $\mathds J^1_c$:
   \begin{align*}
       \mathds J^1_c(g_0, g_1) 
       &= \int_{\R} (g_0( x) - g_1( x))^2 \diff \rho(x)\\
       &=\int_{\R} (F_{\mu_1}^{-1}(F_\rho(x)) - F_{\mu_0}^{-1}(F_\rho(x)))^2 r(x)\diff x\\
       &=\int_0^1 (F_{\mu_1}^{-1}(t) - F_{\mu_0}^{-1}(t))^2\diff t,
   \end{align*}
   where the last equality follows by the change-of-variables formula $t \gets F_\rho(x)$, and thus $\diff t = r(x)\diff x $. This observation completes our proof.
\end{proof}
\begin{proposition}
\label{prop:one-d-closed-form}
    Suppose that $\mu_0, \mu_1 \in \mc P^+$ have finite second moments and $c(x,y)= (x-y)^2$. Then, we have
    \begin{enumerate}[label=\textnormal{(\roman*)}]
        \item $\mathds M_c(\mu_0, \mu_1) = \mathds J^1_c(g_0,g_1)$, and $T\opt = F_{\mu_1}^{-1}\circ F_{\mu_0}$ solves \eqref{eq:monge-problem},
        \item $\mathds K_c(\mu_0, \mu_1) = \mathds J^1_c(g_0,g_1)$, and $(\mathrm{id}_{\R}, T\opt)_{\#}\mu_0$ solves
    \eqref{eq:kantorovich-problem},
    \end{enumerate}
    where $g_i = F_{\mu_i}^{-1} \circ F_\rho \in \mc G^+$ for $i=0,1$.
\end{proposition}
\begin{proof}
\label{ref:prop-closed-form-1d}
By \Cref{lem:diff-orbit-1d}, we have $(g_i)_\#\rho = \mu_i$ for $i=0,1$; hence $\mu_0, \mu_1 \in \mc G^+_\#\rho$. Define $T\opt = g_1 \circ h\opt \circ g_0^{-1}$, where $h\opt$ is the map defined in \Cref{lem:1d-orbits-closed-form}. Consequently, we have $T\opt(x) = F_{\mu_1}^{-1}(F_{\mu_0}(x))$.

In what follows, we define $\varphi(x) = 2\int_{x_0}^x T\opt(s) \diff s- x^2$ for some $x_0 \in \R$, and we compute its first and second order derivatives  
\[\varphi'(x) = 2 T\opt(x) - 2x~\text{and}~\varphi''(x) = 2(T\opt)'(x) - 2.\]
As $T\opt$ is increasing, $(T\opt)'(x) \geq 0 $ for all $x\in \R$, implying $\varphi''(x) \geq -2$. Consequently, by \cite[Example 13.6]{ref:villani2008optimal}, $\varphi$ is $c$-convex. Next, we evaluate \eqref{eq:T-optimal-c-convex-equation}:
\[2T\opt(x) - 2x + 2(x- T\opt(x) ) = 0.\]
The quadratic cost satisfies Assumption~\ref{ass:c-regularity}
directly. Since $\mu_0\ll\mc L^1$,
\Cref{lem:quadratic-c-convex}(iii) gives
Assumption~\ref{ass:c-mu-0}, while
\Cref{lem:quadratic-c-convex}(iv) gives
Assumption~\ref{ass:H-infty}. The finite-second-moment assumption
implies $\mathds K_c(\mu_0,\mu_1)<\infty$. Therefore, by \Cref{thm:monge-kantorovich}(v) $T\opt$ solves \eqref{eq:monge-problem} and the optimal Kantorovich plan is concentrated on the graph of
\(T^\star\).

If \(X\sim\mu_0\), then \(U=F_{\mu_0}(X)\sim\mc U([0,1])\). Hence
\[
\begin{aligned}
  \int_\R |x-T^\star(x)|^2\,\diff\mu_0(x)
  &=
  \int_0^1
  \left(
  F_{\mu_0}^{-1}(t)-F_{\mu_1}^{-1}(t)
  \right)^2\,\diff t.
\end{aligned}
\]
The equality with \(\mathds J^{1}_c(g_0,g_1)\) follows from
Lemma~\ref{lem:1d-orbits-closed-form}.
\end{proof}
\Cref{prop:one-d-closed-form} coincides with \cite[Corollary 2.7]{ref:cuestaalbertos1993optimal} when restricted to the quadratic cost.

\section{Additional example for \Cref{sec:examples}}
\subsection{The diagonal scaling mechanism on \texorpdfstring{$\mathbb R_{>0}^d$}{R[>0]^d}}
\label{subsec:diagonal-scaling}

Let the state space be $\mc X=\mathbb R_{>0}^d$, and let
$G_{\mathrm{diag}}=(\mathbb R_{>0}^d,\odot)$ be the componentwise
multiplicative group. Its identity element is $\bs 1_d$, and the inverse
of $\bs a\in G_{\mathrm{diag}}$ is
$\bs a^{-1}=(a_1^{-1},\ldots,a_d^{-1})$. The componentwise logarithm is
a Lie group isomorphism from $G_{\mathrm{diag}}$ onto
$(\mathbb R^d,+)$, so $G_{\mathrm{diag}}$ is a $d$-dimensional Lie
group.

The group acts smoothly on $\mc X$ by coordinatewise scaling,
$\alpha(\bs a)(\bs x)=\bs a\odot\bs x$ for
$\bs a,\bs x\in\mathbb R_{>0}^d$. Equivalently, writing
$D(\bs a)=\operatorname{diag}(a_1,\ldots,a_d)$, we have
$\alpha(\bs a)(\bs x)=D(\bs a)\bs x$. Thus the action is induced by
the linear representation on $\mc H=\mathbb R^d$, equipped with the
standard inner product, given by $b\equiv0$ and
$\pi(\bs a)=D(\bs a)$. Indeed,
$\pi(\bs a\odot\bs b)=D(\bs a\odot\bs b)
=D(\bs a)D(\bs b)=\pi(\bs a)\pi(\bs b)$, and
$\alpha(\bs a)(\bs x)=b(\bs a)+\pi(\bs a)\bs x$.
Hence Assumption~\ref{ass:affine-induced-pi} holds.

\paragraph{Cartan structure}
Set
$L=\pi(G_{\mathrm{diag}})
=\{D(\bs a):\bs a\in\mathbb R_{>0}^d\}$.
The map $\mathbb R^d\ni\bs s $ $\mapsto$
$\operatorname{diag}($ $e^{s_1},\ldots,e^{s_d})\in L$ is a Lie group
isomorphism, and hence $L$ is a Lie subgroup of
$\GL(\mathbb R^d)$. Define $\theta:L\to L$ by
$\theta(D(\bs a))=D(\bs a)^{-\top}=D(\bs a^{-1})$.
Then $\theta$ is a Lie group involution. Its fixed-point subgroup is
$L^\theta=\{\bs I_d\}$, because $a_i=a_i^{-1}$ and $a_i>0$ imply
$a_i=1$ for every $i$.

The Lie algebra of $L$ is the space of diagonal matrices,
$\mathfrak l
=\{\operatorname{diag}(s_1,\ldots,s_d):s_i\in\mathbb R\}$.
Since $\diff\theta_{\bs I_d}(\bs S)=-\bs S$ for every
$\bs S\in\mathfrak l$, the corresponding eigenspaces are
$\mathfrak k=\{0\}$ and $\mathfrak p=\mathfrak l$. The exponential map
restricts to a diffeomorphism from $\mathfrak p$ onto $L$, with
$\operatorname{diag}(s_1,\ldots,s_d)$ mapped to
$\operatorname{diag}(e^{s_1},\ldots,e^{s_d})$ and inverse given by the
componentwise logarithm. Consequently, the Cartan map
$L^\theta\times\mathfrak p\to L$,
$(\bs I_d,\bs S)\mapsto\exp(\bs S)$, is a diffeomorphism. Thus
$(L,\theta)$ admits a global Cartan decomposition in the sense of
\Cref{def:cartan-decomposition}. Moreover, every
$\bs S\in\mathfrak p$ is diagonal and hence self-adjoint on
$\mathbb R^d$.

\paragraph{Reference distribution}
Let $\rho=\bigotimes_{i=1}^d\mathrm{Exp}(1)$, {\color{black}where
$\mathrm{Exp}(\lambda)$ denotes the exponential distribution with rate
parameter $\lambda>0$}, so that $\rho$ has density
$r(\bs x)=\prod_{i=1}^d e^{-x_i}\mathbbm 1_{\{x_i>0\}}$ with respect
to Lebesgue measure on $\mathbb R^d$. The state space
$\mathbb R_{>0}^d$ is open and convex, while $\rho$ is absolutely
continuous and satisfies
$\mathbb E_{\bs Z\sim\rho}[\|\bs Z\|^2]=2d<\infty$.
Furthermore, since $L^\theta=\{\bs I_d\}$ and the identity element
$\bs 1_d$ belongs to $\Stab_{G_{\mathrm{diag}}}(\rho)$, we have
$L^\theta\subseteq\pi(\Stab_{G_{\mathrm{diag}}}(\rho))$.
Therefore, all the standing assumptions and structural hypotheses of
\Cref{prop:affine-cartan-criterion} are satisfied.

\begin{lemma}[Product exponential orbit and stabilizer]
\label{lem:prod-exp-orbit-stabilizer}
Let $\rho=\bigotimes_{i=1}^d\mathrm{Exp}(1)$. Then:
\begin{enumerate}[label=\textnormal{(\roman*)}]
  \item for every $\bs\beta\in\mathbb R_{>0}^d$,
  $\alpha(\bs\beta^{-1})_\#\rho
  =\bigotimes_{i=1}^d\mathrm{Exp}(\beta_i)$, and hence
  $G_{\mathrm{diag}\,\#}\rho
  =\{\bigotimes_{i=1}^d\mathrm{Exp}(\beta_i):
  \bs\beta\in\mathbb R_{>0}^d\}$;

  \item $\Stab_{G_{\mathrm{diag}}}(\rho)=\{\bs 1_d\}$.
\end{enumerate}
\end{lemma}

\begin{proof}
Let $\bs Z\sim\rho$ and set $\bs X=\bs a\odot\bs Z$. For each
coordinate, $X_i=a_iZ_i$. Since $Z_i\sim\mathrm{Exp}(1)$, the
density of $X_i$ is
$x_i\mapsto a_i^{-1}e^{-x_i/a_i}\mathbbm 1_{\{x_i>0\}}$.
Taking $\bs a=\bs\beta^{-1}$ gives
$X_i\sim\mathrm{Exp}(\beta_i)$ for every $i$, and independence
is preserved under coordinatewise transformations. This proves
\textnormal{(i)}.

For \textnormal{(ii)}, suppose that
$\bs a\in\Stab_{G_{\mathrm{diag}}}(\rho)$. Then, by definition,
$(\alpha(\bs a))_\#\rho=\rho$. {\color{black}Hence, if $\bs Z\sim\rho$, then
$\bs a\odot\bs Z$ also has distribution $\rho$.} In particular,
$a_iZ_i{\color{black}\stackrel{\mathrm d}{=}Z_i}$ for every coordinate $i$.
Taking expectations and using $\mathbb E[Z_i]=1$ gives
$a_i=1$ for every $i$. Hence $\bs a=\bs 1_d$, which proves the claim.
\end{proof}

\begin{corollary}[Optimal transport between products of exponential laws]
\label{cor:prod-exp-ot}
For $j=0,1$, let
$\mu_j = \bigotimes_{i=1}^d\mathrm{Exp}(\beta_{j,i})$, where
$\bs\beta_j=(\beta_{j,1},\ldots,\beta_{j,d})\in\mathbb R_{>0}^d$, and let $c(\bs x,\bs y)=\|\bs x-\bs y\|^2$. Set
$g_j=\bs\beta_j^{-1}$ and $h\opt=\bs 1_d$. Then $h\opt$ solves the
Lie group orbit transport problem, and all the conclusions of
\Cref{prop:affine-cartan-criterion} hold with $h=h\opt$. In particular:
\begin{enumerate}[label=\textnormal{(\alph*)}]
  \item the unique optimal Monge map from $\mu_0$ to $\mu_1$, up to
  $\mu_0$-almost-everywhere equality, is
  $T\opt(\bs x)
  =(\frac{\beta_{0,1}}{\beta_{1,1}}x_1,\ldots,
  \frac{\beta_{0,d}}{\beta_{1,d}}x_d)$;

  \item $(\mathrm{id}_{\mathbb R_{>0}^d},T\opt)_\#\mu_0$ is the
  unique optimal Kantorovich plan;

  \item the optimal Kantorovich value is
  $\mathds K_c(\mu_0,\mu_1)
  =2\sum_{i=1}^d
  (\beta_{0,i}^{-1}-\beta_{1,i}^{-1})^2$;

  \item $\mathds M_c(\mu_0,\mu_1)
  =\mathds K_c(\mu_0,\mu_1)
  =\mathds J_c(g_0,g_1)$.
\end{enumerate}
\end{corollary}

\begin{proof}
By \Cref{lem:prod-exp-orbit-stabilizer}(i),
$\mu_j=\alpha(g_j)_\#\rho$ with
$g_j=\bs\beta_j^{-1}\in G_{\mathrm{diag}}$ for $j=0,1$. The preceding
verification shows that all the hypotheses of
\Cref{prop:affine-cartan-criterion} are satisfied. By
\Cref{lem:prod-exp-orbit-stabilizer}(ii), the stabilizer is trivial,
so the stabilizing element furnished by that proposition must be
$h\opt=\bs 1_d$.

The corresponding orbit-induced map is
$T_{h\opt}=\alpha(g_1h\opt g_0^{-1})
=\alpha(\bs\beta_1^{-1}\odot\bs\beta_0)$, and therefore
$T_{h\opt}(\bs x)
=(\frac{\beta_{0,1}}{\beta_{1,1}}x_1,\ldots,
\frac{\beta_{0,d}}{\beta_{1,d}}x_d)=T\opt(\bs x)$.
Thus the optimality and uniqueness assertions, as well as the equality
of the Monge, Kantorovich, and Lie group orbit transport values, follow
from \Cref{prop:affine-cartan-criterion}.

It remains to evaluate the common value. Since $h\opt=\bs 1_d$ and
$\bs Z\sim\rho$, we have
$\mathds J_c(g_0,g_1)
=\mathbb E[\|g_0\odot\bs Z-g_1\odot\bs Z\|^2]
=\sum_{i=1}^d
(\beta_{0,i}^{-1}-\beta_{1,i}^{-1})^2\mathbb E[Z_i^2]$.
Because $Z_i\sim\mathrm{Exp}(1)$,
$\mathbb E[Z_i^2]=2$. Hence
$\mathds J_c(g_0,g_1)
=2\sum_{i=1}^d
(\beta_{0,i}^{-1}-\beta_{1,i}^{-1})^2$, which proves the stated
formula.
\end{proof}

\begin{remark}
In this example the Cartan decomposition is degenerate: the fixed-point
subgroup $L^\theta$ is trivial. Thus there is no nontrivial stabilizer
optimization or rotational component to absorb. The example nevertheless
fits the same structural template as the affine and congruence
mechanisms: the Cartan decomposition places the relative linear
transformation in $\exp(\mathfrak p)$, where it is self-adjoint and
positive definite, and this algebraic structure certifies optimality.
\end{remark}

Many other product-form scale families admit the same analysis, provided
the shape parameters are fixed and the relevant second moments are
finite. In such cases, the orbit is generated by coordinatewise positive
scalings, the stabilizer is trivial, and the optimal map is again
diagonal. The closed-form value reduces to a sum of one-dimensional
second-moment calculations, exactly as in
\Cref{cor:prod-exp-ot}. Examples include products of Weibull,
Rayleigh, Gamma, inverse-Gamma, Pareto, lognormal, and generalized-Gamma
laws, under the corresponding parameter restrictions ensuring finite
quadratic cost.

More generally, whenever the cost is additively separable across coordinates and measures are product measures, the Kantorovich problem decouples into one-dimensional
problems: a product of one-dimensional optimal plans is optimal, and the
Kantorovich value is the sum of the one-dimensional values. The
one-dimensional case is treated in
\Cref{app:1d-infinite-dimensional-case}.

\section{Proofs and auxiliary results for \Cref{sec:examples}}
\label{app:examples-proofs}

This appendix contains the technical proofs and auxiliary calculations
used in \Cref{sec:examples}, organized according to the corresponding
distribution families.

\subsection{Elliptical distributions}
\label{app:elliptical-example}

\begin{lemma}
\label{lem:elliptical-affine-orbit}
\textnormal{(i)} $\{\alpha(g)_\#\rho:g\in G_{\mathrm{aff}}\}
=\{\mathcal E_{\varsigma}(\bs m,\bs S):
\bs m\in\R^d,\bs S\in\mathbb S_{\succ}^d\}$,
\textnormal{(ii)} $
\Stab_{G_{\mathrm{aff}}}(\rho)
=
\{(\bs 0_d,\bs Q):\bs Q\in\mc O(d)\}$.
\end{lemma}
\begin{proof}
(i) Follows by \cite[\S~1.5]{ref:muirhead2009aspects}.
Indeed, if $\bs X\sim\rho$ and $\bs Y=\bs m+\bs A\bs X$, then the characteristic function of $\bs Y$ evaluated at $\bs t \in \R^d$ is $
\exp({i\bs t^\top\bs m})
\varsigma(\bs t^\top\bs A\bs A^\top\bs t)$,
so $\bs Y\sim\mathcal E_{\varsigma}(\bs m,\bs A\bs A^\top)$.
Conversely, every $\bs S\in\mathbb S_{\succ}^d$ can be written as
$\bs S=\bs A\bs A^\top$ with $\bs A=\bs S^{1/2}$.
(ii) Let $(\bs u,\bs A)\in\Stab_{G_{\mathrm{aff}}}(\rho)$ and let
$\bs x\sim\rho$. Then $
\bs u+\bs A\bs x\sim\rho$.
Since $\rho$ has mean $\bs 0_d$, taking expectations gives $
\bs u+\bs A\mathbb E_{\bs X \sim \rho}[\bs X]=\bs 0_d$,
and hence $\bs u=\bs 0_d$. Since $\rho$ has covariance $\bs I_d$,
taking covariances gives $
\bs A \mathrm{Cov}(\bs x) \bs A^\top
=
\mathrm{Cov}(\bs x)$,
that is, $
\bs A\bs A^\top=\bs I_d$.
Thus $\bs A\in\mc O(d)$, so $
\Stab_{G_{\mathrm{aff}}}(\rho)
\subseteq
\{(\bs 0_d,\bs Q):\bs Q\in\mc O(d)\}$.
Conversely, let $\bs Q\in\mc O(d)$. The characteristic function of
$\rho$ is $
\phi_\rho(\bs t)=\varsigma(\|\bs t\|^2)$.
The characteristic function of $\bs Q\bs x$ is $
\bs t\mapsto \phi_\rho(\bs Q^\top\bs t)
=
\varsigma(\|\bs Q^\top\bs t\|^2)
=
\varsigma(\|\bs t\|^2)
=
\phi_\rho(\bs t)$.
Hence
$\bs Q\bs x{\color{black}\stackrel{\mathrm d}{=}}\bs x$, and therefore
$(\bs 0_d,\bs Q)\in\Stab_{G_{\mathrm{aff}}}(\rho)$.
\end{proof}

\begin{lemma}
\label{lem:elliptical-cartan-structure}
Let
$\rho=\mc E_{\varsigma}(\bs 0_d,\bs I_d)$
be the standardized elliptical reference law introduced above.
Then the following statements hold.
\begin{enumerate}[label=\textnormal{(\roman*)}]
    \item
    The map $\pi$ is a smooth group homomorphism, $b$ is a smooth
    $\pi$-cocycle, and
    $\alpha(g)(\bs x)=b(g)+\pi(g)\bs x$ for every
    $g\in G_{\mathrm{aff}}$ and $\bs x\in\R^d$.
    Hence Assumption~\ref{ass:affine-induced-pi} holds.

    \item
    We have $L=\GL(d)$, and $L$ is a Lie subgroup
    of $\GL(\mc H)$. The map $\theta:L\to L$ defined by
    $\theta(\bs A)=\bs A^{-\top}$ is a Lie group involution.
    Moreover, $\Lie(L)=\mathfrak{gl}(d)$,
    $L^\theta=\mc O(d)$,
    $\mathfrak k
    =
    \{\bs K\in\R^{d\times d}:\bs K^\top=-\bs K\}$, and
    $\mathfrak p=\Sym(d)$.
    The Cartan map
    $L^\theta\times\mathfrak p\to L$,
    $(\bs Q,\bs S)\mapsto\bs Q\exp(\bs S)$,
    is a diffeomorphism. Thus $(L,\theta)$ admits a global Cartan
    decomposition in the sense of
    \Cref{def:cartan-decomposition}.

    \item
    Every element of $\mathfrak p$ is self-adjoint on $\mc H$, and
    $L^\theta
    =
    \pi(\Stab_{G_{\mathrm{aff}}}(\rho))$.
\end{enumerate}
Consequently, Assumption~\ref{ass:affine-induced-pi} and all three
structural hypotheses of
\Cref{prop:affine-cartan-criterion} are satisfied.
\end{lemma}

\begin{proof}[Proof of \Cref{lem:elliptical-cartan-structure}]
\noindent\textnormal{(i)}
For $g=(\bs m,\bs A)$ and
$\widetilde g=(\bs n,\bs B)$ in $G_{\mathrm{aff}}$, the group law gives
$g\widetilde g=(\bs m+\bs A\bs n,\bs A\bs B)$. Therefore,
$\pi(g\widetilde g)=\bs A\bs B=\pi(g)\pi(\widetilde g)$, so
$\pi:G_{\mathrm{aff}}\to\GL(\R^d)$ is a group homomorphism. Moreover,
$b(g\widetilde g)=\bs m+\bs A\bs n=b(g)+\pi(g)b(\widetilde g)$, and
$\alpha(g)(\bs x)=\bs m+\bs A\bs x=b(g)+\pi(g)\bs x$ for every
$\bs x\in\R^d$. The maps $b$ and $\pi$ are smooth because they are the
coordinate projections of the smooth manifold
$\R^d\times\GL(d)$. Hence $(b,\pi)$ is an affine representation
inducing $\alpha$, and Assumption~\ref{ass:affine-induced-pi} holds.

\noindent\textnormal{(ii)}
By part \textnormal{(i)}, $\pi$ is a group homomorphism, so $L = \pi(G_{\mathrm{aff}})$ is a subgroup of $\GL(\R^d)$. Moreover, for every $\bs A\in\GL(d)$, we have
$\pi((\bs 0_d,\bs A))=\bs A$, and thus $L = \GL(d) = \GL(\mc H)$. This implies $L$ is a subgroup of $\GL(\mc H)$.

Define $\theta:L\to L$ by $\theta(\bs A)=\bs A^{-\top}$. For every
$\bs A,\bs B\in L$, we have
$\theta(\bs A\bs B)=(\bs A\bs B)^{-\top}
=\bs A^{-\top}\bs B^{-\top}
=\theta(\bs A)\theta(\bs B)$, and
$\theta^2(\bs A)=\bs A$. Moreover, $\theta$ is smooth because matrix
inversion is smooth on $\GL(d)$ and transposition is linear. Hence
$\theta$ is a Lie group involution.

An element $\bs A\in L$ belongs to $L^\theta$ if and only if
$\bs A^{-\top}=\bs A$, which is equivalent to
$\bs A^\top\bs A=\bs I_d$. Therefore, $L^\theta=\mc O(d)$.

We next identify the eigenspaces of $\diff\theta_{\bs I_d}$. For
$\bs X\in\mathfrak{gl}(d)$, the first-order expansion
$(\bs I_d+t\bs X)^{-\top}
=\bs I_d-t\bs X^\top+o(t)$ shows that
$\diff\theta_{\bs I_d}(\bs X)=-\bs X^\top$. It follows that the
$+1$ eigenspace is
$\mathfrak k=\{\bs K\in\R^{d\times d}:\bs K^\top=-\bs K\}$ and the
$-1$ eigenspace is
$\mathfrak p=\{\bs S\in\R^{d\times d}:\bs S^\top=\bs S\}
=\Sym(d)$.

It remains to establish the global Cartan decomposition. Consider the
Cartan map $\Phi:L^\theta\times\mathfrak p\to L$ given by
$\Phi(\bs Q,\bs S)=\bs Q\exp(\bs S)$. Let
$\bs A\in\GL(d)$ and define
$\bs P=(\bs A^\top\bs A)^{1/2}$,
$\bs Q=\bs A\bs P^{-1}$, and
$\bs S=\log(\bs P)$. Since $\bs P\in\mathbb S_{\succ}^d$, the
principal logarithm $\bs S$ is well defined and belongs to
$\Sym(d)$. Furthermore,
$\bs Q^\top\bs Q
=\bs P^{-1}\bs A^\top\bs A\bs P^{-1}
=\bs I_d$, so $\bs Q\in\mc O(d)$, and
$\bs A=\bs Q\bs P=\bs Q\exp(\bs S)$. Thus $\Phi$ is surjective.

To prove uniqueness, suppose that
$\bs A=\bs Q\exp(\bs S)$ for some
$\bs Q\in\mc O(d)$ and $\bs S\in\Sym(d)$. Since
$\exp(\bs S)$ is symmetric positive definite, we have
$\bs A^\top\bs A=\exp(2\bs S)$. Hence
$(\bs A^\top\bs A)^{1/2}=\exp(\bs S)$ by uniqueness of the positive
definite square root. Consequently,
$\bs S=\log((\bs A^\top\bs A)^{1/2})$ and
$\bs Q=\bs A(\bs A^\top\bs A)^{-1/2}$. Therefore, the factorization is
unique.

The inverse of $\Phi$ is consequently given by
$\bs A\mapsto
\bigl(\bs A(\bs A^\top\bs A)^{-1/2},
\log((\bs A^\top\bs A)^{1/2})\bigr)$.
The map $\Phi$ is smooth, and its inverse is smooth because matrix
multiplication and transposition are smooth and the principal square
root, inverse, and principal logarithm are smooth on
$\mathbb S_{\succ}^d$. Hence $\Phi$ is a diffeomorphism. Therefore,
$(L,\theta)$ admits a global Cartan decomposition in the sense of
\Cref{def:cartan-decomposition}.

\noindent\textnormal{(iii)}
Let $\bs S\in\mathfrak p=\Sym(d)$. For every
$\bs x,\bs y\in\R^d$, we have
$\langle\bs S\bs x,\bs y\rangle_{\mc H}
=\bs x^\top\bs S^\top\bs y
=\bs x^\top\bs S\bs y
=\langle\bs x,\bs S\bs y\rangle_{\mc H}$.
Thus $\bs S$ is self-adjoint on $\R^d$.

Finally, \Cref{lem:elliptical-affine-orbit}(ii) gives
$\Stab_{G_{\mathrm{aff}}}(\rho)
=\{(\bs 0_d,\bs Q):\bs Q\in\mc O(d)\}$.
Therefore,
\[\pi(\Stab_{G_{\mathrm{aff}}}(\rho))
=\{\bs Q:\bs Q\in\mc O(d)\}
=\mc O(d)=L^\theta.\]
Part~\textnormal{(i)} verifies
Assumption~\ref{ass:affine-induced-pi}, while
parts~\textnormal{(ii)} and~\textnormal{(iii)} verify the three
structural hypotheses of
\Cref{prop:affine-cartan-criterion}. This proves the claim.
\end{proof}

{\color{black}\begin{proof}[Proof of \Cref{cor:elliptical-ot}]
By the definition of $\bs A^\star$, $
\bs A^\star \bs\Sigma_0 \bs A^\star
=
\bs\Sigma_0^{-1/2}
\bigl(\bs\Sigma_0^{1/2}\bs\Sigma_1\bs\Sigma_0^{1/2}\bigr)
\bs\Sigma_0^{-1/2}
=
\bs\Sigma_1$.
Since $\bs A^\star$ is symmetric, it follows that $
\bs Q^\star(\bs Q^\star)^\top
=
\bs\Sigma_1^{-1/2}
\bs A^\star\bs\Sigma_0\bs A^\star
\bs\Sigma_1^{-1/2}
=
\bs I_d$.
Hence $\bs Q^\star\in\mc O(d)$, and therefore
$h^\star=(\bs 0_d,\bs Q^\star)\in\Stab_{G_{\mathrm{aff}}}(\rho)$
by \Cref{lem:elliptical-affine-orbit}.

Moreover,
$
g_0^{-1}
=
(-\bs\Sigma_0^{-1/2}\bs m_0,
 \bs\Sigma_0^{-1/2})$,
and the affine group law gives
$
g_1h^\star g_0^{-1}
=
(
\bs m_1-\bs A^\star\bs m_0,
\bs A^\star
)$,
because
$
\bs\Sigma_1^{1/2}\bs Q^\star\bs\Sigma_0^{-1/2}
=
\bs A^\star$.
Recall that
$
T_{h^\star}
=
\alpha(g_1h^\star g_0^{-1})
$
is the corresponding orbit-induced transport map.
Thus,
$
T_{h^\star}(\bs x)
=
\bs m_1+\bs A^\star(\bs x-\bs m_0)
$.
The reference law has mean $\bs 0_d$ and covariance $\bs I_d$, so
the mean and covariance operators of $\mu_i$ are $\bs m_i$ and
$\bs\Sigma_i$, respectively. All conclusions now follow from
\Cref{prop:affine-cartan-criterion}.
\end{proof}}

\subsection{Spectrally invariant distributions}
\label{app:cong-examples-additional}
\begin{lemma}[Congruence orbit and stabilizer of a spectrally invariant law]
\label{lem:spectral-congruence-orbit}
Let
$\rho\in\mc P_2(\mathbb S_{\succ}^d)$
be spectrally invariant. Then the following statements hold.
\begin{enumerate}[label=\textnormal{(\roman*)}]
    \item For every $\bs A\in\GL(d)$,
    $\alpha(\bs A)_\#\rho
    =
    \mu_{\bs A\bs A^\top}$.
    Consequently,
    $G_\#\rho
    =
    \{\mu_{\bs\Sigma}:
    \bs\Sigma\in\mathbb S_{\succ}^d\}$.

    \item The stabilizer of $\rho$ under the congruence action is
    $\Stab_{\GL(d)}(\rho)=\mc O(d)$.
\end{enumerate}
\end{lemma}
\begin{proof}
\noindent\textnormal{(i)}
Fix $\bs A\in\GL(d)$. Its left polar decomposition is
$\bs A=\bs P\bs Q$, where
$\bs P=(\bs A\bs A^\top)^{1/2}\in\mathbb S_{\succ}^d$ and
$\bs Q=\bs P^{-1}\bs A\in\mc O(d)$. Indeed,
$\bs Q\bs Q^\top
=
\bs P^{-1}\bs A\bs A^\top\bs P^{-1}
=
\bs I_d$.

Since $\alpha$ is a left action,
$\alpha(\bs A)=\alpha(\bs P)\circ\alpha(\bs Q)$.
Spectral invariance gives
$\alpha(\bs Q)_\#\rho=\rho$, and therefore
$\alpha(\bs A)_\#\rho
=
\alpha(\bs P)_\#
(\alpha(\bs Q)_\#\rho)
=
\alpha(\bs P)_\#\rho$.
Since $\bs P^2=\bs A\bs A^\top$, the definition of
$\mu_{\bs\Sigma}$ yields
$\alpha(\bs A)_\#\rho
=
\mu_{\bs A\bs A^\top}$.

It follows that every element of $G_\#\rho$ belongs to
$\{\mu_{\bs\Sigma}:
\bs\Sigma\in\mathbb S_{\succ}^d\}$.
Conversely, for every
$\bs\Sigma\in\mathbb S_{\succ}^d$, we have
$\mu_{\bs\Sigma}
=
\alpha(\bs\Sigma^{1/2})_\#\rho$, so
$\mu_{\bs\Sigma}\in G_\#\rho$.
This proves~\textnormal{(i)}.

\noindent\textnormal{(ii)}
Spectral invariance implies 
$\mc O(d)\subseteq\Stab_{\GL(d)}(\rho)$.
It remains to prove the reverse inclusion:
Let $\bs X\sim\rho$ and set
$\bs M=\mathbb E[\bs X]$.
The expectation is finite because
$\rho\in\mc P_2(\mathbb S_{\succ}^d)$.
For every $\bs Q\in\mc O(d)$, spectral invariance gives
$\bs Q\bs X\bs Q^\top\sim \rho$, and hence
$\bs Q\bs M\bs Q^\top=\bs M$.
We first show that $\bs M$ is a scalar multiple of the identity.
For each $r\in\{1,\ldots,d\}$, let $\bs D_r$ be the diagonal
orthogonal matrix whose $r$th diagonal entry is $-1$ and whose
remaining diagonal entries are $1$. Since
$\bs D_r\bs M\bs D_r=\bs M$, comparison of the
$(r,j)$ entry for $j\neq r$ gives $-M_{rj}=M_{rj}$.
Thus $M_{rj}=0$ whenever $r\neq j$, and therefore $\bs M$ is
diagonal. Next, invariance under permutation matrices implies that
all diagonal entries of $\bs M$ are equal. Consequently,
$\bs M=\tau_\rho\bs I_d$, where
$\tau_\rho=d^{-1}\mathbb E[\Tr(\bs X)]$.

Because $\bs X\in\mathbb S_{\succ}^d$ almost surely,
$\Tr(\bs X)>0$ almost surely. The random variable
$\Tr(\bs X)$ is integrable, and hence $\tau_\rho>0$.

Now let
$\bs A\in\Stab_{\GL(d)}(\rho)$. Then
$\bs A\bs X\bs A^\top\stackrel{\mathrm d}{=}\bs X$, so taking expectations gives
$\bs A\bs M\bs A^\top=\bs M$. Substituting
$\bs M=\tau_\rho\bs I_d$ and using $\tau_\rho>0$ yields
$\bs A\bs A^\top=\bs I_d$. Thus
$\bs A\in\mc O(d)$.

We have proved
$\Stab_{\GL(d)}(\rho)\subseteq\mc O(d)$.
Together with the reverse inclusion obtained from spectral
invariance, this gives
$\Stab_{\GL(d)}(\rho)=\mc O(d)$.
\end{proof}

\begin{proof}[Proof of \Cref{lem:congruence-cartan-structure}]
\noindent\textnormal{(i)}
For every $\bs A\in G=\GL(d)$ and $\bs X\in\mc H=\Sym(d)$,
the matrix $\bs A\bs X\bs A^\top$ is symmetric. Moreover,
$\pi(\bs A^{-1})$ is the inverse of $\pi(\bs A)$, so
$\pi(\bs A)\in\GL(\mc H)$.

For $\bs A,\bs B\in G$ and $\bs X\in\mc H$, we have
$\pi(\bs A\bs B)(\bs X)
=\bs A\bs B\bs X\bs B^\top\bs A^\top
=\pi(\bs A)(\pi(\bs B)(\bs X))$.
Thus $\pi(\bs A\bs B)=\pi(\bs A)\pi(\bs B)$.
The map $\pi:G\to\GL(\mc H)$ is smooth because its matrix
coefficients, in any fixed basis of $\Sym(d)$, are polynomial
functions of the entries of $\bs A$.

Since $b\equiv0$, for every $\bs A,\bs B\in G$ we have
$b(\bs A\bs B)=\bs 0_{\mc H}
=b(\bs A)+\pi(\bs A)b(\bs B)$.
Furthermore, for every $\bs X\in\mc X=\mathbb S_{\succ}^d$,
$\alpha(\bs A)(\bs X)
=\bs A\bs X\bs A^\top
=b(\bs A)+\pi(\bs A)(\bs X)$.
Hence the congruence action is induced by the affine representation
$(b,\pi)$, and Assumption~\ref{ass:affine-induced-pi} holds.

\noindent\textnormal{(ii)}
We first show that $L=\pi(G)$ is closed in $\GL(\mc H)$.
Let $\|\cdot\|_{\mathrm{op}}$ denote the operator norm on
$\GL(\mc H)$ induced by the Frobenius norm on $\mc H$.
For every $\bs A\in G$,
$\|\pi(\bs A)\|_{\mathrm{op}}=\|\bs A\|_{\mathrm{op}}^2$.
Indeed,
$\|\bs A\bs X\bs A^\top\|_{\mathrm F}
\leq\|\bs A\|_{\mathrm{op}}^2\|\bs X\|_{\mathrm F}$,
which gives the upper bound. For the reverse inequality, let
$\bs v$ be a unit right-singular vector of $\bs A$ corresponding
to its largest singular value, and set
$\bs X=\bs v\bs v^\top$. Then
$\|\bs X\|_{\mathrm F}=1$ and
$\|\pi(\bs A)(\bs X)\|_{\mathrm F}
=\|(\bs A\bs v)(\bs A\bs v)^\top\|_{\mathrm F}
=\|\bs A\bs v\|^2
=\|\bs A\|_{\mathrm{op}}^2$.
Applying the same identity to $\bs A^{-1}$ gives
$\|\pi(\bs A)^{-1}\|_{\mathrm{op}}
=\|\bs A^{-1}\|_{\mathrm{op}}^2$.

Now suppose that $\pi(\bs A_n)\to\ell$ in
$\GL(\mc H)$. Since inversion is continuous on $\GL(\mc H)$,
we also have $\pi(\bs A_n)^{-1}\to\ell^{-1}$. The preceding norm
identities imply that both $\{\bs A_n\}$ and
$\{\bs A_n^{-1}\}$ are bounded. Passing to a subsequence, we may
assume that $\bs A_n\to\bs A$ for some
$\bs A\in\mathbb R^{d\times d}$.

There exists $M<\infty$ such that
$\|\bs A_n^{-1}\|_{\mathrm{op}}\leq M$ for every $n$. Hence
$\|\bs A_n\bs x\|\geq M^{-1}\|\bs x\|$ for every
$\bs x\in\mathbb R^d$. Passing to the limit gives
$\|\bs A\bs x\|\geq M^{-1}\|\bs x\|$, so $\bs A$ is injective
and therefore belongs to $\GL(d)$. By continuity of $\pi$,
$\ell=\pi(\bs A)\in L$. Thus $L$ is closed in $\GL(\mc H)$ and hence is an embedded Lie
subgroup by Cartan's closed subgroup theorem
\cite[Corollary~1.10.7]{ref:duistermaat2012lie}.

 For
$\bs A\in G$ and $\bs X,\bs Y\in\mc H$, cyclicity of the trace
gives
$\langle\pi(\bs A)(\bs X),\bs Y\rangle_{\mathrm F}
=\Tr(\bs A\bs X\bs A^\top\bs Y)
=\Tr(\bs X\bs A^\top\bs Y\bs A)
=\langle\bs X,\pi(\bs A^\top)(\bs Y)\rangle_{\mathrm F}$.
Consequently, $\pi(\bs A)^*=\pi(\bs A^\top)$.

It follows that the map
$\theta:L\to L$ defined by $\theta(\ell)=(\ell^*)^{-1}$ satisfies
$\theta(\pi(\bs A))
=(\pi(\bs A)^*)^{-1}
=\pi(\bs A^\top)^{-1}
=\pi(\bs A^{-\top})$.
In particular, $\theta$ maps $L$ into itself.

For $\ell_1,\ell_2\in L$, the identity
$(\ell_1\ell_2)^*=\ell_2^*\ell_1^*$ gives
$\theta(\ell_1\ell_2)
=((\ell_1\ell_2)^*)^{-1}
=(\ell_1^*)^{-1}(\ell_2^*)^{-1}
=\theta(\ell_1)\theta(\ell_2)$.
Moreover,
$\theta^2(\ell)=\ell$, because
$((\ell^*)^{-1})^*=\ell^{-1}$.
The map $\theta$ is smooth since the adjoint operation is linear
and inversion is smooth on $\GL(\mc H)$. Hence $\theta$ is a Lie
group involution.

We now identify its fixed-point subgroup. If
$\bs Q\in\mc O(d)$, then $\bs Q^{-\top}=\bs Q$, and therefore
$\theta(\pi(\bs Q))=\pi(\bs Q)$. Thus
$\pi(\mc O(d))\subseteq L^\theta$.

Conversely, suppose that $\pi(\bs A)\in L^\theta$. Then
$\pi(\bs A)=\pi(\bs A^{-\top})$. Evaluating both operators at
$\bs I_d$ gives
$\bs A\bs A^\top
=\bs A^{-\top}\bs A^{-1}
=(\bs A\bs A^\top)^{-1}$.
Since $\bs A\bs A^\top$ is symmetric positive definite, this
implies $(\bs A\bs A^\top)^2=\bs I_d$. Every eigenvalue of
$\bs A\bs A^\top$ is positive and has square equal to one, so
$\bs A\bs A^\top=\bs I_d$. Hence $\bs A\in\mc O(d)$, and
therefore $L^\theta=\pi(\mc O(d))$.

We next identify the Lie algebra of $L$. Since
$\pi:G\to L$ is a surjective smooth Lie-group homomorphism, it is
a submersion, and thus its differential at the identity, $\diff\pi_{\bs I_d}: \mathfrak{gl}(d) \to \Lie(L)$, is surjective. For
$\bs H\in\mathfrak{gl}(d)$ and $\bs X\in\mc H$,
differentiating
$\pi(\bs I_d+t\bs H)(\bs X)
=(\bs I_d+t\bs H)\bs X(\bs I_d+t\bs H)^\top$
at $t=0$ gives
$\diff\pi_{\bs I_d}(\bs H)(\bs X)
=\bs H\bs X+\bs X\bs H^\top
=\mc D_{\bs H}(\bs X)$.
Consequently,
$\Lie(L)=\{\mc D_{\bs H}:\bs H\in\mathfrak{gl}(d)\}$.

We record that the map
$\bs H\mapsto\mc D_{\bs H}$ is injective. Indeed, suppose that
$\mc D_{\bs H}=0$. Evaluating at $\bs I_d$ gives
$\bs H+\bs H^\top=0$, so $\bs H$ is skew-symmetric. It then
follows that $\bs H\bs X-\bs X\bs H=0$ for every
$\bs X\in\Sym(d)$. Choosing $\bs X$ to be diagonal with pairwise
distinct diagonal entries shows that every off-diagonal entry of
$\bs H$ vanishes. Since a skew-symmetric matrix has zero diagonal,
we obtain $\bs H=0$.

For $\bs H\in\mathfrak{gl}(d)$, we have
$\theta(\pi(\exp(t\bs H)))
=\pi(\exp(-t\bs H^\top))$. Differentiating at $t=0$ yields
$\diff\theta_e(\mc D_{\bs H})
=\mc D_{-\bs H^\top}$, where $e$ denotes the identity of $L$.
By injectivity of $\bs H\mapsto\mc D_{\bs H}$,
$\diff\theta_e(\mc D_{\bs H})=\mc D_{\bs H}$ if and only if
$\bs H^\top=-\bs H$, whereas
$\diff\theta_e(\mc D_{\bs H})=-\mc D_{\bs H}$ if and only if
$\bs H^\top=\bs H$. Therefore,
$\mathfrak k
=\{\mc D_{\bs K}:\bs K^\top=-\bs K\}$ and
$\mathfrak p
=\{\mc D_{\bs S}:\bs S\in\Sym(d)\}$.

Compatibility of Lie-group homomorphisms with exponential maps
gives
$\exp(\mc D_{\bs H})=\pi(\exp(\bs H))$ for every
$\bs H\in\mathfrak{gl}(d)$, and in particular
$\exp(\mc D_{\bs S})=\pi(\exp(\bs S))$ for every
$\bs S\in\Sym(d)$.

We now prove the global Cartan decomposition. Let
$\Phi:L^\theta\times\mathfrak p\to L$ denote the Cartan map
$\Phi(k,Z)=k\exp(Z)$.

Fix $\ell=\pi(\bs A)\in L$ and let
$\bs A=\bs Q\bs P$ be the right polar decomposition of $\bs A$, where
$\bs P=(\bs A^\top\bs A)^{1/2}\in\mathbb S_{\succ}^d$ and
$\bs Q=\bs A\bs P^{-1}\in\mc O(d)$. Set
$\bs S=\log(\bs P)\in\Sym(d)$. Then
$\pi(\bs Q)\in L^\theta$,
$\mc D_{\bs S}\in\mathfrak p$, and
$\exp(\mc D_{\bs S})=\pi(\exp(\bs S))=\pi(\bs P)$.
Consequently,
$\ell=\pi(\bs Q)\exp(\mc D_{\bs S})$, and hence $\Phi$ is
surjective.
 Suppose that
$\ell=k_i\exp(Z_i)$, where
$k_i\in L^\theta$ and $Z_i\in\mathfrak p$ for $i=1,2$.
By the description of $\mathfrak p$, there exist
$\bs S_i\in\Sym(d)$ such that
$Z_i=\mc D_{\bs S_i}$. Set
$\bs P_i=\exp(\bs S_i)\in\mathbb S_{\succ}^d$.
Then
$\exp(Z_i)=\pi(\bs P_i)$.

Every element of $L^\theta=\pi(\mc O(d))$ is an orthogonal operator
on $\Sym(d)$. Moreover, $\pi(\bs P_i)$ is self-adjoint because
$\pi(\bs P_i)^*=\pi(\bs P_i^\top)=\pi(\bs P_i)$, and it is positive
definite because, for every nonzero $\bs X\in\Sym(d)$,
$\langle\pi(\bs P_i)(\bs X),\bs X\rangle_{\mathrm F}
=
\|\bs P_i^{1/2}\bs X\bs P_i^{1/2}\|_{\mathrm F}^2>0$.
Thus
$\ell=k_i\pi(\bs P_i)$ is the polar decomposition of $\ell$ as an
invertible operator on $\Sym(d)$.

Uniqueness of the operator polar decomposition gives
$k_1=k_2$ and
$\pi(\bs P_1)=\pi(\bs P_2)$. Evaluating the latter identity at
$\bs I_d$ yields
$\bs P_1^2=\bs P_2^2$. Since both matrices are positive definite,
uniqueness of the positive-definite square root implies
$\bs P_1=\bs P_2$. Taking principal logarithms gives
$\bs S_1=\bs S_2$, and therefore $Z_1=Z_2$. Hence $\Phi$ is
bijective.

It remains to verify smoothness of the inverse. For $\ell\in L$, set
$p_\ell=(\ell^*\ell)^{1/2}$,
$k_\ell=\ell p_\ell^{-1}$,
$\bs S_\ell=\frac12\log(p_\ell(\bs I_d))$, and
$Z_\ell=\mc D_{\bs S_\ell}$.
To see that these formulas recover the preceding factorization, write
$\ell=\pi(\bs A)$ and let
$\bs A=\bs Q\bs P$ be its right polar decomposition. Then
$\ell^*\ell
=
\pi(\bs A^\top\bs A)
=
\pi(\bs P^2)
=
\pi(\bs P)^2$.
As shown above, $\pi(\bs P)$ is self-adjoint and positive definite,
and hence
$p_\ell=\pi(\bs P)$. It follows that
$k_\ell=\pi(\bs Q)$ and
$p_\ell(\bs I_d)=\bs P^2$. Therefore,
$\bs S_\ell=\log(\bs P)$ and
$Z_\ell=\mc D_{\log(\bs P)}$. Thus
$\Phi^{-1}(\ell)=(k_\ell,Z_\ell)$.

The map $\Phi$ is smooth. Its inverse is smooth because the adjoint,
operator multiplication, positive-definite operator square root,
operator inversion, evaluation at $\bs I_d$, and the principal matrix
logarithm on $\mathbb S_{\succ}^d$ are smooth. Hence $\Phi$ is a
diffeomorphism, and $(L,\theta)$ admits a global Cartan decomposition
in the sense of \Cref{def:cartan-decomposition}.

\noindent\textnormal{(iii)}
Let $Z\in\mathfrak p$. By the description of $\mathfrak p$ in
part~\textnormal{(ii)}, there exists
$\bs S\in\Sym(d)$ such that $Z=\mc D_{\bs S}$.
For every $\bs X,\bs Y\in\Sym(d)$, cyclicity of the trace and
$\bs S^\top=\bs S$ give
$\langle Z(\bs X),\bs Y\rangle_{\mathrm F}
=
\Tr((\bs S\bs X+\bs X\bs S)\bs Y)
=
\Tr(\bs X(\bs Y\bs S+\bs S\bs Y))
=
\langle\bs X,Z(\bs Y)\rangle_{\mathrm F}$.
Thus $Z$ is self-adjoint on $\Sym(d)$. Since $Z\in\mathfrak p$ was
arbitrary, every element of $\mathfrak p$ is self-adjoint.\\
Finally,
\Cref{lem:spectral-congruence-orbit}\textnormal{(ii)} gives
$\Stab_G(\rho)=\mc O(d)$. Consequently,
$L^\theta=\pi(\mc O(d))
=\pi(\Stab_G(\rho))$.

Part~\textnormal{(i)} verifies
Assumption~\ref{ass:affine-induced-pi}.
Part~\textnormal{(ii)} verifies that $L$ is a Lie subgroup, that
$\theta$ is a Lie group involution, and that $(L,\theta)$ admits a
global Cartan decomposition. Part~\textnormal{(iii)} verifies the
self-adjointness and stabilizer hypotheses. Hence all the structural
hypotheses of \Cref{prop:affine-cartan-criterion} are satisfied.

\end{proof}

\begin{proof}[Proof of \Cref{lem:spectral-moment-structure}]
{\color{black}Let $\bs M=\EE[\bs X]$. Spectral invariance gives
$\bs Q\bs M\bs Q^\top=\bs M$ for every $\bs Q\in\mc O(d)$.
Invariance under diagonal sign matrices implies that $\bs M$ is
diagonal, while invariance under permutation matrices implies that
all of its diagonal entries are equal. Thus
$\bs M=\tau_\rho\bs I_d$, where
$\tau_\rho=d^{-1}\EE[\Tr(\bs X)]$.
Since $\bs X\in\mathbb S_{\succ}^d$ almost surely and
$\Tr(\bs X)$ is integrable, $\tau_\rho>0$.}

{\color{black}Let $\bs X\sim\rho$ and define the fourth-order tensor $
T_{ijkl}
=
\EE_{\bs X\sim\rho}[X_{ij}X_{kl}]$,
$i,j,k,l\in\{1,\ldots,d\}$.
Since $\rho$ is invariant under orthogonal congruence, we have
$\bs Q\bs X\bs Q^\top\sim\rho$ for every $\bs Q\in\mc O(d)$.
Therefore, for every $i,j,k,l\in\{1,\ldots,d\}$,
\[
T_{ijkl}
=
\EE_{\bs X\sim\rho}
\bigl[
(\bs Q\bs X\bs Q^\top)_{ij}
(\bs Q\bs X\bs Q^\top)_{kl}
\bigr].
\]}
Using $
(\bs Q\bs X\bs Q^\top)_{ij}
=
\sum_{i',j'=1}^d Q_{ii'}X_{i'j'}Q_{jj'}$,
we obtain $T_{ijkl}
=
\sum_{i',j',k',l'=1}^d
Q_{ii'}Q_{jj'}Q_{kk'}Q_{ll'}$ $ T_{i'j'k'l'}$.

In other words, the tensor
$T_{ijkl}=\EE_{\bs X \sim \rho}[X_{ij}X_{kl}]$
is invariant under the simultaneous congruence
$T\mapsto
\bs Q^{\otimes 4}T$ for every $\bs Q\in\mc O(d)$.
By the First Fundamental Theorem for the orthogonal
group, every $O(d)$-invariant rank-four tensor is a linear combination of
the three pairings $
\delta_{ij}\delta_{kl},~
\delta_{ik}\delta_{jl},~
\delta_{il}\delta_{jk}$.
Hence, for some scalars $ A,B,C$, $
T_{ijkl}
=
A \delta_{ij}\delta_{kl}
+
B \delta_{ik}\delta_{jl}
+
C \delta_{il}\delta_{jk}$.
Since $\bs X$ is symmetric, we have $T_{ijkl}=T_{jikl}=T_{ijlk}$, which forces
$B=C$. Writing $a_\rho=A$ and $b_\rho=B$ yields $
T_{ijkl}
=
a_\rho \delta_{ij}\delta_{kl}
+
b_\rho(\delta_{ik}\delta_{jl}+\delta_{il}\delta_{jk})$.
For the identity~\eqref{eq:master-trace},
note that
$\Tr(\bs U\bs X\bs V\bs X)
=\sum_{i,j,k,l }U_{ij}V_{kl }X_{jk}X_{l  i}$.
Substituting~\eqref{eq:second-moment-ansatz} for
$\EE_{\bs X \sim \rho}[X_{jk}X_{l i}]$:
the term $a_\rho\delta_{jk}\delta_{l i}$ contributes
$a_\rho\sum_{i,j}U_{ij}V_{ji}=a_\rho\Tr(\bs U\bs V)$;
the term
$b_\rho\delta_{j l}\delta_{ki}$ contributes
$b_\rho\sum_{j,k}U_{kj}V_{kj}
=b_\rho\Tr(\bs U\bs V^\top)$;
and the term $b_\rho\delta_{ji}\delta_{kl}$ contributes $
b_\rho
\sum_{i=1}^d U_{ii}
\sum_{k=1}^d V_{kk}
=
b_\rho\Tr(\bs U)\Tr(\bs V)$.

Next,
evaluating~\eqref{eq:second-moment-ansatz}
with $i=j,k=l$ gives
$\EE_{\bs X \sim \rho}[(\Tr(\bs X))^2]=a_\rho d^2+2b_\rho d$,
and with $i=k,j=l$ gives
$\EE_{\bs X \sim \rho}[\Tr(\bs X^2)]=a_\rho d+b_\rho d(d+1)$.
Solving this $2\times 2$ system
yields~\eqref{eq:alpha-beta-from-traces}.

Finally, fix $\bs\Sigma\in\mathbb S_{\succ}^d$, set
$\bs R=\bs\Sigma^{1/2}$, and let
$\bs Y=\bs R\bs X\bs R\sim\mu_{\bs\Sigma}$.
Since $\bs R$ is symmetric and $\bs R^2=\bs\Sigma$, we have
$\EE[\bs Y]
=
\bs R\EE[\bs X]\bs R
=
\bs R(\tau_\rho\bs I_d)\bs R
=
\tau_\rho\bs\Sigma$.
This proves the first identity in
\eqref{eq:spectral-mean-covariance}.

We next identify the covariance operator. Fix
$\bs H,\bs K\in\Sym(d)$ and define
$\bs U=\bs R\bs H\bs R$ and
$\bs V=\bs R\bs K\bs R$.
Both $\bs U$ and $\bs V$ belong to $\Sym(d)$. Moreover, cyclicity
of the trace gives
$\langle\bs H,\bs Y\rangle_{\mathrm F}
=
\Tr(\bs H\bs R\bs X\bs R)
=
\Tr(\bs R\bs H\bs R\bs X)
=
\Tr(\bs U\bs X)$
and, similarly,
$\langle\bs K,\bs Y\rangle_{\mathrm F}
=
\Tr(\bs V\bs X)$.
All the expectations below are finite because
$|\Tr(\bs U\bs X)|
\leq
\|\bs U\|_{\mathrm F}\|\bs X\|_{\mathrm F}$
and $\rho$ has finite second moment.

We first calculate the uncentered second moment. Since $\bs X$ is
symmetric,
$\Tr(\bs U\bs X)=\sum_{i,j=1}^d U_{ij}X_{ij}$ and
$\Tr(\bs V\bs X)=\sum_{k,l=1}^d V_{kl}X_{kl}$.
Consequently,
\[\EE[\Tr(\bs U\bs X)\Tr(\bs V\bs X)]
=
\sum_{i,j,k,l=1}^d
U_{ij}V_{kl}\EE[X_{ij}X_{kl}].\]
Substituting \eqref{eq:second-moment-ansatz}, the term containing
$a_\rho\delta_{ij}\delta_{kl}$ contributes
$a_\rho\sum_{i,k=1}^dU_{ii}V_{kk}
\!=\!
a_\rho\Tr(\bs U)\Tr(\bs V)$.
The term containing
$b_\rho\delta_{ik}\delta_{jl}$ contributes
$b_\rho\sum_{i,j=1}^dU_{ij}V_{ij}
=
b_\rho\Tr(\bs U\bs V^\top)
=
b_\rho\Tr(\bs U\bs V)$,
where the last equality uses $\bs V^\top=\bs V$.
The term containing
$b_\rho\delta_{il}\delta_{jk}$ contributes
$b_\rho\sum_{i,j=1}^dU_{ij}V_{ji}
=
b_\rho\Tr(\bs U\bs V)$.
Therefore,
$\EE[\Tr(\bs U\bs X)\Tr(\bs V\bs X)]
=
a_\rho\Tr(\bs U)\Tr(\bs V)
+
2b_\rho\Tr(\bs U\bs V)$.

On the other hand,
$\EE[\Tr(\bs U\bs X)]
=
\Tr(\bs U\EE[\bs X])
=
\tau_\rho\Tr(\bs U)$,
and likewise
$\EE[\Tr(\bs V\bs X)]
=
\tau_\rho\Tr(\bs V)$.
Hence, by the definition of the covariance operator,
$\langle\bs H,
\mathsf C_{\bs\Sigma}(\bs K)\rangle_{\mathrm F}
=
2b_\rho\Tr(\bs U\bs V)
+
(a_\rho-\tau_\rho^2)
\Tr(\bs U)\Tr(\bs V)$.

It remains to express this identity in terms of
$\bs H$, $\bs K$, and $\bs\Sigma$. Cyclicity of the trace and
$\bs R^2=\bs\Sigma$ give
$\Tr(\bs U)
=
\Tr(\bs R\bs H\bs R)
=
\Tr(\bs\Sigma\bs H)$,
$\Tr(\bs V)
=
\Tr(\bs\Sigma\bs K)$, and
$\Tr(\bs U\bs V)
=
\Tr(\bs R\bs H\bs R^2\bs K\bs R)
=
\Tr(\bs H\bs\Sigma\bs K\bs\Sigma)$.
It follows that
$\langle\bs H,
\mathsf C_{\bs\Sigma}(\bs K)\rangle_{\mathrm F}
=
2b_\rho
\Tr(\bs H\bs\Sigma\bs K\bs\Sigma)
+
(a_\rho-\tau_\rho^2)
\Tr(\bs\Sigma\bs H)\Tr(\bs\Sigma\bs K)$.

Now define
$\widetilde{\mathsf C}_{\bs\Sigma}:\Sym(d)\to\Sym(d)$ by
$\widetilde{\mathsf C}_{\bs\Sigma}(\bs K)
=
2b_\rho\bs\Sigma\bs K\bs\Sigma
+
(a_\rho-\tau_\rho^2)
\Tr(\bs\Sigma\bs K)\bs\Sigma$.
For every $\bs H\in\Sym(d)$, we have
$\langle\bs H,
\widetilde{\mathsf C}_{\bs\Sigma}(\bs K)\rangle_{\mathrm F}
=
2b_\rho
\Tr(\bs H\bs\Sigma\bs K\bs\Sigma)
+
(a_\rho-\tau_\rho^2)
\Tr(\bs\Sigma\bs H)\Tr(\bs\Sigma\bs K)$.
Thus
$\langle\bs H,
\mathsf C_{\bs\Sigma}(\bs K)\rangle_{\mathrm F}
=
\langle\bs H,
\widetilde{\mathsf C}_{\bs\Sigma}(\bs K)\rangle_{\mathrm F}$
for every $\bs H\in\Sym(d)$.
By nondegeneracy of the Frobenius inner product,
$\mathsf C_{\bs\Sigma}(\bs K)
=
\widetilde{\mathsf C}_{\bs\Sigma}(\bs K)$.
Since $\bs K\in\Sym(d)$ was arbitrary, this proves the second
identity in \eqref{eq:spectral-mean-covariance} and completes the
proof.
\end{proof}
\begin{proof}[Proof of \Cref{thm:congruence-template}]
The cone
\(\mc X=\mathbb S_{\succ}^d\) is an open convex subset of the
finite-dimensional Hilbert space
\(\mc H=\Sym(d)\).
The absolute-continuity and finite-second-moment assumptions on
\(\rho\) are precisely those required in
\Cref{prop:affine-cartan-criterion}, while
\Cref{lem:congruence-cartan-structure} verifies
Assumption~\ref{ass:affine-induced-pi} and all three structural
hypotheses of that proposition. Hence
\Cref{prop:affine-cartan-criterion} applies.

Let \(\hat h\in\Stab_G(\rho)\) be the element supplied by
\Cref{prop:affine-cartan-criterion}, and set
\(\mathsf A_{\hat h}
=\pi(g_1\hat h g_0^{-1})\).
Then
\(\pi(\hat h)\in L^\theta\) and
\(\mathsf A_{\hat h}\in\exp(\mathfrak p)\).
By \Cref{lem:congruence-cartan-structure}(ii),
\(\mathfrak p
=
\{\mc D_{\bs S}:\bs S\in\Sym(d)\}\).
Thus there exists \(\bs S\in\Sym(d)\) such that
\(\mathsf A_{\hat h}=\exp(\mc D_{\bs S})\).
Set \(\bs P=\exp(\bs S)\).
Since \(\bs S\) is symmetric,
\(\bs P\in\mathbb S_{\succ}^d\), and the compatibility of the
congruence representation with exponential maps gives
\(\mathsf A_{\hat h}=\pi(\bs P)\).
Because \(b\equiv0\), the associated orbit-induced map is therefore
\(T_{\hat h}(\bs X)=\bs P\bs X\bs P\) for every
\(\bs X\in\mathbb S_{\succ}^d\).

We next identify \(\bs P\).
By \Cref{lem:spectral-moment-structure}, the mean of \(\mu_i\) is
\(\tau_\rho\bs\Sigma_i\), where \(\tau_\rho>0\).
If \(\bs X\sim\mu_0\), then
\(T_{\hat h}(\bs X)=\bs P\bs X\bs P\sim\mu_1\).
Taking expectations consequently yields
\(\tau_\rho\bs\Sigma_1
=
\bs P(\tau_\rho\bs\Sigma_0)\bs P\).
Since \(\tau_\rho>0\), it follows that
\(\bs P\bs\Sigma_0\bs P=\bs\Sigma_1\).
Set
\(\bs R=\bs\Sigma_0^{1/2}\bs P\bs\Sigma_0^{1/2}\).
Then \(\bs R\in\mathbb S_{\succ}^d\), and
\[
\begin{aligned}
\bs R^2
&=
\bs\Sigma_0^{1/2}
\bs P\bs\Sigma_0\bs P
\bs\Sigma_0^{1/2} =
\bs\Sigma_0^{1/2}
\bs\Sigma_1
\bs\Sigma_0^{1/2}
=
\bs C^2.
\end{aligned}
\]
Both \(\bs R\) and \(\bs C\) are positive definite. By uniqueness of
the positive-definite square root, \(\bs R=\bs C\), and hence
\(\bs P
=
\bs\Sigma_0^{-1/2}\bs C\bs\Sigma_0^{-1/2}
=
\bs A\opt\).
Consequently,
\(\mathsf A_{\hat h}=\pi(\bs A\opt)\) and
\(T_{\hat h}(\bs X)
=
\bs A\opt\bs X\bs A\opt
=
T\opt(\bs X)\).

It remains to show that the explicitly defined element
\(h\opt=\bs Q\opt\) can be used in place of \(\hat h\).
The preceding argument gives
\(\bs A\opt\bs\Sigma_0\bs A\opt=\bs\Sigma_1\).
Since \(\bs A\opt\) is symmetric, the definition of
\(\bs Q\opt\) gives
\[
\begin{aligned}
\bs Q\opt(\bs Q\opt)^\top
&=
\bs\Sigma_1^{-1/2}
\bs A\opt
\bs\Sigma_0
\bs A\opt
\bs\Sigma_1^{-1/2}
=
\bs I_d,\\
g_1h\opt g_0^{-1}
&=
\bs\Sigma_1^{1/2}
\bs Q\opt
\bs\Sigma_0^{-1/2}
=
\bs A\opt.
\end{aligned}
\]
Thus \(\bs Q\opt\in\mc O(d)\). Spectral invariance of \(\rho\)
therefore implies
\(h\opt=\bs Q\opt\in\Stab_G(\rho)\).
Moreover,
\(\pi(h\opt)\in\pi(\mc O(d))=L^\theta\), and
\(\mathsf A_{h\opt}
=
\pi(g_1h\opt g_0^{-1})
=
\pi(\bs A\opt)
=
\mathsf A_{\hat h}\).
In particular,
\(T_{h\opt}=T_{\hat h}=T\opt\).

The conclusions of
\Cref{prop:affine-cartan-criterion} concerning the linear operator,
the optimal Monge map, the optimal Kantorovich plan, and the equality
of the three optimal values therefore hold with \(h=h\opt\).
To verify the remaining conclusion concerning the orbit problem,
observe that for every
\(\widetilde h\in\Stab_G(\rho)\), the change of variables
\(\bs X=\alpha(g_0)(\bs Z)\) gives
\[
\begin{aligned}
&\int_{\mathbb S_{\succ}^d}
c\bigl(
\alpha(g_0)(\bs Z),
\alpha(g_1\widetilde h)(\bs Z)
\bigr)
\diff\rho(\bs Z)=
\int_{\mathbb S_{\succ}^d}
c\bigl(\bs X,T_{\widetilde h}(\bs X)\bigr)
\diff\mu_0(\bs X).
\end{aligned}
\]
Since \(T_{h\opt}=T_{\hat h}\), the objective in \eqref{eq:transport-wihtin-orbits} has the same
value at \(h\opt\) and \(\hat h\). The element \(\hat h\)
attains the infimum by
\Cref{prop:affine-cartan-criterion}; therefore \(h\opt\) also solves
\eqref{eq:transport-wihtin-orbits}. This proves that all the
conclusions of \Cref{prop:affine-cartan-criterion} hold with
\(h=h\opt\).

It remains to evaluate the common optimal value.
Let \(\bs Z\sim\rho\), set
\(\bs X_0
=
\bs\Sigma_0^{1/2}\bs Z\bs\Sigma_0^{1/2}\), and let
\(\bs X_1=T\opt(\bs X_0)\).
Then \(\bs X_i\sim\mu_i\) for \(i=0,1\).
For each \(i\), cyclicity of the trace and
\eqref{eq:master-trace}, applied with
\(\bs U=\bs V=\bs\Sigma_i\), give
\(\EE[\Tr(\bs X_i^2)]
=
(a_\rho+b_\rho)\Tr(\bs\Sigma_i^2)
+
b_\rho(\Tr\bs\Sigma_i)^2\).

By the definitions of \(\bs A\opt\) and \(\bs C\), we have
\(\bs C
=
\bs\Sigma_0^{1/2}\bs A\opt\bs\Sigma_0^{1/2}\).
Consequently, cyclicity of the trace gives
\(\Tr(\bs X_0\bs X_1)
=
\Tr(\bs Z\bs C\bs Z\bs C)
=
\Tr(\bs C\bs Z\bs C\bs Z)\).
Since \(\bs C\) is symmetric, another application of
\eqref{eq:master-trace}, now with
\(\bs U=\bs V=\bs C\), yields
\(\EE[\Tr(\bs X_0\bs X_1)]
=
(a_\rho+b_\rho)\Tr(\bs C^2)
+
b_\rho(\Tr\bs C)^2\).
Furthermore,
\(\bs C^2
=
\bs\Sigma_0^{1/2}\bs\Sigma_1\bs\Sigma_0^{1/2}\), and hence
\(\Tr(\bs C^2)=\Tr(\bs\Sigma_0\bs\Sigma_1)\).

Since \(T\opt\) is optimal, its expected transportation cost equals
\(\mathds K_c(\mu_0,\mu_1)\). Combining the preceding identities,
we obtain
\[
\begin{aligned}
\mathds K_c(\mu_0,\mu_1)
&=
\EE\!\left[
\|\bs X_0-\bs X_1\|_{\mathrm F}^2
\right]\\
&=
(a_\rho \! +\!b_\rho)
\left(
\Tr(\bs\Sigma_0^2)
\!+\!
\Tr(\bs\Sigma_1^2)
-
2\Tr(\bs\Sigma_0\bs\Sigma_1)
\right)+
b_\rho
\left(
(\Tr\bs\Sigma_0)^2
+
(\Tr\bs\Sigma_1)^2
-
2(\Tr\bs C)^2
\right)\\
&=
(a_\rho+b_\rho)
\|\bs\Sigma_0-\bs\Sigma_1\|_{\mathrm F}^2+
b_\rho
\left(
(\Tr\bs\Sigma_0)^2
+
(\Tr\bs\Sigma_1)^2
-
2(\Tr\bs C)^2
\right).
\end{aligned}
\]
This is precisely \eqref{eq:congruence-value}, and the proof is
complete.
\end{proof}

\medskip

In what follows, we collect the orbit, stabilizer, and
second-moment computations for the distributional families
summarized in Table~\ref{tab:congruence-families}.
Each subsection verifies the hypotheses of
Theorem~\ref{thm:congruence-template} and derives the
constants $a_\rho$ and $b_\rho$ that enter the optimal transport
cost formula~\eqref{eq:congruence-value}.

\subsubsection*{Wishart distributions}
The Wishart distribution $\mc W_d(\bs\Sigma,p)$ with scale matrix
$\bs\Sigma\in\mathbb S_\succ^d$ and degrees of freedom $p>d-1$ is a
probability law on $\mathbb S_\succ^d$ with density
\begin{equation}
\label{eq:wishart-density}
\bs X
\mapsto
\frac{
\det(\bs X)^{(p-d-1)/2}
\exp\!\left(-\frac12\Tr(\bs\Sigma^{-1}\bs X)\right)
}{
2^{pd/2}{\color{black}\kappa_d}
\det(\bs\Sigma)^{p/2}
\mathsf\Gamma_d(p/2)
},
\end{equation}
with respect to
$\mc L_{\Sym(d)}|_{\mathbb S_\succ^d}$, where \(\mathsf\Gamma_d\) denotes the multivariate gamma function
{\color{black}Here,
$\kappa_d=2^{d(d-1)/4}$.}
When $d=1$, this reduces to a scaled $\chi^2$ distribution with $p$
degrees of freedom and scale parameter $\Sigma>0$.

\begin{lemma}[Wishart orbit, stabilizer, and second moments]
\label{lem:wishart}
Let $\rho=\mc W_d(\bs I_d,p)$ with $p >  d-1$. 
Then:
\begin{enumerate}[label=\textnormal{(\roman*)}]
  \item $\alpha(\bs A)_\#\rho=\mc W_d(\bs A\bs A^\top,p)$
  for every $\bs A\in\mathrm{GL}(d)$, 
  \item $\Stab_{\mathrm{GL}(d)}(\rho)=\mc O(d)$,
  \item for $d\ge2$, the constants in \Cref{lem:spectral-moment-structure} are $
  a_\rho=p^2$,
  $b_\rho=p$.
  For $d=1$, the identifiable combination is $
  a_\rho+2b_\rho=p(p+2)$.
\end{enumerate}
\end{lemma}

\begin{proof}
(i) This is the standard transformation property of the Wishart
distribution; see, for example,
\cite[Proposition~8.1]{ref:eaton1983multivariate} or
\cite[Theorem~3.2.5]{ref:muirhead2009aspects}.
{(ii)} We next identify the stabilizer. If $\bs Q\in\mc O(d)$, then by (i), we have $
\alpha(\bs Q)_\#\rho
=
\mc W_d(\bs I_d,p)
=
\rho$.
Thus $
\mc O(d)\subseteq \Stab_{\GL(d)}(\rho)$.
Conversely, suppose $\bs A\in\Stab_{\GL(d)}(\rho)$. Then $
\mc W_d(\bs A\bs A^\top,p)
=
\alpha(\bs A)_\#\rho
=
\rho
=
\mc W_d(\bs I_d,p)$.
Then, the equality of the two densities implies $
\det(\bs A\bs A^\top)^{-p/2}
\exp (
-\tfrac12\Tr(((\bs A\bs A^\top)^{-1}-\bs I_d)\bs Y)
)
=
1$
for almost every $\bs Y\in\mathbb S_{\succ}^d$. Since both sides are
continuous in $\bs Y$, the identity holds for all
$\bs Y\in\mathbb S_{\succ}^d$. The exponent is a linear function of $\bs Y$. Since the displayed identity
holds for every $\bs Y\in\mathbb S_{\succ}^d$, this linear function is
constant on the open cone $\mathbb S_{\succ}^d$. A linear function that is
constant on a nonempty open set must vanish identically; so $
(\bs A\bs A^\top)^{-1}-\bs I_d=0$.
Thus $\bs A\bs A^\top=\bs I_d$, and hence $\bs A\in\mc O(d)$.
This proves $
\Stab_{\GL(d)}(\rho)=\mc O(d)$.
{(iii)} It remains to compute the constants $a_\rho$ and $b_\rho$.
We compute them from the two scalar moments $
M_1=\EE_{\bs X \sim \rho}[(\Tr(\bs X))^2]$ and 
$M_2=\EE_{\bs X \sim \rho}[\Tr(\bs X^2)]$.

Let $
Y_{\bs H}= \Tr(\bs H\bs X)$,
$\bs H\in\Sym(d)$,
where $\bs X\sim \rho=\mc W_d(\bs I_d,p)$.
For $t$ in a neighborhood of $0$, the Laplace transform of $Y_{\bs H}$ is $
L_{\bs H}(t) = 
\EE_{\bs X \sim \rho} \left[\exp({-tY_{\bs H}})\right]
=
\EE_{\bs X \sim \rho} \left[\exp({-t\Tr(\bs H\bs X)})\right]
=
\det(\bs I_d+2t\bs H)^{-p/2}$,
where the last equality follows from the characteristic function formula for the Wishart distribution in \cite[Theorem~3.2.3]{ref:muirhead2009aspects} by evaluating it at the imaginary argument $it \bs H$.
Since $L_{\bs H}(t)=\EE_{\bs X \sim \rho}[\exp({-tY_{\bs H}})]$, differentiation at $t=0$ gives $
L_{\bs H}'(0)=-\EE_{\bs X \sim \rho}[Y_{\bs H}]$ and 
$L_{\bs H}''(0)=\EE_{\bs X \sim \rho}[Y_{\bs H}^2]$.
Equivalently, $
L_{\bs H}'(0)=-\EE_{\bs X \sim \rho}[\Tr(\bs H\bs X)]$ and 
$L_{\bs H}''(0)=\EE_{\bs X \sim \rho}[(\Tr(\bs H\bs X))^2]$.

To compute these quantities, we differentiate $\log (L_{\bs H}(t))=
-\frac p2 \Tr (\log(\bs I_d+2t\bs H))$.
Hence $
(\log(L_{\bs H}))'(t)
=
-p\Tr \bigl((\bs I_d+2t\bs H)^{-1}\bs H\bigr)$,
and
\[
(\log(L_{\bs H}))''(t)
=
2p\Tr \bigl((\bs I_d+2t\bs H)^{-1}\bs H
(\bs I_d+2t\bs H)^{-1}\bs H\bigr).
\]
Evaluating the first and second derivatives of $\log(L_{\bs H})$ at $t=0$ yields $
(\log(L_{\bs H}))'(0)=-p\Tr(\bs H)$ and 
$(\log(L_{\bs H}))''(0)=2p\Tr(\bs H^2)$.

Now we use the identity: $(\log(L_{\bs H}))''(0)
=
{L_{\bs H}''(0)}/{L_{\bs H}(0)}
-
({L_{\bs H}'(0)}/{L_{\bs H}(0)})^2$. 
Since $L_{\bs H}(0)=1$, we have $
L_{\bs H}''(0)
=
(\log L_{\bs H})''(0)
+((\log L_{\bs H})'(0))^2$.
Therefore
\begin{equation}
\EE_{\bs X \sim \rho}[(\Tr(\bs H\bs X))^2]
=
2p\Tr(\bs H^2)+p^2(\Tr(\bs H))^2.
\label{eq:second-moment-tensor}
\end{equation}
Taking $\bs H=\bs I_d$ gives $
M_1 = 
\EE_{\bs X \sim \rho}[(\Tr(\bs X))^2]
=
2p\Tr(\bs I_d^2)+p^2(\Tr(\bs I_d))^2
=
2pd+p^2d^2$.

To compute $M_2$, let
\[\mc{ON} = 
\{ \bs E_{aa}:1\le a\le d\}
\cup
\left\{
\frac{\bs E_{ab}+\bs E_{ba}}{\sqrt2}:1\le a<b\le d
\right\}
\]
be the standard Frobenius-orthonormal basis of $\Sym(d)$. Since $
\Tr(\bs X^2)=\|\bs X\|_{\mathrm F}^2
=$
$\sum_{\bs H \in \mc {ON}}$ $
(\Tr(\bs H\bs X))^2$,
where the sum is over this orthonormal basis,
and thus we have
{\color{black}\[
\begin{aligned}
M_2
&= \EE_{\bs X\sim\rho}[\Tr(\bs X^2)]\\
&= \EE_{\bs X\sim\rho}\left[
\sum_{a=1}^d (\Tr(\bs E_{aa}\bs X))^2
+
\sum_{1\le a<b\le d}
\left(
\Tr\left(\frac{\bs E_{ab}+\bs E_{ba}}{\sqrt2}\bs X\right)
\right)^2
\right]\\
&=
\sum_{a=1}^d
\left(
p^2(\Tr(\bs E_{aa}))^2
+
2p\Tr(\bs E_{aa}^2)
\right)\\
&\quad+
\sum_{1\le a<b\le d}
\left(
p^2\left(\Tr\left(\frac{\bs E_{ab}+\bs E_{ba}}{\sqrt2}\right)\right)^2
+
2p\Tr\left(
\left(\frac{\bs E_{ab}+\bs E_{ba}}{\sqrt2}\right)^2
\right)
\right)\\
&=
d(p^2+2p)
+
\frac{d(d-1)}{2}\,2p
=
p^2d+pd(d+1).
\end{aligned}
\]}
where the third equality follows by applying \eqref{eq:second-moment-tensor} to each element of the Frobenius orthonormal basis above.

For $d\ge2$, Lemma~\ref{lem:spectral-moment-structure} gives $
b_\rho
=
\frac{M_2-\frac1d M_1}{d^2+d-2}$.
Substituting the values of $M_1$ and $M_2$,
\[
b_\rho
=
\frac{
p^2d+pd(d+1)
-
\frac1d(p^2d^2+2pd)
}{
d^2+d-2
}
=
\frac{p(d^2+d-2)}{d^2+d-2}
=
p.
\]
Then $
a_\rho
=
\frac{M_1-2b_\rho d}{d^2}
=
\frac{p^2d^2+2pd-2pd}{d^2}
=
p^2$.
For $d=1$, we can only identify the combination $
a_\rho+2b_\rho
=
\EE_\rho[X^2]
=
p^2+2p
=
p(p+2)$,
as claimed.
\end{proof}
\begin{corollary}[Optimal transport between Wishart distributions]
\label{cor:wishart-ot}
For $d\geq 2$, let $\mu_i=\mc W_d(\bs\Sigma_i,p)$ with
$\bs\Sigma_i\in\mathbb S_\succ^d$, $p > d-1$, $i=0,1$,
and let
$c(\bs X,\bs Y)=\|\bs X-\bs Y\|_{\mathrm F}^2$.
Set $\bs C=(\bs\Sigma_0^\hf\bs\Sigma_1\bs\Sigma_0^\hf)^{\hf}$.
Then:
\begin{enumerate}[label=\textnormal{(\alph*)}]
    \item the unique optimal Monge map from $\mu_0$ to $\mu_1$,
up to $\mu_0$-almost-everywhere equality, is $T\opt$,
where $T\opt$ is as defined in \eqref{eq:congruence-map},
    \item $(\mathrm{id}_{\mathbb S_{\succ}^d}, T\opt)_{\#}\mu_0$ is the unique optimal Kantorovich plan,
    \item $  \mathds K_c(\mu_0,\mu_1)
  =p(p+1)\|\bs\Sigma_0-\bs\Sigma_1\|_{\mathrm F}^2
  +p\left((\Tr(\bs\Sigma_0))^2+(\Tr(\bs\Sigma_1))^2
  -2(\Tr(\bs C))^2\right),$
    \item $\mathds M_c(\mu_0, \mu_1) = \mathds K_c(\mu_0, \mu_1)$. 
\end{enumerate}
\end{corollary}

\begin{proof}
By Lemma~\ref{lem:wishart}, $\rho = \mathcal{W}_d(\bs I_d, p)$ satisfies the hypotheses of \Cref{thm:congruence-template} for $p > d-1$. Substituting the corresponding values of $a_\rho$ and $b_\rho$ from \Cref{lem:wishart} into \eqref{eq:congruence-value} yields the stated value of $\mathds K_c$. The remaining claims follow directly from \Cref{thm:congruence-template}.
\end{proof}

\subsubsection*{Inverse-Wishart distributions}
The inverse-Wishart distribution
$\mc{IW}_d(\bs\Psi,p)$ with scale matrix
$\bs\Psi\in\mathbb S_\succ^d$ and degrees of freedom $p>d-1$ is a
probability law on $\mathbb S_\succ^d$ with density
\begin{equation}
\label{eq:iw-density}
\bs X
\mapsto
\frac{
\det(\bs\Psi)^{p/2}
}{
2^{pd/2}{\color{black}\kappa_d}
\mathsf\Gamma_d(p/2)
}
\det(\bs X)^{-(p+d+1)/2}
\exp\left(
-\frac12\Tr(\bs\Psi\bs X^{-1})
\right),
\end{equation}
with respect to
$\mc L_{\Sym(d)}|_{\mathbb S_\succ^d}$,
{\color{black} where
$\kappa_d=2^{d(d-1)/4}$.}

Its mean exists for $p>d+1$, and the second moments of its matrix
entries exist for $p>d+3$.

\begin{lemma}[Inverse-Wishart orbit, stabilizer, and second moments]
\label{lem:iw}
Let $
\rho=\mc{IW}_d(\bs I_d,p)$ with $p>d+3$,
and set $m = p-d$.
Then:
\begin{enumerate}[label=\textnormal{(\roman*)}]
  \item For every $\bs A\in\GL(d)$, $
  \alpha(\bs A)_\#\rho
  =
  \mc{IW}_d(\bs A\bs A^\top,p)$.
  \item $
  \Stab_{\GL(d)}(\rho)=\mc O(d)$.
  \item For $d\ge2$, the constants in
  Lemma~\ref{lem:spectral-moment-structure} are $
 a_\rho
  =
  ({m-2})/({m(m-1)(m-3)})$ and $
  b_\rho
  =({m(m-1)(m-3)})^{-1}$.
  If $d=1$, only the combination
  $a_\rho+2b_\rho$ is identifiable, and in this case $
 a_\rho+2b_\rho
  =({(p-2)(p-4)})^{-1}$.
\end{enumerate}
\end{lemma}

\begin{proof}
Parts \textnormal{(i)} and \textnormal{(ii)} are proved by the same arguments as in the proof of
Lemma~\ref{lem:wishart}. We therefore omit the proofs of those assertions.

\noindent (iii)
The condition $p>d+3$ is equivalent to $m=p-d>3$, which is the condition
under which the second moments of the inverse-Wishart entries are finite.
Then, we have
$\EE_{\bs X \sim\rho}[\bs X]=({m-1})^{-1}\bs I_d$,
so $
\EE_{\bs X \sim \rho}[X_{ab}]
=
\delta_{ab}/({m-1})$.
Moreover, the entrywise covariance is
\[
\mathrm{Cov}(X_{ab},X_{cd})
=
\frac{
2\delta_{ab}\delta_{cd}
+
(m-1)(\delta_{ac}\delta_{bd}+\delta_{ad}\delta_{bc})
}{
m(m-1)^2(m-3)
};
\]
see, for example, \cite[\S 3.8]{Mardia1979}.
Using $
\EE_{\bs X \sim \rho}[X_{ab}X_{cd}]
=
\EE_{\bs X \sim \rho}[X_{ab}]\EE_{\bs X \sim \rho}[X_{cd}]
+
\mathrm{Cov}(X_{ab},X_{cd})$,
we get
\begin{align*}
\EE_{\bs X\sim\rho}[X_{ab}X_{cd}]
&=
\frac{\delta_{ab}\delta_{cd}}{(m-1)^2}
+
\frac{
2\delta_{ab}\delta_{cd}
+
(m-1)(\delta_{ac}\delta_{bd}+\delta_{ad}\delta_{bc})
}{
m(m-1)^2(m-3)
} \\
&=
\left(
\frac{1}{(m-1)^2}
+
\frac{2}{m(m-1)^2(m-3)}
\right)\delta_{ab}\delta_{cd} \\
&\qquad
+
\frac{1}{m(m-1)(m-3)}
(\delta_{ac}\delta_{bd}+\delta_{ad}\delta_{bc}).
\end{align*}
A routine calculation shows that the coefficient of $\delta_{ab}\delta_{cd}$ simplifies as $
\frac{m-2}{m(m-1)(m-3)}$.
Therefore
\[
\EE[X_{ab}X_{cd}]
=
\frac{m-2}{m(m-1)(m-3)}
\delta_{ab}\delta_{cd}
+
\frac{1}{m(m-1)(m-3)}
(\delta_{ac}\delta_{bd}+\delta_{ad}\delta_{bc}).
\]
Comparing this with the tensor form in
Lemma~\ref{lem:spectral-moment-structure} yields, for $d\ge2$,
\[
a_\rho
=
\frac{m-2}{m(m-1)(m-3)},
\qquad
b_\rho
=
\frac{1}{m(m-1)(m-3)}.
\]

For $d=1$, the coefficients $a_\rho$ and $b_\rho$ are not
separately identifiable. Only the combination $
a_\rho+2b_\rho$
is meaningful. Substituting $d=1$, so that $m=p-1$, gives
\[
a_\rho+2b_\rho
=
\frac{m-2+2}{m(m-1)(m-3)}
=
\frac{1}{(m-1)(m-3)}
=
\frac{1}{(p-2)(p-4)}.
\]
This completes the proof.
\end{proof}
\begin{corollary}[Optimal transport between inverse-Wishart distributions]
\label{cor:iw-ot}
For $d\geq 2$ and $i=0,1$, let $\mu_i=\mc{IW}_d(\bs\Sigma_i, p)$ with
$\bs\Sigma_i\in\mathbb S_\succ^d$,
$p >d+3$,
and let $c(\bs X,\bs Y)=\|\bs X-\bs Y\|_{\mathrm F}^2$.
Set $m = p -d$ and
$\bs C=(\bs\Sigma_0^\hf\bs\Sigma_1\bs\Sigma_0^\hf)^{\hf}$.
Then, 
\begin{enumerate}[label=\textnormal{(\alph*)}]
   \item the unique optimal Monge map from $\mu_0$ to $\mu_1$,
up to $\mu_0$-almost-everywhere equality, is $T\opt$,
where $T\opt$ is as defined in \eqref{eq:congruence-map},
    \item $(\mathrm{id}_{\mathbb S_\succ^d}, T\opt)_\#\mu_0$ is the unique optimal Kantorovich plan,
    \item $\mathds K_c(\mu_0, \mu_1) = ({m(m-3)})^{-1}
   \|\bs\Sigma_0-\bs\Sigma_1\|_{\mathrm F}^2
  +\frac{1}{m(m-1)(m-3)}\bigl(
   (\Tr\bs\Sigma_0)^2+(\Tr\bs\Sigma_1)^2
   -2(\Tr\bs C)^2\bigr),$
   \item $\mathds M_c(\mu_0, \mu_1) = \mathds K_c(\mu_0, \mu_1)$.
\end{enumerate}
\end{corollary}
\begin{proof}
By Lemma~\ref{lem:iw}, $\rho = \mathcal{IW}_d(\bs I_d, p)$ satisfies the hypotheses of \Cref{thm:congruence-template} for $p > d+3$. Substituting the corresponding values of $a_\rho$ and $b_\rho$ from \Cref{lem:iw} into \eqref{eq:congruence-value} yields the stated value of $\mathds K_c$. The remaining claims follow directly from \Cref{thm:congruence-template}.
\end{proof}

\subsubsection*{Matrix beta type~II distributions}
The matrix beta type~II distribution
$\mc{MB}_d^{\mathrm{II}}(q_1,q_2,\bs\Sigma)$ with scale matrix
$\bs\Sigma\in\mathbb S*\succ^d$ and shape parameters
$q_1>(d-1)/2$ and $q_2>(d-1)/2$ is a probability law on
$\mathbb S_\succ^d$ with density
\begin{equation}
\label{eq:mbeta-density}
\bs X
\mapsto
\frac{
\det(\bs\Sigma)^{q_2}
}{
{\color{black}\kappa_d}\mathsf B_d(q_1,q_2)
}
\det(\bs X)^{q_1-(d+1)/2}
\det(\bs\Sigma+\bs X)^{-(q_1+q_2)},
\end{equation}
with respect to
$\mc L_{\Sym(d)}|_{\mathbb S_\succ^d}$,
{\color{black} where
$\kappa_d=2^{d(d-1)/4}$} and \(\mathsf B_d\) denotes the multivariate beta function.
When $q_1>(d-1)/2$ and $2q_2>d+3$, the entry-wise second moments are finite; see \cite[Theorem 5.3.20]{ref:gupta2018matrix}.
\begin{lemma}[Matrix beta~II orbit, stabilizer, and second moments]
\label{lem:mbeta}
Let
$\rho=\mc{MB}_d^{\mathrm{II}}(q_1,q_2,\bs I_d)$
with $q_1>(d-1)/2$ and $2q_2>d+3$,
and set $q=2q_2-d$.
Then:
\begin{enumerate}[label=\textnormal{(\roman*)}]
  \item $\alpha(\bs A)_\#\rho
        =\mc{MB}_d^{\mathrm{II}}(q_1,q_2,\bs A\bs A^\top)$
  for every $\bs A\in\mathrm{GL}(d)$;
  \item $\Stab_{\mathrm{GL}(d)}(\rho)=\mc O(d)$;
  \item for $d \geq 2$, the constants in \Cref{lem:spectral-moment-structure} are 
  \[\displaystyle
        a_\rho
        =\frac{4q_1\bigl(q_1(q-2)+1\bigr)}{q(q-1)(q-3)}~\text{and}~ b_\rho
        =\frac{2q_1(2q_1+q-1)}{q(q-1)(q-3)}.\]
\end{enumerate}
\end{lemma}

\begin{proof}
(i) follows by \cite[Theorem 5.2.2]{ref:gupta2018matrix}.

(ii)
If $\alpha(\bs A)_\#\rho=\rho$, then by part~(i),
 $
\mc{MB}_d^{\mathrm{II}}(q_1, q_2,\bs A\bs A^\top)
=
\mc{MB}_d^{\mathrm{II}}(q_1,q_2,\bs I_d)$.
Since the scale parameter is uniquely determined by the density
in~\eqref{eq:mbeta-density}, it follows that
$\bs A\bs A^\top=\bs I_d$. Conversely, if $\bs Q\in\mc O(d)$, then part~\textnormal{(i)} gives
\[
\alpha(\bs Q)_\#\rho
=
\mc{MB}_d^{\mathrm{II}}
(q_1,q_2,\bs Q\bs Q^\top)
=
\mc{MB}_d^{\mathrm{II}}(q_1,q_2,\bs I_d)
=
\rho.
\]

(iii)
By \cite[Theorem 5.2.5]{ref:gupta2018matrix}, 
if $\bs W_1\sim\mc W_d(\bs I_d,2q_1)$ and
$\bs W_2\sim\mc W_d(\bs I_d,2q_2)$ are independent, then $\bs X=\bs W_2^{-\hf}\bs W_1\bs W_2^{-\hf}$ satisfies $\bs X \sim \mathcal{MB}_d^{\mathrm{II}}(q_1, q_2, \bs I_d)$.

To compute the second moments of \(\bs X\), we condition on \(\bs W_2\).
Given \(\bs W_2\), the matrix \(\bs W_2^{-1/2}\) is deterministic and
\(\bs W_1\) remains distributed as \(\mc W_d(\bs I_d,2q_1)\), because
\(\bs W_1\) and \(\bs W_2\) are independent. Therefore, by the
congruence transformation rule for the Wishart distribution,
\[
  \bs X \mid \bs W_2
  =
  \bs W_2^{-1/2}\bs W_1\bs W_2^{-1/2}
  \sim
  \mc W_d(\bs W_2^{-1},2q_1).
\]
Here the scale matrix \(\bs W_2^{-1}\) is understood conditionally on
\(\bs W_2\).

Now we apply the Wishart second-moment formula to
$\bs X \mid \bs W_2$. If
$\bs Y\sim\mc W_d(\bs \Sigma,p)$, then by \Cref{lem:wishart}(iii)
\[
  \EE_{\bs Y \sim \mc W_d(\bs \Sigma, p)}[Y_{ab}Y_{cd}]
  =
  p^2\Sigma_{ab}\Sigma_{cd}
  +
  p\bigl(\Sigma_{ac}\Sigma_{bd}+\Sigma_{ad}\Sigma_{bc}\bigr).
\]
Taking here $p=2q_1$ and $\bs\Sigma=\bs W_2^{-1}$ yields
\begin{align}\label{eq:mbii-cond-second}
\nonumber
  &\EE_{\bs X \sim \mc{MB}_{d}^{\mathrm{II}}(q_1, q_2, \bs I_d)}[X_{ij}X_{kl}\mid\bs W_2]
  =\\
  &\hspace{0.2cm}(2q_1)^2(W_2^{-1})_{ij}(W_2^{-1})_{kl}
  +
  2q_1\bigl((W_2^{-1})_{ik}(W_2^{-1})_{jl}
  +(W_2^{-1})_{il}(W_2^{-1})_{jk}\bigr).
\end{align}

Next, since $\bs W_2^{-1}\sim \mc{IW}_d(\bs I_d,2q_2)$, we may use the
inverse-Wishart second-moment tensor from Lemma~\ref{lem:iw}. Writing
\[
  \EE_{\bs W_2^{-1} \sim \mc{IW}_d(\bs I_d, 2q_2)}[(W_2^{-1})_{ij}(W_2^{-1})_{k\ell}]
  =
  \alpha_{\mathrm{IW}}\delta_{ij}\delta_{k\ell}
  +
  \beta_{\mathrm{IW}}
  \bigl(\delta_{ik}\delta_{j\ell}+\delta_{i\ell}\delta_{jk}\bigr),
\]
with
\[
  \alpha_{\mathrm{IW}}=\frac{q-2}{q(q-1)(q-3)},
  \qquad
  \beta_{\mathrm{IW}}=\frac{1}{q(q-1)(q-3)},
  \qquad
  q=2q_2-d,
\]
we take expectations in \eqref{eq:mbii-cond-second}.

For the first term,
\[
  \EE_{\bs W_2^{-1} \sim \mc{IW}_d(\bs I_d, 2q_2)}[(W_2^{-1})_{ij}(W_2^{-1})_{kl}]
  =
  \alpha_{\mathrm{IW}}\delta_{ij}\delta_{kl}
  +
  \beta_{\mathrm{IW}}
  \bigl(\delta_{ik}\delta_{jl}+\delta_{il}\delta_{jk}\bigr).
\]

For the second term, apply the same formula twice:
\begin{align*}
  \EE_{\bs W_2^{-1} \sim \mc{IW}_d(\bs I_d, 2q_2)}[(W_2^{-1})_{ik}(W_2^{-1})_{jl}]
  &=
  \alpha_{\mathrm{IW}}\delta_{ik}\delta_{jl}
  +
  \beta_{\mathrm{IW}}
  \bigl(\delta_{ij}\delta_{kl}+\delta_{il}\delta_{jk}\bigr), \\
  \EE_{\bs W_2^{-1} \sim \mc{IW}_d(\bs I_d, 2q_2)}[(W_2^{-1})_{il}(W_2^{-1})_{jk}]
  &=
  \alpha_{\mathrm{IW}}\delta_{il}\delta_{jk}
  +
  \beta_{\mathrm{IW}}
  \bigl(\delta_{ij}\delta_{kl}+\delta_{ik}\delta_{jl}\bigr).
\end{align*}
Summing these two identities gives
\begin{align*}
  &\EE_{\bs W_2^{-1} \sim \mc{IW}_d(\bs I_d, 2 q_2)}[(W_2^{-1})_{ik}(W_2^{-1})_{jl}]
  +
  \EE[(W_2^{-1})_{il}(W_2^{-1})_{jk}]\\
  &\hspace{0.5cm}=
  2\beta_{\mathrm{IW}}\delta_{ij}\delta_{kl}
  +
  (\alpha_{\mathrm{IW}}+\beta_{\mathrm{IW}})
  \bigl(\delta_{ik}\delta_{jl}+\delta_{il}\delta_{jk}\bigr).
\end{align*}

Substituting back into \eqref{eq:mbii-cond-second}, we obtain $
  \EE[X_{ij}X_{kl}]
  =
  a_\rho\delta_{ij}\delta_{kl}
  +
  b_\rho(\delta_{ik}\delta_{jl}+\delta_{il}\delta_{jk})$,
where
\begin{align*}
  a_\rho
  &=
  (2q_1)^2\alpha_{\mathrm{IW}}
  +
  2q_1\cdot 2\beta_{\mathrm{IW}},\quad b_\rho=
  (2q_1)^2\beta_{\mathrm{IW}}
  +
  2q_1(\alpha_{\mathrm{IW}}+\beta_{\mathrm{IW}}).
\end{align*}
Finally, inserting the values of $\alpha_{\mathrm{IW}}$ and
$\beta_{\mathrm{IW}}$ yields
\begin{align*}
  a_\rho
  &=
  \frac{4q_1^2(q-2)+4q_1}{q(q-1)(q-3)}
  =
  \frac{4q_1\bigl(q_1(q-2)+1\bigr)}{q(q-1)(q-3)}, \\
  b_\rho
  &=
  \frac{4q_1^2+2q_1(q-1)}{q(q-1)(q-3)}
  =
  \frac{2q_1(2q_1+q-1)}{q(q-1)(q-3)}.
\end{align*}
\end{proof}

\begin{corollary}[Optimal transport between matrix beta~II distributions]
\label{cor:mbeta-ot}
For $d\geq 2$ and $i=0,1$ let
$\mu_i=\mc{MB}_d^{\mathrm{II}}(q_1,q_2,\bs\Sigma_i)$
with $\bs\Sigma_i\in\mathbb S_\succ^d$,
$q_1>(d-1)/2$, $2q_2>d+3$, 
and let $c(\bs X,\bs Y)=\|\bs X-\bs Y\|_{\mathrm F}^2$.
Set $q=2q_2-d$ and
$\bs C=(\bs\Sigma_0^\hf\bs\Sigma_1\bs\Sigma_0^\hf)^{\hf}$.
Then, 
\begin{enumerate}[label=\textnormal{(\alph*)}]
    \item the unique optimal Monge map from $\mu_0$ to $\mu_1$,
up to $\mu_0$-almost-everywhere equality, is $T\opt$,
where $T\opt$ is as defined in \eqref{eq:congruence-map},
    \item $(\mathrm{id}_{\mathbb S_\succ^d}, T\opt)_\#\mu_0$ is the unique optimal Kantorovich plan,
    \item 
    \begin{align*}\mathds K_c(\mu_0,\mu_1)
  &=\frac{2q_1(2q_1q-2q_1+q+1)}{q(q-1)(q-3)}
   \|\bs\Sigma_0-\bs\Sigma_1\|_{\mathrm F}^2
  +\\&\hspace{0.3cm}\frac{2q_1(2q_1+q-1)}{q(q-1)(q-3)}
  \left((\Tr\bs\Sigma_0)^2+(\Tr\bs\Sigma_1)^2
  -2(\Tr\bs C)^2\right),
  \end{align*}
   \item $\mathds M_c(\mu_0, \mu_1) = \mathds K_c(\mu_0, \mu_1)$.
\end{enumerate}\end{corollary}

\begin{proof}
By Lemma~\ref{lem:mbeta}, $\rho = \mathcal{MB}_d^{\rm II}(q_1, q_2, \bs I_d)$ satisfies the hypotheses of \Cref{thm:congruence-template} for $q_1 > (d-1) / 2$ and $2q_2 > d+3$. Substituting the corresponding values of $a_\rho$ and $b_\rho$ from \Cref{lem:mbeta} into \eqref{eq:congruence-value} yields the stated value of $\mathds K_c$. The remaining claims follow directly from \Cref{thm:congruence-template}.
\end{proof}

\subsection{Multivariate normal sinh--arcsinh distributions}
\label{app:sinh-arcsinh-proofs}
\color{black}
\begin{lemma}
\label{lem:dependent-sinh-arcsinh-orbit}
The following statements hold.
\begin{enumerate}[label=\textnormal{(\roman*)}]
    \item
    The group $G_{\bs\Omega}$ is a $d$-dimensional Lie group, and $\alpha$ defines a smooth left action
    of $G_{\bs\Omega}$ on $\R^d$. Moreover,
    $(G_{\bs\Omega})_\#\rho
    =\{\mu_{\bs\epsilon}:\bs\epsilon\in\R^d\}$.

    \item

    $\Stab_{G_{\bs\Omega}}(\rho)
    =\{(\bs 0_d,\bs Q):
    \bs Q\in\mathfrak B_{\bs\Omega}\}$.

    \item
    Every orbit law $\mu_{\bs\epsilon}$ has density
    \eqref{eq:dependent-sinh-arcsinh-density} and belongs to
    $\mc P_2(\R^d)$.
\end{enumerate}
\end{lemma}
\begin{proof}[Proof of \Cref{lem:dependent-sinh-arcsinh-orbit}]
\noindent\textnormal{(i)}
The set $\mathfrak B_{\bs\Omega}$ is a subgroup of the finite group
$\mathfrak B_d$ and is therefore a zero-dimensional Lie group with the
discrete topology. Consequently,
$G_{\bs\Omega}=\R^d\rtimes\mathfrak B_{\bs\Omega}$ is the finite disjoint
union of the components
$\R^d\times\{\bs Q\}$, with
$\bs Q\in\mathfrak B_{\bs\Omega}$, each of which is diffeomorphic to
$\R^d$. Multiplication and inversion are smooth on every component, so
$G_{\bs\Omega}$ is a Lie group of dimension $d$.

Every signed permutation matrix commutes with a componentwise odd function.
Hence, for every $\bs Q\in\mathfrak B_d$ and $\bs x\in\R^d$,
$\operatorname{arsinh}(\bs Q\bs x)
=\bs Q\operatorname{arsinh}(\bs x)$ and
$\sinh(\bs Q\bs x)=\bs Q\sinh(\bs x)$.
Let $g=(\bs\epsilon,\bs Q)$ and
$\widetilde g=(\bs\eta,\bs P)$ belong to $G_{\bs\Omega}$.
Since $\operatorname{arsinh}(\sinh(\bs z))=\bs z$, we have
$\alpha(g)(\alpha(\widetilde g)(\bs x))
=\sinh(\bs Q\bs P\operatorname{arsinh}(\bs x)
-\bs Q\bs\eta-\bs\epsilon)
=\alpha(g\widetilde g)(\bs x)$, while
$\alpha(e)=\mathrm{id}_{\R^d}$. Thus
$\alpha$ defines a left action. For each
fixed $\bs Q$, the evaluation map is a composition of smooth maps in
$(\bs\epsilon,\bs x)$; since $\mathfrak B_{\bs\Omega}$ is discrete, the
action is smooth on $G_{\bs\Omega}\times\R^d$.

Using
\(\operatorname{arsinh}(\bs Q\bs x)
=
\bs Q\operatorname{arsinh}(\bs x)\) for $\bs Q \in \mathfrak B_\Omega$, we obtain
\[
\begin{aligned}
\alpha((\bs\epsilon,\bs Q))(\bs x)
&=
\sinh\bigl(
\bs Q\operatorname{arsinh}(\bs x)-\bs\epsilon
\bigr)=
\sinh\bigl(
\operatorname{arsinh}(\bs Q\bs x)-\bs\epsilon
\bigr)
=
T_{\bs\epsilon}(\bs Q\bs x).
\end{aligned}
\]

If
$\bs Q\in\mathfrak B_{\bs\Omega}$ and
$\bs Z\sim\mc N(\bs 0_d,\bs\Omega)$, then
$\bs Q\bs Z\sim
\mc N(\bs 0_d,\bs Q\bs\Omega\bs Q^\top)=\rho$.
Therefore,
$\alpha((\bs\epsilon,\bs Q))_\#\rho
=(T_{\bs\epsilon})_\#((\bs Q)_\#\rho)
=(T_{\bs\epsilon})_\#\rho
=\mu_{\bs\epsilon}$.
Conversely,
$\mu_{\bs\epsilon}
=\alpha((\bs\epsilon,\bs I_d))_\#\rho$.
This proves~\textnormal{(i)}.

\noindent\textnormal{(ii)}
If $\bs Q\in\mathfrak B_{\bs\Omega}$, then
$\alpha((\bs 0_d,\bs Q))(\bs x)=\bs Q\bs x$, and the preceding Gaussian
invariance gives
$(\bs 0_d,\bs Q)\in\Stab_{G_{\bs\Omega}}(\rho)$.

Conversely, suppose that
$(\bs\epsilon,\bs Q)\in\Stab_{G_{\bs\Omega}}(\rho)$, let
$\bs Z\sim\rho$, and set
$\bs Y=\alpha((\bs\epsilon,\bs Q))(\bs Z)$.
Then $\bs Y\sim\rho$ and
$\operatorname{arsinh}(\bs Y)
=\bs Q\operatorname{arsinh}(\bs Z)-\bs\epsilon$.
The vector $\operatorname{arsinh}(\bs Z)$ is integrable because
$|\operatorname{arsinh}(z)|\leq|z|$. Each marginal of $\bs Z$ is centered
and symmetric, so oddness of $\operatorname{arsinh}$ gives
$\EE[\operatorname{arsinh}(\bs Z)]=\bs 0_d$. The same identity holds for
$\bs Y$, because $\bs Y$ and $\bs Z$ have the same law. Taking expectations
in the preceding relation gives $\bs\epsilon=\bs 0_d$. Since the second
component of every element of $G_{\bs\Omega}$ belongs to
$\mathfrak B_{\bs\Omega}$, this proves
$\Stab_{G_{\bs\Omega}}(\rho)
=\{(\bs 0_d,\bs Q):
\bs Q\in\mathfrak B_{\bs\Omega}\}$.

\noindent\textnormal{(iii)}
The map $T_{\bs\epsilon}$ is a smooth diffeomorphism with inverse
$\bs u_{\bs\epsilon}=T_{-\bs\epsilon}$.
The Jacobian matrix of $\bs u_{\bs\epsilon}$ at $\bs x$ is diagonal, and
its $j$th diagonal entry is
$\cosh(\epsilon_j+\operatorname{arsinh}(x_j))/\sqrt{1+x_j^2}$.
Applying the change-of-variables formula to the density of
$\mc N(\bs 0_d,\bs\Omega)$ gives
\eqref{eq:dependent-sinh-arcsinh-density}.

The hyperbolic subtraction formula gives
$T_{\bs\epsilon}(\bs x)_j
=x_j\cosh(\epsilon_j)
-\sqrt{1+x_j^2}\sinh(\epsilon_j)$.
Since $\sqrt{1+x_j^2}\leq1+|x_j|$, there exists a finite constant
$C_{\bs\epsilon}$ such that
$\|T_{\bs\epsilon}(\bs x)\|^2
\leq C_{\bs\epsilon}(1+\|\bs x\|^2)$ for every $\bs x\in\R^d$.
Because the Gaussian reference law belongs to $\mc P_2(\R^d)$, it follows
that $\mu_{\bs\epsilon}\in\mc P_2(\R^d)$.
\end{proof}

\begin{lemma}[Cartan verification for the multivariate sinh-arcsinh action]
\label{lem:dependent-sinh-arcsinh-cartan}
Define $\theta:G_{\bs\Omega}\to G_{\bs\Omega}$ by
$\theta(\bs\epsilon,\bs Q)=(-\bs\epsilon,\bs Q)$. Then the following
statements hold.
\begin{enumerate}[label=\textnormal{(\roman*)}]
    \item
    The map $\theta$ is a Lie group involution, and
    $G_{\bs\Omega}^\theta
    =\{(\bs 0_d,\bs Q):
    \bs Q\in\mathfrak B_{\bs\Omega}\}$ and $G_{\bs \Omega}^\theta 
    =\Stab_{G_{\bs\Omega}}(\rho)$.

    \item
    Under the identification
    $\Lie(G_{\bs\Omega})\simeq\R^d$, the eigenspaces of
    $\diff\theta_e$ are
    $\mathfrak k=\{\bs 0_d\}$ and $\mathfrak p=\R^d$.
    Moreover, $\exp(\bs S)=(\bs S,\bs I_d)$ for every
    $\bs S\in\mathfrak p$, and
    $(G_{\bs\Omega},\theta)$ admits a global Cartan decomposition in the
    sense of \Cref{def:cartan-decomposition}.

    \item
    For every $\bs S\in\mathfrak p$ and $\bs x\in\R^d$,
    $D_{\bs x}\alpha(\exp(\bs S))$ is self-adjoint on $\R^d$.

\end{enumerate}
Consequently, together with
\Cref{lem:dependent-sinh-arcsinh-orbit}, all the hypotheses of
\Cref{thm:nonlinear-cartan-exactness} are satisfied.
\end{lemma}
\begin{proof}[Proof of \Cref{lem:dependent-sinh-arcsinh-cartan}]
\noindent\textnormal{(i)}
Let $g=(\bs\epsilon,\bs Q)$ and
$\widetilde g=(\bs\eta,\bs P)$. The group law gives
$\theta(g\widetilde g)
=(-\bs\epsilon-\bs Q\bs\eta,\bs Q\bs P)
=\theta(g)\theta(\widetilde g)$.
Moreover, $\theta^2=\mathrm{id}_{G_{\bs\Omega}}$, and $\theta$ is smooth
on each connected component. Hence $\theta$ is a Lie group involution. Its
fixed-point subgroup is
$G_{\bs\Omega}^\theta
=\{(\bs 0_d,\bs Q):
\bs Q\in\mathfrak B_{\bs\Omega}\}$, which equals
$\Stab_{G_{\bs\Omega}}(\rho)$ by
\Cref{lem:dependent-sinh-arcsinh-orbit}(ii).

\noindent\textnormal{(ii)}
A neighborhood of $e=(\bs 0_d,\bs I_d)$ is contained in the identity
component $\R^d\times\{\bs I_d\}$. Its tangent space at $e$ is
identified with $\R^d$, and under this identification
$\diff\theta_e(\bs S)=-\bs S$. Therefore,
$\mathfrak k=\{\bs 0_d\}$ and $\mathfrak p=\R^d$.
For $\bs S\in\mathfrak p$, the curve
$t\mapsto(t\bs S,\bs I_d)$ is the one-parameter subgroup with initial
velocity $\bs S$, so $\exp(\bs S)=(\bs S,\bs I_d)$.

Let $\Psi:G_{\bs\Omega}^\theta\times\mathfrak p
\to G_{\bs\Omega}$ be the Cartan map. The group law gives
$\Psi((\bs 0_d,\bs Q),\bs S)
=(\bs Q\bs S,\bs Q)$.
This map is bijective, with inverse
$(\bs\epsilon,\bs Q)\mapsto
((\bs 0_d,\bs Q),\bs Q^{-1}\bs\epsilon)$.
Both maps are smooth on every connected component. Hence $\Psi$ is a
diffeomorphism and $(G_{\bs\Omega},\theta)$ admits a global Cartan
decomposition.

\noindent\textnormal{(iii)}
Fix $\bs S\in\mathfrak p$ and $\bs x\in\R^d$. By
\Cref{lem:dependent-sinh-arcsinh-cartan}(ii),
$\exp(\bs S)=(\bs S,\bs I_d)$. Hence, by the definition of $\alpha$,
$\alpha(\exp(\bs S))(\bs x)_j
=
\sinh(\operatorname{arsinh}(x_j)-S_j)$
for $j=1,\ldots,d$.
Using the hyperbolic subtraction identity
$\sinh(a-b)=\sinh(a)\cosh(b)-\cosh(a)\sinh(b)$, together with
$\sinh(\operatorname{arsinh}(x_j))=x_j$ and
$\cosh(\operatorname{arsinh}(x_j))=\sqrt{1+x_j^2}$, we obtain
$\alpha(\exp(\bs S))(\bs x)_j
=
x_j\cosh(S_j)-\sqrt{1+x_j^2}\sinh(S_j)$.
Differentiating coordinatewise shows that
$D_{\bs x}\alpha(\exp(\bs S))$ is diagonal, with $j$th diagonal entry
$\cosh(S_j)-\frac{x_j}{\sqrt{1+x_j^2}}\sinh(S_j)$.
Hence $D_{\bs x}\alpha(\exp(\bs S))$ is self-adjoint.

By \Cref{lem:dependent-sinh-arcsinh-orbit}(i) and (iii), the action $\alpha$ is smooth and the
reference and orbit laws satisfy the required absolute-continuity and
second-moment conditions. Parts~\textnormal{(i)} and~\textnormal{(ii)}
verify the stabilizer condition and global Cartan decomposition, while
part~\textnormal{(iii)} verifies the differential-symmetry condition.
Thus all the hypotheses of
\Cref{thm:nonlinear-cartan-exactness} are satisfied.
\end{proof}

\begin{proof}[Proof of
\Cref{cor:dependent-sinh-arcsinh-transport}]
By
\Cref{lem:dependent-sinh-arcsinh-orbit,lem:dependent-sinh-arcsinh-cartan},
all the hypotheses of
\Cref{thm:nonlinear-cartan-exactness} are satisfied. Moreover,
\Cref{lem:dependent-sinh-arcsinh-orbit}(i) gives
$\mu_i=\mu_{\bs\epsilon_i}$ for $i=0,1$.

Since $\mathfrak B_{\bs\Omega}$ is a subgroup and
$\bs Q_0,\bs Q_1\in\mathfrak B_{\bs\Omega}$, we have
$\bs Q_1^{-1}\bs Q_0\in\mathfrak B_{\bs\Omega}$. Hence,
\Cref{lem:dependent-sinh-arcsinh-cartan}(i) implies that
$h\opt=(\bs 0_d,\bs Q_1^{-1}\bs Q_0)$ belongs to
$G_{\bs\Omega}^{\theta}
=\Stab_{G_{\bs\Omega}}(\rho)$.
Using
$g_0^{-1}
=(-\bs Q_0^{-1}\bs\epsilon_0,\bs Q_0^{-1})$
and the group law, we obtain
$g_1h\opt=(\bs\epsilon_1,\bs Q_0)$ and
$g_1h\opt g_0^{-1}
=(\bs\epsilon_1-\bs\epsilon_0,\bs I_d)
=(\bs\Delta,\bs I_d)$.
By
\Cref{lem:dependent-sinh-arcsinh-cartan}(ii),
$\exp(\bs\Delta)=(\bs\Delta,\bs I_d)$, and therefore
$g_1h\opt g_0^{-1}=\exp(\bs\Delta)$.
Thus $h\opt$ realizes the Cartan absorption used in the proof of
\Cref{thm:nonlinear-cartan-exactness}, and the conclusions of that
theorem hold with $h=h\opt$.

By the definition of $T_h$ in
\Cref{thm:nonlinear-cartan-exactness},
$T_{h\opt}(\bs x)
=
\alpha(g_1h\opt g_0^{-1})(\bs x)
=
\alpha((\bs\Delta,\bs I_d))(\bs x)
=
\sinh(\operatorname{arsinh}(\bs x)-\bs\Delta)$.
Thus $T_{h\opt}=T\opt$.

It remains to evaluate the common optimal value. Let
$\bs Z\sim\rho=\mc N(\bs 0_d,\bs\Omega)$, and recall that
$T_{\bs\epsilon}(\bs x)
=\sinh(\operatorname{arsinh}(\bs x)-\bs\epsilon)$.
Because every $\bs Q\in\mathfrak B_{\bs\Omega}$ is a signed
permutation matrix and $\operatorname{arsinh}$ is odd and applied
componentwise, we have
$\operatorname{arsinh}(\bs Q\bs x)
=\bs Q\operatorname{arsinh}(\bs x)$.
Consequently,
$\alpha((\bs\epsilon,\bs Q))(\bs x)
=T_{\bs\epsilon}(\bs Q\bs x)$.

Since $h\opt$ solves
\eqref{eq:transport-wihtin-orbits}, its optimal value is
$\int_{\R^d}
c(\alpha(g_0)(\bs z),\alpha(g_1h\opt)(\bs z))
\diff\rho(\bs z)$.
Let $\bs Z\sim\rho$. Since
$g_0=(\bs\epsilon_0,\bs Q_0)$ and
$g_1h\opt=(\bs\epsilon_1,\bs Q_0)$, and since
$\alpha((\bs\epsilon,\bs Q))(\bs x)
=
T_{\bs\epsilon}(\bs Q\bs x)$, the preceding optimal value is
$\EE[
\|T_{\bs\epsilon_0}(\bs Q_0\bs Z)
-
T_{\bs\epsilon_1}(\bs Q_0\bs Z)\|^2
]$.
Moreover, $\bs Q_0\in\mathfrak B_{\bs\Omega}$ implies
$\bs Q_0\bs\Omega\bs Q_0^\top=\bs\Omega$, and therefore
$\bs Q_0\bs Z\sim
\mc N(\bs 0_d,\bs Q_0\bs\Omega\bs Q_0^\top)
=
\mc N(\bs 0_d,\bs\Omega)
=
\rho$.
Hence the optimal value of
\eqref{eq:transport-wihtin-orbits} is
$\EE[
\|T_{\bs\epsilon_0}(\bs Z)
-
T_{\bs\epsilon_1}(\bs Z)\|^2
]$.
Finally, by
\Cref{thm:nonlinear-cartan-exactness}, the Kantorovich, Monge, and
Lie group orbit transport values coincide. Consequently,
$\mathds K_c(\mu_0,\mu_1)
=
\EE[
\|T_{\bs\epsilon_1}(\bs Z)
-
T_{\bs\epsilon_0}(\bs Z)\|^2
]$.

For each $j=1,\ldots,d$, the hyperbolic subtraction identity gives
$T_{\bs\epsilon_i}(\bs Z)_j
=
Z_j\cosh(\epsilon_{i,j})
-
\sqrt{1+Z_j^2}\sinh(\epsilon_{i,j})$.
Therefore,
$T_{\bs\epsilon_1}(\bs Z)_j
-
T_{\bs\epsilon_0}(\bs Z)_j
=
\Delta C_j Z_j
-
\Delta S_j\sqrt{1+Z_j^2}$.
Since $\bs\Omega$ is a correlation matrix,
$Z_j\sim\mc N(0,1)$ for every $j$. Thus
$\EE[Z_j^2]=1$,
$\EE[1+Z_j^2]=2$, and
$\EE[Z_j\sqrt{1+Z_j^2}]=0$, where the last identity follows
because the integrand is odd and the standard normal distribution
is symmetric. It follows that
$\EE[
(T_{\bs\epsilon_1}(\bs Z)_j
-
T_{\bs\epsilon_0}(\bs Z)_j)^2
]
=
(\Delta C_j)^2+2(\Delta S_j)^2$.
Summing the preceding identity
over $j=1,\ldots,d$ yields
$\mathds K_c(\mu_0,\mu_1)
=
\sum_{j=1}^d
\bigl((\Delta C_j)^2+2(\Delta S_j)^2\bigr)$,
which proves the stated formula.
\end{proof}